\documentclass{elsarticle} \journal{arXiv}
\usepackage[a4paper, margin=4cm]{geometry}
\usepackage{amsmath,amssymb,amsthm,color,enumerate}
\usepackage[colorlinks,citecolor=magenta,linkcolor=blue]{hyperref}

\theoremstyle{plain}
\newtheorem{theorem}{Theorem}[section]

\theoremstyle{definition}
\newtheorem{definition}[theorem]{Definition}
\newtheorem{proposition}[theorem]{Proposition}
\newtheorem{remark}[theorem]{Remark}
\newtheorem{corollary}[theorem]{Corollary}
\newtheorem{lemma}[theorem]{Lemma}
\newtheorem{example}[theorem]{Example}
\renewenvironment{proof}{{\noindent \bf  Proof.}}{\qed}

\newcommand{\Ml}{D^\psi_{+,\infty}}
\newcommand{\Mr}{D^\psi_{-,\infty}}
\newcommand{\Mpm}{D^\psi_{\pm,\infty}}
\newcommand{\Cl}{\partial^\psi_{+}}
\newcommand{\Cr}{\partial^\psi_{-}}
\newcommand{\Cpm}{\partial^\psi_{\pm}}
\newcommand{\Conel}{\partial^{\psi-1}_{+}}
\newcommand{\Coner}{\partial^{\psi-1}_{-}}
\newcommand{\Conepm}{\partial^{\psi-1}_\pm}
\newcommand{\RLl}{D^\psi_{+}}

\newcommand{\RLpm}{D^\psi_\pm}
\newcommand{\RLonel}{D^{\psi-1}_{+}}

\newcommand{\RLonepm}{D^{\psi-1}_\pm}
\newcommand{\Clh}{\partial^\psi_{+h}}
\newcommand{\Crh}{\partial^\psi_{-h}}
\newcommand{\Cpmh}{\partial^\psi_{\pm h}}

\newcommand{\Conerh}{\partial^{\psi-1}_{- h}}
\newcommand{\Conepmhzero}{\partial^{\psi-1}_{\pm h 0}}
\newcommand{\Conelhzero}{\partial^{\psi-1}_{+ h 0}}
\newcommand{\Conerhzero}{\partial^{\psi-1}_{- h 0}}

\newcommand{\Cpmhzero}{\partial^{\sym}_{\pm h 0}}
\newcommand{\Crhzero}{\partial^{\sym}_{- h 0}}
\newcommand{\Clhzero}{\partial^{\sym}_{+ h 0}}

\newcommand{\sym}{\psi}
\newcommand{\Gru}{\calG^{\sym}}
\newcommand{\Gruone}{\calG^{\sym -1}}

\newcommand{\LTp}{\beta}
\newcommand{\SuS}{S}
\newcommand{\barlambda}{\Bar\lambda}
\newcommand{\mh}{-h}
\newcommand{\ph}{+h}

\newcommand{\Il}{I^\psi_+}
\newcommand{\Ir}{I^\psi_-}
\newcommand{\Ipm}{I^\psi_\pm}
\newcommand{\I}{I^{\psi}}

\newcommand{\Elq}{E^{\psi,\LTp}_+}
\newcommand{\Erq}{E^{\psi,\LTp}_-}
\newcommand{\Epmq}{E^{\psi,\LTp}_\pm}

\newcommand{\Gen}{G}
\newcommand{\AF}{A}

\newcommand{\Yh}[1]{ Y^{{\rm #1},h}}
\newcommand{\Y}[1]{ Y^{{\rm #1}}}

\newcommand{\ind}[1]{\mathbf 1_{\{ #1 >0\}}}

\newcommand{\indi}[1]{\mathbf 1_{\{ #1 \}}}

 \newcommand{\supnorm}[2]{ \|#1\|_{#2,\infty} }

\newcommand{\dd}{ { \mathrm{d}} }

\newcommand{\R}{\mathbb {R}}
\newcommand{\D}{\mathcal {D}}
\newcommand{\BC}{\mathrm{LR}}

\newcommand{\calG}{\mathcal G}
 
\makeatletter
\let\orgdescriptionlabel\descriptionlabel
\renewcommand*{\descriptionlabel}[1]{%
  \let\orglabel\label
  \let\label\@gobble
  \phantomsection 
  \edef\@currentlabel{#1}%
  \let\label\orglabel
  \orgdescriptionlabel{#1}%
}
\makeatother
\allowdisplaybreaks

\begin{document}

\begin{frontmatter}
\title{Boundary conditions for nonlocal one-sided pseudo-differential operators and the associated stochastic processes I\tnoteref{dedication}} 
\tnotetext[dedication]{Dedicated to Mark Meerschaert, you were a great friend and inspiration.}
\fntext[marsden]{Baeumer and Kov\'acs were partially funded by the Marsden Fund administered by the Royal Society of New Zealand.} 
\author{Boris Baeumer\fnref{marsden}} \address{University of Otago, New Zealand}\ead{bbaeumer@maths.otago.ac.nz} 
\author{Mih\'aly Kov\'acs\fnref{marsden}} \address{P\'azm\'any P\'eter Catholic University } \ead{kovacs.mihaly@itk.ppke.hu} \author{Lorenzo Toniazzi\fnref{hari}} \address{University of Otago, New Zealand}\fntext[hari]{Toniazzi was fully funded by the Marsden Fund administered by the Royal Society of New Zealand.} \ead{ltoniazzi@maths.otago.ac.nz} 
\begin{abstract} We connect boundary conditions for one-sided pseudo-differential operators with the generators of modified one-sided L\'evy processes. On one hand this allows modellers to use appropriate boundary conditions with confidence when restricting the modelling domain.  On the other hand it allows for numerical techniques based on differential equation solvers to obtain fast approximations of densities or other statistical properties of restricted one-sided L\'evy processes encountered, for example, in finance. In particular we identify a new nonlocal mass conserving boundary condition by showing it corresponds to fast-forwarding, i.e.  removing the time the process spends outside the domain. We treat all combinations of killing, reflecting and fast-forwarding boundary conditions.  

In Part I we show wellposedness of the backward and forward Cauchy problems with a one-sided pseudo-differential operator with boundary conditions as generator. We do so by showing convergence  of Feller semigroups based on grid point approximations of the modified L\'evy process.

In Part II we show that the limiting Feller semigroup is indeed the semigroup associated with the modified L\'evy process by showing continuity of the modifications with respect to the Skorokhod topology. \end{abstract}

\end{frontmatter}

\tableofcontents

\section{Introduction} 
The last three decades have seen a surge in the application and theoretical development of fractional \cite{MR2218073,MR1658022, MR2090004,MR2884383,Benson2000b,Schumer2009} and nonlocal \cite{T20,MR2343205,MR3156646,MR2780345,D19} integro-differential  equations. An important reason is the discovery of clear probabilistic explanations for solutions and boundary conditions involving  processes with jumps (hence the nonlocality) \cite{MR2884383}, which is the foundation of particle tracking/Monte Carlo numerical methods, e.g. \cite{Zhang2006a}. In contrast to the 
 complete understanding of boundary conditions for one-dimensional  diffusion  (continuous) processes \cite{MR0345224},  one-dimensional L\'evy jump (discontinuous) processes are yet to be fully classified with respect to their   pseudo-differential operators/boundary conditions for the associated backward and forward Kolmogorov equations. Indeed this is an active field of research \cite{MR3467345, T20, MR2006232, MR3217703,MR3413862,MR3323906}. Identifying boundary conditions is especially complex when one wants to impose a mass conserving boundary condition modeling jumps across the boundary of the domain, as a variety of natural  modifications of the trajectories appear. Four examples are: censoring and its maximal extension \cite{MR2006232}, stochastic reflection  \cite{MR3467345} and fast-forwarding \cite{MR3582209}.  
A particularly well-understood case is the one of the spectrally positive $\alpha$-stable  L\'evy processes $Y$, for $\alpha\in(1,2)$. In fact, the work \cite{MR3720847}, which originates  in \cite{MR3413862}, gives a detailed description of the backward and forward equations identified by restricting trajectories of $Y$ to $[-1,1]$ either by killing, stochastically reflecting or fast-forwarding. In particular, fast-forwarding $Y$ by removing the time $Y$ spends outside $[-1,1]$ results in the nonlocal (or ``reinsertion in the interior'') boundary condition
\begin{equation}\label{eq:ff_intro}
\partial^{\alpha-1}_- f(-1 )=  \int_{0}^{2}f'(y-1) \frac{y^{1-\alpha}}{\Gamma(2-\alpha)} \,\dd y=0
\end{equation}
for the backward equation, where  $\partial^{\alpha-1}_- $ is a (right) Caputo derivative of order $\alpha-1\in(0,1)$ on $(-\infty,1]$. The methods of  \cite{MR3720847} are based on a finite difference/Gr\"unwald approximation of the Feller generator of $Y$, which provides a rigorous and yet intuitive explanation of \eqref{eq:ff_intro}. Briefly, the nonlocal boundary condition \eqref{eq:ff_intro} describes the following conservation of mass: when $Y$    leaves $(-1,\infty)$ by a drift, its mass is redistributed  \textit{inside the domain} (hence the nonlocality) according to the location of $Y$ at its first jump back inside $(-1,\infty)$. This is entirely different from stochastic reflection, where exiting particles are forced to stay on $\{-1\}$, which results in the standard Neumann boundary condition $f'(-1)=0$.  \\

The main purpose of this work is to extend the methods  and results for stable processes in \cite{MR3720847} to  recurrent one-sided L\'evy processes without the aid of scaling properties.  This means characterising the backward and forward Cauchy problems of the restrictions to an interval of $Y$   via a discrete finite difference approximation.  This turns out to be a rather daunting and lengthy task. The length is partially due to the in-depth treatment of twelve different integro-differential equations and their jump processes (Table \ref{explicitProcesses}) along with the delicate study  of their respective finite difference approximations. But it is also due to the necessity of combining new results with many known results from different fields, such as fluctuation theory of L\'evy processes \cite{MR2250061}, Post–Widder asymptotics \cite{MR0005923}, nonlocal convolution quadrature \cite{MR923707} and Skorokhod maps \cite{MR1876437}. 
To help the reader we provide a substantial introduction distributed in the three subsections below. But first we introduce our one-sided process on $\mathbb R$ (before we restrict it to an interval).\\



 Spectrally positive L\'evy processes posses a rich and well developed theory   \cite{MR1406564,MR2250061,MR3014147} along with several applications, for example in finance, hydrology, and queues \cite{MR2023021,CG00,CW03,MR1919609,MR2343205,MR0391297,MR1492990,MR2577834,MR3342453}. Importantly, their fluctuation theory features several explicit and semi-explicit formulae not available for most L\'evy processes. Such formulae are   expressed in terms of the scale function $k_0$,   defined by its  Laplace transform $1/\psi$, for $\psi$ being the Laplace exponent of the process   \cite[Chapter VII]{MR1406564}. We follow this tradition obtaining  a full description in terms of scale functions of backward and forward generators of our Feller processes on an interval. \\
 We denote by $Y$ any recurrent  spectrally positive L\'evy process with paths of unbounded variation (and no diffusion component). Standard results show that the process $Y$ is characterised by the   L\'evy symbol/Laplace exponent/pseudo-differential operator  
\begin{equation}\label{eq:sym_intro}
\sym(\xi)=\int_{(0,\infty)} \left(e^{-\xi y}-1+\xi y\right)\phi({\rm d}y),\quad \Re\xi\ge 0,
\end{equation} 
recalling that $\exp\{\sym(\xi)\}=\mathbb E[\exp\{ -\xi Y_1\}]$,
 for a L\'evy measure $\phi$ supported on $[0,\infty)$ such that $\int_0^1y\,\phi(\dd y)=\infty $ and  $\int_1^\infty y\,\phi(\dd y)<\infty $. Note that requiring $Y$ to be recurrent is a natural assumption to ensure that the fast-forwarding boundary conditions always conserve mass. This assumption also implies that defining $\Phi(x)=\int_x^\infty\phi(y,\infty)\,{\rm d}y\mathbf1_{\{x>0\}}$ we can write the generator of $Y$ as a \emph{L\'evy derivative} 
 \begin{align}\label{eq:gen_introR}
 \Mr     f(x) 
 &
 =\frac{{\rm d}^2}{{\rm d} x^2} \Phi(-\cdot)\star f(x) =\frac{{\rm d}^2}{{\rm d} x^2}\int_{0}^\infty f(x+y)\Phi(y)\,{\rm d}y,
 \end{align}
 by \cite[Theorem 31.1]{MR3185174} and $\Mr     f(x)=\int_{(0,\infty)}(f(x+y)-f(x)-yf'(x)) \,\phi(\dd y)$ for $f\in C^2_c(\mathbb R)$, 
where $\star$   denotes the convolution operator. 
Note that if $\phi({\rm d}y)=y^{-1-\alpha}\dd y/\Gamma(-\alpha)$ for $\alpha\in(1,2)$, then \eqref{eq:gen_introR} becomes a (right)  fractional derivative of order $\alpha$ on $\mathbb R$   and $Y$ becomes the  spectrally  positive $\alpha$-stable process with L\'evy symbol $\psi(ik)=(ik)^\alpha$ \cite{MR3720847}.\\

We discuss our results and contributions in the next three subsections, covering respectively:  wellposedness of backward (Feller) and forward (Fokker-Planck) equations in terms of scale functions; identification of the corresponding pathwise modification of $Y$; and construction of a finite difference/Gr\"unwald type approximation, its embedding into a Feller process and its convergence.  


\subsection{Boundary conditions for one-sided L\'evy derivatives}\label{sec:intro1}
In analogy with the stable case, representation     \eqref{eq:gen_introR} leads to the  left (`$+$') and  right (`$-$')  L\'evy derivatives of Riemann–Liouville and mixed Caputo type 
 \begin{align}\label{eq:gen_intro}
 \RLl   f(x)   =\frac{{\rm d}^2}{{\rm d} x^2} \Phi\star f\quad \text{and}\quad   \Cpm   f  =\frac{{\rm d}}{{\rm d} x} \Phi(\pm\cdot)\star   f', 
 \end{align}
 respectively, and the right and left   L\'evy derivatives of order ``$\psi-1$''
  \begin{align}\label{eq:phi-1_intro}
  \RLonel   f(x)   =\frac{{\rm d}}{{\rm d} x} \Phi\star f\quad\text{and}\quad\Conepm   f(x)   = \Phi(\pm\cdot)\star     f',
 \end{align}
 all  defined on the 0 extension to $\mathbb R\backslash [-1,1]$ of an appropriate $f$ (see Definition  \ref{def:ndo}). 
 It is not hard to show that these operators admit the explicit right inverse $\Ipm  f= k_0(\cdot\mp 1)\star f$.
Note also that $ \RLl  k_0 =  \RLl  k_0' = 0$ and $ \Cl  k_0 =  \Cl  1 = 0$, providing two degrees of freedom to impose two boundary conditions at $-1$ and $1$, respectively. This procedure leads us to construct  the candidate backward (`$-$') and forward (`$+$') generators. We study all the combinations of the following boundary conditions. 
\begin{itemize}
\item {\bf Killing} (${\rm D}$): The function is zero at the boundary, which   takes the form of  $\lim_{x\downarrow -1}f(x)=0$ or   $\lim_{x\uparrow 1}f(x)=0$.
\item {\bf Fast-forwarding} (${\rm N}$): The ``$\psi-1$'' mixed Caputo type derivative is zero at the boundary, meaning that we impose 
\[
 \lim_{x\downarrow -1} \Conepm  f(x)=0 \quad \text{or}\quad  \lim_{x\uparrow 1} \Conepm  f(x)=0.
\]
\item {\bf Reflecting} (${\rm N^*}$): For the backward equation, the first derivative is zero at the boundary, meaning that we impose  $\lim_{x\downarrow -1}f'(x)=0$ or   $\lim_{x\uparrow 1}f'(x)=0$. For the forward equation the ``$\psi-1$'' Riemann–Liouville type derivative is zero at the boundary, i.e. \[\lim_{x\downarrow -1} \RLonel  f(x)=0 \quad \text{or}\quad \lim_{x\uparrow 1} \RLonel  f(x)=0.\] 
\end{itemize}
For a given  function space $X$, a L\'evy  derivative $\Gen\in\{\RLl,\Cpm\}$  and $\mathrm{L},\mathrm{R}\in\{\mathrm{D},\mathrm{N},\mathrm{N^*}\}$ we refer to $\D(\Gen,\mathrm{LR})$ as the set of  functions in $X$  satisfying $\Gen\D(\Gen,\mathrm{LR})\subset X$, the ${\rm L}$ boundary condition at $-1$ and the ${\rm R}$ boundary condition at $1$. For example, for $X= C_0[-1,1)$ \[\D(\Cr,\mathrm{N^*D})=\Big\{f\in  X: \Cr f\in X,\,\lim_{x\downarrow -1}f'(x)=0,\,\lim_{x\uparrow 1}f(x)=0\Big\},\]
 or for $X=L^1[-1,1]$
\[\D(\Cl,\mathrm{NN})=\Big\{f\in X: \Cl f\in X,\,\lim_{x\downarrow -1}\Conel f(x)=0,\,\lim_{x\uparrow 1}\Conel f(x)=0\Big\}.\]
Then, for a given $X$ we denote by $(\Gen,\BC)$ the operator $\Gen:\D(\Gen,\BC)\to X$. Note that we use the convention $(\RLl,\mathrm{N^*N}):=(\RLl,\mathrm{N^*N^*})$. If we want to comprehensively talk about several operators with an unspecified left or right boundary condition, we just state $\mathrm{L}$ or $\mathrm{R}$, respectively; e.g., $\mathrm{LD}$ would refer to having a left boundary condition of type $\mathrm{D},\mathrm{N},$ or $\mathrm{N^*}$ and a right boundary condition of type $\mathrm{D}$.  We use the notation $C_0(\Omega)$  to be the closure with respect to the supremum norm of the space of continuous functions with compact support in $\Omega$. The set $\Omega$ refers to the interval \[\Omega\in\big\{(-1,1),[-1,1),(-1,1],[-1,1]\big\}\] with an endpoint excluded if the problem has a Dirichlet boundary condition there (if an endpoint is excluded, any function in $C_0(\Omega)$ will converge to zero there). \\

To treat all boundary modifications of $Y$ it is enough to consider the twelve cases shown in Table  \ref{explicitProcesses} (as explained in the next section). Our main  analytical results concerning the  operators $(\Gen,\BC)$ are given in Corollaries \ref{cor:C0semi} and \ref{cor:adjointsemi}. Here we prove that each of the twelve operators in Table \ref{explicitProcesses} generates a positive strongly continuous contraction semigroup $\left(P^\pm_t\right)_{t\ge 0}$ on the respective $X\in \{C_0(\Omega),L^1[-1,1]\}$, and that $(\Gen^-,\BC)$ and $(\Gen^+,\BC)$ are dual to each other. 
These results are achieved by first showing that   the operators $(\Gen,\BC)$ are closed, densely defined,  the range of $1-\Gen$ is $X$ and their duality relation. We then prove  dissipativity of $(\Gen,\BC)$ and positivity of the generated semigroups via an $L^1$ finite difference approximation discussed in Section \ref{sec:intro3}. 

We highlight that the boundary conditions corresponding to ${\rm NR}$ are new, and highly non-trivial. (Their probabilistic intuition and derivation is explained detail in Section \ref{sec:intro3}.) To display this boundary condition in a more familiar way, one may rewrite  it (for the backward equations) as
\begin{align}\label{eq:intro_boundaryN}
 \int_{(-1,1]} (f(y)-f(-1))\nu(\dd y)=\Coner f(-1)=0,
\end{align}
where $\nu(dy)= \phi(y+1,\infty)\,\dd y+\Phi(2)\delta_1(\dd y)$. Then we see that \eqref{eq:intro_boundaryN} takes the form of the ``reinsertion in the interior'' boundary condition for diffusion found by Feller in \cite{MR47886} (see, e.g., \cite[pages 187, 188]{MR0345224}). 
Another analytical contribution of this work is the first full description of backward and forward generators for the more common boundary conditions ${\rm D}$ and ${\rm N^*}$. Interestingly, we discover that reflecting $Y$ at the lower boundary results in a process without a dual; its backward equation is not the forward equation of another stochastic process as its solution is in general gaining mass (the generator is dissipative with respect to the supremum norm but not with respect to the $L^1$ norm).  

\subsection{Identification of the stochastic processes}\label{sec:intro_2}
The second part of this work identifies the six Feller processes $\Y{LR}$ associated to the Feller (backward) equations on $C_0(\Omega) $ and Fokker-Planck (forward) equations on $L^1[-1,1]$. That is, \[ P_t^-f(x)=\mathbb E[f(\Y{LR}_t)|\Y{LR}_0=x],\]
for $f\in C_0(\Omega)$ and  
 \[ P_t^+f(x)\dd x=\mathbb P[\Y{LR}_t\in\dd x|\Y{LR}_0\sim f],\]
 for an $L^1$ probability density $f$.
This result is summarised in Table \ref{explicitProcesses}, with the notation appearing in the first column explained below.

For a stochastic process $(X_t)_{t\ge 0}$ we denote
\[ (X_t)^{\mathrm{kill}}=\begin{cases} X_t, & t<\inf\{s:X_s\not\in \Omega \},\\ \delta,& \mathrm{else,}\end{cases}\]
where $\delta$ is a graveyard point, and \[X_t^{\min}=\inf_{0\le s\le t}\{(X_s+1)\wedge 0\}.\]  
 The fast-forwarded processes $X_{\tau}$ are defined by the time changes $\tau\in \{\tau^{l}, \tau^r, \tau^{l,r}\}$ given by the right inverse
 \[ \tau_t=\inf\{s>0: \AF(s)>t\}\]
 of the respective additive functionals $\AF \in\{\AF^{l}, \AF^r, \AF^{l,r}\}$, where $  \AF^{l}(s)=\int_0^s  \mathbf1_{\{X_z> -1\}} \,\dd z$, $\AF^{r}(s)=\int_0^s  \mathbf1_{\{X_z<1\}} \,\dd z$ and $\AF^{l,r}(s)=\int_0^s \mathbf 1_{\{-1<X_z<1\}} \,\dd z$.
We then have the following interpretation of the six processes in Table \ref{explicitProcesses}:
\begin{enumerate}
\item $\Y{DD}$: $Y$  is killed as soon as it leaves $(-1,1)$. 
\item $\Y{DN}$: $Y$ is killed if it drifts across the left boundary. If it jumps across the right boundary we make a time change deleting the time for which $Y$ is to the right of the right boundary (in other words, fast-forward to the time the process is again inside $(-\infty,1)$). By  \cite[Lemma 2]{MR1175272}, $\Y{DN}$ equals in law  $\Y{DN^*}$, reflecting $Y$ at the right boundary and then killing it at the left boundary.
\item $\Y{ND}$: we make a time change deleting the time for which $Y$ is to the left of the left boundary.  This process is then killed if it jumps across the right boundary.  (We remark again that for a drift across the left boundary fast-forwarding $Y$  restarts the process inside the domain \cite{MR3582209}, which differs  from   reflecting the process at the left boundary.)
\item $\Y{NN}$: we make a time change  deleting the time for which $Y$ is outside the domain $(-1,1)$. This process equals in law $\Y{NN^*}$, i.e. reflecting $Y$ at the right boundary and then fast-forwarding the paths at the left boundary.
\item $\Y{N^*D}$: $Y$ is reflected at the left boundary and then killed if it jumps across the right boundary.
\item $\Y{N^*N}$: $Y$ is reflected at the left boundary and then fast-forwarded at the right boundary. This process equals in law the two sided reflection of $Y$, defined as in \cite{MR3467345}. (This is proved in Part II.)
\end{enumerate}
From the descriptions above we see that we treat all possible combinations of $\BC$ for a recurrent $Y$. \\

We obtain as a corollary of Table \ref{explicitProcesses} several new and known results concerning exit problems and   resolvent measures for one-sided L\'evy processes (which we present in Part II). 
The new results   are the representation of the resolvent measures for  $\Y{NN}$ and $\Y{ND}$, and the solution for the  exit problem for  $\Y{ND}$. Note that this exit problem describes a natural quantity, namely the time spent by a spectrally positive L\'evy process  in an interval $[a,b]$ before its first jump above $b$. As for the the known results, we provide a new proof of the known representations of the resolvent measures for $\Y{DN}$, $\Y{N^*D}$ and  $\Y{N^*N}$, first proved in \cite{MR2054585} and \cite{MR1995924}. Unfortunately, in the  cases $\Y{N^*D}$ and  $\Y{N^*N}$ our work provides a new proof  only by introducing the regularity assumption \ref{H1} on the L\'evy symbol $\psi$.  This is because of the singularities arising in the finite difference approximation,  discussed in Section \ref{sec:intro3}.  It is an open problem whether we can drop \ref{H1} in our proof. We also recall that we treat only recurrent $Y$ with paths of unbounded variation (and no diffusion component), although our strategy appears be extendable to all one-sided L\'evy processes with  twice continuously differentiable $k_0$.

\begin{table}
\centering
\vline
\begin{tabular}{l|c|c}
  \hline
\hspace{1cm} Process $\Y{LR}$ &
\hspace{.1cm}	Forward generator 
\hspace{.1cm} & \hspace{.1cm} Backward generator
\hspace{.1cm}\\
	\hline
	1.  $\Y{DD}_t=Y^{\mathrm{kill}}_t$& $(\Cl , \mathrm{DD})$  & $(\Cr , \mathrm{DD})$   \\
  \hline
2.	$\Y{DN}_t=\left(Y_{\tau^r_t}\right)^{\mathrm{kill}}$ & $(\Cl , \mathrm{DN})$ & $(\Cr , \mathrm{DN})$\\
  \hline
3.	 $\Y{ND}_t=\big(Y_{\tau^l_t}\big)^{\mathrm{kill}}$& $(\Cl , \mathrm{ND})$ &$(\Cr , \mathrm{ND})$\\
  \hline
4.	$\Y{NN}_t=Y_{\tau^{l,r}_t}$ & $(\Cl , \mathrm{NN})$ & $(\Cr , \mathrm{NN})$\\
  \hline
5.	$\Y{N^*D}_t=(Y_t-Y^{\min}_t)^{\mathrm{kill}}$ &$(\RLl , \mathrm{N^*D})$ & $(\Cr , \mathrm{N^*D})$\\
  \hline
6.	 $\Y{N^*N}_t=(Y_t-Y^{\min}_t)_{\tau^r_t}$ & $(\RLl , \mathrm{N^*N})$ & $(\Cr , \mathrm{N^*N})$ \\
  \hline 
\end{tabular}\vline
\caption{\label{explicitProcesses} Restrictions of the spectrally positive process $Y$ to $[-1,1]$ and the associated forward and backward generators of  strongly continuous contraction semigroups  on    $X=L^1[-1,1]$ and $X=C_0(\Omega)$, respectively. The stochastic processes are defined in Section \ref{sec:intro_2}. The generators $(\Gen,\BC)$ are defined in Definition \ref{def:ndo} and the explicit representation for the domains  $\D(\Gen,\BC)$ can be found in Table \ref{tab:ndo}.}
\end{table}

\subsection{Gr\"unwald type approximations}\label{sec:intro3}
We now discuss  our construction of a finite difference approximation for both the Feller  processes and the  backward/forward generators of Table \ref{explicitProcesses}, allowing us to prove the claim made in the same table. As a motivating example observe that the finite difference approximation  $(f(x-h)-2f(x)+f(x+h))/h^2$ to the second derivative (and hence to the Brownian motion) is given by the pseudo-differential operator 
\[
e^{\xi h}\left(\frac{1-e^{-\xi h}}{h}\right)^2=\frac{1}{h^2}e^{\xi h} -\frac{2}{h^2} +\frac{1}{h^2} e^{-\xi h},
\]
where we recall that $\xi^2 $ is the L\'evy symbol of the Brownian motion. Analogously, the symbol  
\[
e^{\xi h}\sym\left(\frac{1-e^{-\xi h}}{h}\right)=e^{\xi h}\sum_{j=0}^\infty \Gru_{j,h}e^{-j\xi h}
\]
is the pseudo-differential operator of the finite difference scheme
\begin{equation}\label{eq:intro_Dh}
 \Crh f(x) =\sum_{j=0}^\infty \Gru_{j,h} f(x+(j-1)h)    \to  D^\psi_{-\infty} f(x)
\end{equation}
as $h\to 0$ \cite{MR923707,MR1269502}. In the stable case $\sym(\xi)=\xi^\alpha$ for $\alpha\in(1,2)$, these coefficients reduce to  $\Gru_{j,h}=h^{-\alpha}(-1)^j\binom\alpha j$, the well-known Gr\"unwald coefficients \cite{MR2884383}. It is not difficult  to prove that $\sum_{j=0}^\infty\Gru_{j,h}=0$ and $-\Gru_{1,h},\Gru_{j,h}>0$ for $j\neq1$, so that \eqref{eq:intro_Dh} generates a discrete compound Poisson process $Y^h$. We then obtain that $Y^h\Rightarrow Y$ as $h\to 0$ in $D([0,\infty),\mathbb R)$, the    c\`ad\`ag functions equipped with  the  Skorokhod $J_1$-topology, where `$\Rightarrow $' denotes  convergence in distribution. In Part II we show that the fast-forward maps (and therefore all the maps defined in Section \ref{sec:intro_2}) are continuous at $Y$. Thus, for any $Y^h_0=Y_0\in (-1,1)$, by Continuous Mapping Theorem, 
\begin{equation}\label{eq:intro_weakconv}
\Yh{LR} \Rightarrow \Y{LR}\,\,  \text{in}\,\, D([0,\infty),\mathbb R),
\end{equation}
where $\Yh{LR}$ is defined by applying the same path modification that maps $Y$ to $\Y{LR}$. Due to the discrete nature of $Y^h$,  we can compute the pointwise generator of $\Yh{LR}$.  In particular we prove that the conservation of mass due to fast-forwarding $Y^h$ at $-1$ yields the following backward generator (applied to $f$) at $-1$, \begin{equation}\label{eq:intro_boundaN}
-\frac1h\Conerh f(-1)=-\frac1h \sum_{j=0}^\infty \Gruone_{j,h} f(-1+jh),
\end{equation}
with the coefficients $\Gruone_{j,h}=h \sum_{i=0}^j  \Gru_{i,h}$.  We then show  (Lemma \ref{lem:firtorderapprox}) that  $\frac1h\Conerh f(-1)$ converges if and only if for $\Cr$ the boundary condition ${\rm NR}$ holds, which takes the form \eqref{eq:intro_boundaryN} on the bounded interval $[-1,1]$.\\

Having obtained a discrete approximation  of the Feller process $\Y{LR}$, one can try to use it to also approximate the Feller generators $(\Gen^-,\BC)$, in order to prove the results in Table \ref{explicitProcesses}. In view of \eqref{eq:intro_weakconv} and the Trotter–Kato Theorem \ref{thm:TK}, it would be enough to prove that the generator of $\Yh{LR}$ converges uniformly to $\Cr f$ for some sequence $f_h\to f\in \D(\Cr,\BC)$. But in general  killing and fast-forwarding do not preserve the Feller property (see Example \ref{exp:killB}), and in particular    $\Yh{DR}$, $\Yh{LD}$, $\Yh{NR}$ and $\Yh{LN}$    are not Feller process. We overcome this issue by extending the theory of interpolated matrices of \cite{MR3720847}. Namely, we first embed $\Yh{LR}$ into a Feller  process and then find the suitable approximating functions $f_h$. To do so  we divide $[-1,1]$ in $n+1=2/h$ intervals and construct a matrix-valued function that defines a bounded Feller generator that agrees with the pointwise generator of $\Yh{LR}$ on grid points, and \textit{interpolates} it  between grid points (see Example \ref{exp:killB} for an illustration). Most of this work is then dedicated to the delicate task of finding a suitable sequence  of approximating functions $f_h\to f\in \D(\Gen,\BC)$. We perform this procedure for both backward and forward approximations. The backward approximations of Section \ref{sec:caseC} are applied in combination with \eqref{eq:intro_weakconv} to give the pathwise characterisation of the backward generators (in Part II). The forward approximations of Section \ref{sec:caseL}  are used to obtain the wellposedness of the Cauchy problems of Section \ref{sec:intro1}.  Although the interpolation matrices regularise each semigroup of $\Yh{LR}$ (by making it Feller), we pay a cost in the approximation procedure, often due to an extra term of order $\Gru_{0,h}=\psi(1/h)$  to be compensated at the boundary. In four cases out of the twelve in Table \ref{explicitProcesses}, differentiability of $k_0'$  arises as a natural condition to obtain convergence. Here we  work under the stronger assumption \ref{H1}. However, for these singular cases the  resolvent measures are known, thus we are  still able to prove Table \ref{explicitProcesses} for all our spectrally positive $Y$ (see Remark \ref{rmk:knowntwosided}). Besides the above difficulties, our interpolation matrix method allows to construct Feller semigroups on a bounded domain that yield the three equivalent type of convergence of Theorem \ref{thm:TK}, namely, the convergence of generators, semigroups and c\`adl\`ag paths. This is remarkable because  before the embedding procedure the approximating processes are not Feller, and the domain of the generators in Table \ref{explicitProcesses} is not smooth. Moreover it is natural to expect that this technique is  extendable to fast-forwarded non-recurrent  $Y$, Robin and other mixed boundary conditions, to other L\'evy processes such as   $\alpha$-stable processes  in 1-dimension \cite{MR3875559,MR144387}, and possibly to general one-sided Markov additive processes \cite{MR3301294} and higher dimensions. These extensions are indeed part of our current investigations.  

\section{Nonlocal operators with boundary conditions}\label{sec:2}
\subsection{General notation}
 We denote by $\mathbb N,\mathbb N_0, \mathbb R$, $\mathbb C$, $\mathbf1_{A}$, $a\wedge b$ and $a\vee b$ the positive integers, the non-negative integers, the real numbers, the complex plane, the indicator function of a set $A$, the minimum and the maximum of $\{a,b\}\subset \R$ respectively. We denote by $f^{(n)}= \frac{\dd^n}{\dd x^n}f$ the $n$-th integer differentiation of functions on subsets of $\mathbb R$, and if $n=1,2$  we may use   $f'$ and $f''$,  respectively.  For two functions $f,g$ supported on $\mathbb R$ and $a\in \mathbb R$ we use the asymptotic notation    $f=o(g)$ and $f=O(g)$ as $x\to a$ to respectively mean $\lim_{x\to a}f(x)/g(x)=0$ and there exists a $c>0$ such that  $|f(x)|\le c|g(x)|$ for all $x$ close to $a$.   We define  
the standard operators for $x\in \mathbb R$ and $\xi\in \mathbb C$ 
 \begin{align*}
 f\star g(x)&=\int_{\mathbb R} f(x-y)g(y)\,\dd y,  \quad I_{\pm}f(x)=[\mathbf1_{\{\pm x>0\}}]\star f(x),\quad
\widehat f(\xi)=\int_{\mathbb R}e^{-\xi y}f(y)\,\dd y,
 \end{align*}
for suitable functions $f,g:\mathbb R\to\mathbb R$, and  we recall the identity
  \begin{equation}\label{eq:convolutionsign}
 f\star g(x)=[f(-\cdot)]\star [g(-\cdot)](-x).
 \end{equation}
For $f: \Omega '\to\mathbb R$ and $\Omega \subset \Omega '\subset \mathbb R$, we define the zero extension operator  
\begin{equation}\label{eq:zeroext}
\Pi^{-1}_{\Omega }f(x)=\left\{\begin{aligned}
&f(x), &x\in \Omega ,\\
&0, &x\in \mathbb R\backslash \Omega ,
\end{aligned}
\right.
\end{equation} 
and the projection operator $\Pi_{\Omega }f(x)=f(x),$ $x\in \Omega $. 
\begin{remark}[Notational convention]\label{rem:zeroext}
If $f:\Omega \to \mathbb R$, $g:\Omega '\to\mathbb R$, for $\Omega ,\Omega '\subset \mathbb R$, then we abuse notation    by writing  $$f\star g=\Pi^{-1}_\Omega f\star \Pi^{-1}_{\Omega '}g.$$ Also, if $\Omega \subset\Omega '$ and   $A$ is an operator defined on functions on $\Omega$ then we abuse notation by writing $Ag = A\Pi_\Omega g$.
\end{remark}

For any set $\Omega\subset \mathbb R$ we denote by $C(\Omega)$ the set of real-valued continuous functions on $\Omega$ and we denote the supremum norm on $\Omega$ by $\supnorm{\cdot}{\Omega}$. We denote by $C_c(\Omega)$ and $C_0(\Omega)$ the subsets of $C(\Omega)$ consisting of the compactly supported functions and the closure in the supremum norm of $C_c(\Omega)$, respectively. We recall that $C_0(\Omega)$ equipped with the supremum norm is a Banach space.
 For an interval $\Omega$, we define $C_0^n(\Omega):=\{f\in C_0(\Omega): f^{(m)}\in C_0(\Omega), \,m\le n \}$ for any $n\in\mathbb N$,  $C_0^\infty(\Omega):=\cap_{m\ge1}C_0^m(\Omega)$, and $C_c^\infty(\Omega):=C_c(\Omega)\cap C_0^\infty(\Omega)$. 
If $X$ is a Banach space, we denote its norm by $\|\cdot\|_X$. We denote by $L^1(\Omega)$ the Banach space of real-valued Lebesgue integrable functions on $\Omega$. By $L^1_{{\rm loc}}(\Omega)$ we denote the space of real-valued locally integrable functions on $\Omega$, and by $ W^{1,1}(\Omega)$ we denote the standard Sobolev space of functions in $L^1(\Omega)$ with weak derivative in $L^1(\Omega)$.  

We recall some notions from the theory of operator semigroups \cite{MR3156646}.   A collection of operators $P=\{P_t\}_{t\ge 0}$ is said to be a \emph{strongly continuous contraction semigroup}  on a Banach space $X$ if $P_t:X\to X$ is   bounded and linear for any $t>0$,  $P_sP_t=P_{s+t}$ for all $t,s>0$,   $P_0$ is the identity operator and  for any $f\in X$ it holds that $\|P_t f\|_X\le \|f\|_X$ for all $t>0$  and $\|P_t f-f\|_X\to 0$ as $t\downarrow 0$.  We define the \textit{generator of} $P$ to be the pair $(\calG,\mathcal D)$, where $\mathcal D:=\{f \in X:  \calG f \;\text{converges in}\; X\}$ with $\calG f:=\lim_{t\downarrow 0}(P_tf-f)/t$, and we call $\mathcal D$ the \textit{domain of the generator} (of $P$). A set $\mathcal C\subset \mathcal D$ is said to be a \textit{core for $(\calG,\mathcal D)$} if $(\calG,\mathcal D)$ equals the closure of the restriction of $\calG$ to $\mathcal C$. Moreover, $P$ is said to be \emph{positive} if $0\le f \in X$ implies  $0\le P_tf$ for any $t>0$. A positive strongly continuous  contraction semigroup $P$ is said to be a \emph{Feller} semigroup if $X=C_0(\Omega)$ for some $\Omega\subset \mathbb R$.  We recall \cite[Page 15]{MR3156646} that  there exists a one-to-one correspondence between Feller semigroups on $C_0(\Omega)$ and Feller processes with state-space $\Omega$, where  a Feller process with state-space $\Omega$ is a time-homogeneous  Markov processes   $\{Z_t\}_{t\ge0}$ such that  $f(\,\cdot\,)\mapsto P_tf(\,\cdot\,):=\mathbb E[f(Z_t)\,|\,Z_0\!=\cdot\,] $, $t\ge0$, is a Feller semigroup on $C_0(\Omega)$. We recall the following well known Trotter–Kato convergence Theorem, which is a combination of \cite[Theorem 6.1, Chapter 1]{MR838085} and \cite[Theorem 17.25]{MR1464694}.
\begin{theorem}[Trotter–Kato]\label{thm:TK}
Suppose $P,\,P_j$, $j\in\mathbb N$, are strongly continuous contraction semigroups on $X$, denote by $(\calG,\D),\,(\calG_j,\D_j)$, $j\in\mathbb N$, the respective generators and let $\mathcal C$ be a core for $(\calG,\D)$. Then the following are equivalent:
\begin{enumerate}[(i)]
    \item if $f\in X$, then $P_{j,t}f\to P_tf$ in $X$ as $j\to\infty$  uniformly in $t$ in any compact time interval;
    \item if $f\in  \mathcal C $, then there exist $f_j\in \D_j$, $ j\in\mathbb N$, such that $f_j\to f$ and $\calG_jf_j\to \calG f$ as $j\to\infty$.
    \end{enumerate}
    Moreover, if $P,\,P_j$, $j\in\mathbb N$, are Feller semigroups on $X=C_0(E)$ for $E\subset \mathbb R$ and  $Z,\,Z_j$, $j\in\mathbb N$, are the respective Feller processes, then (i) and (ii) above are equivalent to
    \begin{enumerate}[(iii)]
        \item If $Z_{j,0}\Rightarrow Z_{0}$ as $j\to\infty$  on $E$, then $Z_j\Rightarrow Z$ as $j\to\infty$ on $D([0,\infty), E_\delta)$,
    \end{enumerate}
    where $D([0,\infty), E_\delta)$ is the space of $ E_\delta$-valued c\`adl\`ag paths with the Skorokhod $J_1$ topology and $  E_\delta$ is the one-point compactification of $E$.
\end{theorem}
 
\subsection{One-sided L\'evy derivatives and integrals}

In this section we first present our main assumptions and some basic properties of scale functions along with smoothness results. Then we construct our candidate backward and forward generators of Table \ref{explicitProcesses} by defining the L\'evy Riemann–Liouville  and mixed Caputo  derivatives and imposing the boundary conditions $\BC$.  \\
We will work with the following hypotheses.
\begin{description}
\item[(H0)\label{H0}] For a non-negative Borel measure $\phi$  on $(0,\infty)$ such that
\[
\int_{(0,\infty)} (z^2\wedge z)\,\phi(\dd z)<\infty\quad\text{and}\quad \int_{(0,1)} z\,\phi(\dd z)=\infty,
\]
  we define the symbol $\sym(\xi) =\xi^2\widehat\Phi(\xi)$ for $\Re \xi\ge 0$, where  
\begin{equation*}
\Phi(x)=\Phi^+(x)=\left\{\begin{split}&\int_x^\infty \phi(y,\infty)\,\dd y,  &x>0,\\
&0,  & x\le0,
\end{split}\right.
\label{eq:}
\end{equation*}  $\Phi^-(x)=\Phi^+(-x)$, and $\phi(y,\infty):=\phi((y,\infty))$.


\item[(H1)\label{H1}]  Assumption \ref{H0} holds and  either:\begin{enumerate}[(i)]
\item   $\Phi$ is 2-regular \cite[Definition 3.3]{MR2964432} and for any $c>0$,
$\int_1^\infty \bar w_c(r)\,\dd r<\infty$ where  $\bar w_c(r):=\sup\{|\xi/\psi(\xi)|:|\xi|\ge r,\,\Re \xi>c\}/r$; or
 
\item     the function $x\mapsto \phi(x,\infty)$ is completely monotone.

\end{enumerate}

\end{description}
 
\begin{remark}\label{rmk:phi}$\,$
\begin{enumerate}[(i)]
\item Note that \ref{H0} with Fubini's Theorem  imply that  $\sym$ is indeed the L\'evy symbol \eqref{eq:sym_intro}. Also, recall that a  one-sided L\'evy process $Y$ with no diffusion component has paths of unbounded variation if and only if its L\'evy measure satisfies $\int_{(0,1)} z\,\phi(\dd z)=\infty$ and it is recurrent if and only if its L\'evy symbol satisfies  $\sym'(0+)=0$ \cite[Chapter 8.1]{MR2250061}. With these results it is easy to prove that \ref{H0} describes precisely all spectrally positive recurrent L\'evy processes with paths of unbounded variation and no diffusion component.
	\item 
Assumption \ref{H0} implies that $\Phi\in C_0[1/T,\infty)\cap L^1(0,T) $ for all $T>0$, it is non-negative, non-increasing and convex. Also, for any $c>0$, the function $\widehat\Phi(\xi)= \int_0^\infty e^{-\xi x}\Phi(x)\,\dd x$ is analytic and bounded   on $\mathbb C\cap \{ \Re\xi\ge c\}$ and  
  $\xi\widehat\Phi(\xi)$ is a Bernstein function \cite[Eq. (3.3)]{MR2978140}.
Moreover, by the Dominated Convergence Theorem, as $x\to \infty$ 
\begin{equation}
\begin{split}
\frac{\sym(x)}{x^2}&\to 0,\quad
\frac{\sym(x) }x\to \infty,\quad
\sym(1/x)\to 0, \quad
\sym'(x) \to \infty,\quad
\frac{\sym'(x)}x \to 0,
\end{split}
\label{eq:convhomega}
\end{equation}  
using $\sym(x)=x\int_0^\infty(1-e^{-xy})\phi(y,\infty)\dd y$ and $\sym'(x)=\int_0^\infty (1-e^{-xy})y\phi(\dd y)$.
 \item Assumption \ref{H1}-(i) is used at its full potential only to prove \eqref{eq:GClimk-2k0Tricky}, and this result is only used in the proof of Theorem \ref{thm:conv_C_N*N}. As for the remaining three theorems of Section \ref{sec:case} where \ref{H1} is assumed, we can weaken  \ref{H1}-(i) to: $\Phi$ is 2-regular and   for every $c>0$ it holds that $\int_{c+i\mathbb R}1/|\sym(\xi)|\,\dd \xi<\infty$. 
\end{enumerate}

\end{remark}

\begin{example}\label{exp:ker}$\,$
\begin{enumerate}[(i)]
\item In the fractional case $\phi(\dd y)=y^{-1-\alpha}/\Gamma(-\alpha)\,\dd y$, $\alpha\in(1,2)$, so that $\sym(\xi)=\xi^\alpha$, which satisfies both \ref{H1}-(i) and \ref{H1}-(ii). 
\item The tempered fractional case $\phi(\dd y)=c_\alpha e^{-\LTp y}y^{-1-\alpha}\,\dd y$, $\alpha\in[1,2)$, $c_\alpha>0$ and a tempering parameter $\LTp>0$,     satisfies \ref{H1}-(ii) (recalling that the product of completely monotone functions is completely monotone). If $\alpha\in(1,2)$, $c_\alpha=1/\Gamma(-\alpha)$,  then $\sym(\xi)=(\LTp+\xi)^\alpha-\LTp^\alpha-\alpha \LTp^{\alpha-1} \xi$ \cite[Eq. (14)]{MR3342453}. 
\item Assumption  \ref{H0} allows for L\'evy measures with finite range of interaction, such as the truncated stable density $\phi(\dd y)=y^{-1-\alpha}\indi{y\in(0,K]}\,\dd y$, $\alpha\in[1,2)$, where the finite range of interaction $K$ is often a preferred choice for numerical schemes \cite{MR3096457,MR3323906,D19}. Also note that by the example below, this last L\'evy measure satisfies \ref{H1}-(i)  for all $\alpha\in(1,2)$.
\item A  sufficient condition for \ref{H1}-(i) to be satisfied is for $\phi$ to:  posses a non-increasing  density (implying that   $\Phi$ is 3-monotone \cite[Definition 3.4]{MR2964432} so that by \cite[Proposition 3.3]{MR2964432}  $\Phi$ is 2-regular);  and to satisfy the stable-like assumption $\lim_{x\to0}\phi(x,\infty)/x^{-\alpha} =C,$ for some $\alpha\in(1,2)$ and $C>0$. This is because the second  condition  implies the Abelian type estimate 
\begin{equation}\label{eq:Abelianest}
    |\sym(\xi)|\ge  b |\xi|^\alpha
\end{equation} for all large $|\xi|$ with $\Re \xi>0$ and some positive constant $b$ independent of  $\xi$. This implies the integrability condition in \ref{H1}-(i). (A proof of \eqref{eq:Abelianest} is given in Appendix \ref{app:y^2ker2}.)

\item Note that assumption \ref{H0} allows for   L\'evy measures that are not absolutely continuous with respect to Lebesgue measure.



\end{enumerate}
\end{example}
We now use the well known theory of scale functions for spectrally positive L\'evy processes (see, e.g., \cite[Chapter VII]{MR1406564}) to derive some identities key to this work.

\begin{lemma}\label{thm:ker}
Assume \ref{H0}. Then there exists a non-negative function of sub-exponential growth $ k_{-1}\in L^1_{\text{loc}}[-1,\infty)\cap C(-1,\infty)$   such that 
\begin{equation*}
\int_0^\infty e^{-\xi y}k_{-1}(y-1)\,\dd y=\frac{\xi}{\sym(\xi)},\quad\text{for}\quad\Re\xi>0.
\end{equation*}
 In particular, recalling the notation in \eqref{eq:zeroext} and defining $k^+_{-1}=\prod^{-1}_{(-1,\infty)}k_{-1}$, $k^-_{-1}(x)=k^+_{-1}(-x)$ and 
\[
\text{$k_{n-1}^{\pm}:=I^n_{\pm}k_{-1}^{\pm}$}\quad\text{for integers $n\ge 1$,}
\]
 the following identities hold for all $x\in\mathbb R$
\begin{align}
\Phi^{\pm}\star k_n^{\pm}(x)&= \frac{(1\pm x)^{n+1}}{(n+1)!} \ind{1\pm x},\quad n\ge-1, \label{eq:psikn}
\end{align}
where $0!=1$ by convention.

\end{lemma}
\begin{proof} 
 By \cite[Theorem 2.1 and Lemma 2.3]{MR3014147} there exists a scale function $W\in C[0,\infty)$ strictly increasing that satisfies $W(0)=0$ and $\int_0^\infty e^{-\xi x}W(x)\,\dd x=1/\sym(\xi)$ for $\Re \xi>0$. Moreover, $W\in C^1(0,\infty)$ by \cite[Lemma 2.4]{MR3014147}, so that $W'\ge0$. Then
 \begin{align*}
 \frac{\xi}{\sym(\xi)}&= \lim_{b\to \infty,a\to 0+}\xi \int_a^b e^{-\xi y}W(y)\,\dd y\\
  &= \lim_{b\to \infty,a\to 0+}\int_a^b e^{-\xi y}W'(y)\,\dd y -[e^{-\xi y}W(y)]_a^b\\
  &=  \int_0^\infty e^{-\xi y}W'(y)\,\dd y.
 \end{align*}
 Then  $W'\in L^1_{\text{loc}}[0,\infty)$ and it is of sub-exponential growth. Also, letting $k_{-1}(y-1)=W'(y)$ on $(0,\infty)$, we obtain the  identities in \eqref{eq:psikn}. Indeed, for the `$+$' case,  $\Phi^{+}\star k_n^{+}\in C[-1,\infty)$ is of sub-exponential growth for each $n\ge0$, as $\Phi^{+}\star k_n^{+}(x)\le k_n^+(x)[\int_{0}^{2}\Phi(y)\,\dd y+(x-1)\Phi(2)]$ for $x>1$, and observe that $\Phi^{+}\star k_{-1}^{+}\in L^1_{\text{loc}}(\mathbb R)$. Then for any $\LTp>0$ 
\begin{align*}
\widehat{\Phi^{+}\star k_{n}^{+}}(\LTp)&=\int_{-1}^\infty e^{-\LTp y}\int_{-1}^y\Phi^{+}(y-z) k_{n}^{+}(z)\,\dd z\,\dd y\\
&=\widehat{\Phi}(\LTp)\int_{-1}^\infty e^{-\LTp z}k_{n}^{+}(z)\,\dd z=\widehat{\Phi}(\LTp )\frac{e^{\LTp}}{\LTp^{n+1}}\frac{\LTp}{\sym(\LTp)}=\frac{e^{\LTp}}{\LTp^{n+2}},
\end{align*}
 for all $n\ge-1$, so that uniqueness of Laplace transforms \cite[Theorem 1.7.3]{MR2798103} proves \eqref{eq:psikn}. The `$-$' case now follows immediately from \eqref{eq:convolutionsign}.
 
\end{proof}

For the approximation Theorems \ref{thm:conv_C_DD}, \ref{thm:conv_C_N*D},  \ref{thm:conv_C_N*N} and \ref{thm:conv_L_N*D}  we require continuous differentiability of $k_{-1}$. Apparently this result is not available  for spectrally positive process with paths of unbounded variation if there is no diffusion term \cite{MR2824871}. We therefore prove the proposition below ourselves under assumption \ref{H1}. 

\begin{proposition}\label{prop:y^2ker2} 
If \ref{H1} holds, then there exists a function $k_{-2}\in C(-1,\infty)$ so that   $y^2k_{-2}(y-1)$ is in $L^1(0,1)$, of sub-exponential growth and satisfies
\begin{equation}\label{eq:y^2k-2lap}
\int_0^\infty e^{-\xi y}y^2k_{-2}^+(y-1)\,\dd y=\left(\frac{\xi^2}{\sym(\xi)}\right)'',\quad\text{for}\quad\Re\xi>0,
\end{equation}  
 denoting $k_{-2}^+=\prod^{-1}_{(-\infty,-1)}k_{-2}$.
\end{proposition} 
\begin{proof} Assume \ref{H1}-(ii). Then, as $\phi(x,\infty)$ is completely monotone, $\sym(\xi)/\xi =\int_0^\infty (1-e^{-\xi x})\phi(x,\infty)\,\dd x$ is a complete Bernstein function. Then by \cite[Proposition 7.1]{MR2978140}  there exists a completely monotone function $\chi^*$  with $\int_0^\infty(1\wedge x)\chi^*(x)\,\dd x<\infty$ and real constants $a^*,\,b^*$ so that $\xi^2/\sym(\xi)$ is the complete Bernstein function  
\[
a^*+b^*\xi +\int_0^\infty(1-e^{-\xi x})\chi^*(x)\,\dd x.
\]
By differentiating twice in $\xi$ and relabelling $k_{-2}(x-1)=-\chi^*(x)$ our proof is complete. The proof under assumption \ref{H1}-(i) can be found in Appendix \ref{app:y^2ker2}.
\end{proof}
\begin{remark}
It is possible to show that $\phi(\dd x):=(x^{-2.5}\dd x+\delta_1(\dd x))$ satisfies \ref{H0}, but $k_{-2}\not\in C(-1,\infty)$, and thus it does not satisfy \ref{H1}. (Here $\delta_1$ is the Dirac delta measure at 1.) However  we omit the proof as inconsequential for this work. 
\end{remark}
We are now ready to define our  nonlocal L\'evy derivatives of Riemann–Liouville and mixed Caputo type along with the nonlocal L\'evy integral of Riemann–Liouville type. 
 
\begin{definition}\label{def:ndo} Assume \ref{H0} and recall the conventions of Remark \ref{rem:zeroext}. We define  \emph{the left (`$+$') and right (`$-$')   L\'evy derivatives     on $\mathbb R$} by
\begin{align*}
 \Mpm f&= \frac{\dd^2}{\dd x^2}\left[\Phi^{\pm}\star f\right],\end{align*}
 for   $f\in L^1_{{\rm loc}}(\mathbb R)$ such that $\Mpm f\in L^1_{{\rm loc}}(\mathbb R)$.\\ We define  \emph{the left (`$+$') and right (`$-$') L\'evy derivatives of Riemann–Liouville type  on $X\in\{L^{1}[-1,1], C_0(\Omega)\}$}  by 
 \begin{align*}
\RLpm  f&=\Pi_{[-1,1]}\Mpm f, \end{align*}
 for $f\in X$ such that $ \RLpm f\in X$.\\ We define   \emph{the left (`$+$') and right (`$-$') L\'evy  derivatives of mixed Caputo  type on $X\in\{L^{1}[-1,1], C_0(\Omega)\}$}  by   \begin{align*}
\Cpm f&= \Pi_{[-1,1]} \frac{\dd}{\dd x}\left[\Phi^{\pm}\star \frac{\dd}{\dd x}f\right],
\end{align*} 
for $f\in  W^{1,1}[-1,1]$ such that $\Cpm f\in X$.\\
We define the  left (`$+$') and right (`$-$')  L\'evy derivatives $\RLonepm $ and $\Conepm $ for a suitable function $f\in L^1[-1,1]$ respectively as 
 \[ 
\RLonepm f=  \Pi_{[-1,1]} \frac{\dd}{\dd x}\left[\Phi^{\pm}\star f\right]\quad \text{and} \quad \Conepm f = \Pi_{[-1,1]} \left[\Phi^{\pm}\star  \frac{\dd}{\dd x} f\right].
\]
 We define the \emph{left (`$+$') and right (`$-$')   L\'evy   integral of  Riemann–Liouville type} for $f\in L^1[-1,1]$ by
 \[
\Ipm  f= \Pi_{[-1,1]} k^{\pm}_0(\cdot \mp1)\star f. 
\]
For the twelve cases in Table \ref{explicitProcesses} we define the operators $(\Gen,\BC)$ and their domains $\D(\Gen,\BC)$ as in Section \ref{sec:intro1}. 
To ease notation, we sometimes use $\Gen^\pm\in \{\Cpm ,\RLpm \}$ and $\Gen\in\{\Gen^\pm\}$.  
\end{definition}

\begin{remark}\label{rmk:ndo} We collect several simple identities that are   used repeatedly in this work. This also allows us to not have to mention the conventions of  Remark \ref{rem:zeroext} further, simplifying the presentation. 
\begin{enumerate}[(i)]
\item  The operators $\Mpm$ are well-defined for  $f\in C_c^2(\mathbb R)$, because $\Mpm f\in C(\mathbb R)$, and their L\'evy symbols are $\sym(\mp ik)$. Moreover, if the support of $f$ is contained in some interval $[a,b]$, then for $x\neq a,b$
\begin{align*}
\Mr f(x)
&= \frac{\dd^2}{\dd x^2}\int_{0}^{b-x}\Phi(y)f(x+ y)\,\dd y\, \ind{b-x},\\
\Ml f(x)
&= \frac{\dd^2}{\dd x^2}\int_{0}^{x-a}\Phi(y)f(x- y)\,\dd y\, \ind{x-a}.
\end{align*} 
\item Recalling  Remark \ref{rem:zeroext}, notice that $\Ipm  f(\mp1)=0$   if $f\in L^1[-1,1]$,  and that 
  for all $x\in [-1,1]$,
\begin{equation}
\Ipm 1(x)= I_\pm k_0^\pm(x)=k_1^\pm(x). 
\label{eq:Iphiofind}
\end{equation}

	\item Recalling  Remark \ref{rem:zeroext}, note that   \eqref{eq:psikn} implies 
 \begin{equation}\label{eq:ker}
0=\RLpm k_{-1}^\pm=\RLpm k_0^\pm=\Cpm k_0^\pm = \Cpm 1,
\end{equation}
using $\Pi_{[-1,1]}^{-1}\Pi_{[-1,1]}k_n^\pm(x)=k_n^\pm(x)$ for $\pm x\in(-\infty,1]$, $n\ge-1$, so that, for example,
\[
 \Phi^\pm\star \Pi^{-1}_{[-1,1]}[\Pi_{[-1,1]}k_{-1}^\pm](x)=  \Phi^\pm\star k_{-1}^\pm(x)=   \ind{1\pm x}.
\]  
Also, the operators $\Ipm $ are right inverses of our nonlocal derivatives, as for $f\in L^1[-1,1]$
\begin{equation}\label{eq:inv}
\Cpm \Ipm  f=\RLpm \Ipm  f=f.
\end{equation}
 To see \eqref{eq:inv}, observe that  $\Ipm  f(\mp1)=0$ and Lemma \ref{lem:diphi} imply the first identity, and for the second identity use $\Phi^\pm\star k_0^\pm(\cdot\mp 1)(x)= \pm x\ind{\pm x}$ 
 to prove that for  all $x\in(-1,1)$
\begin{align*}
&\,\Phi^\pm\star \Pi_{[-1,1]}^{-1} \left[   \Pi_{[-1,1]}  [k_0^\pm(\cdot\mp 1)\star \Pi_{[-1,1]}^{-1} f]\right](x)\\
=&\, \Phi^\pm\star [k_0^\pm(\cdot\mp 1)\star \Pi_{[-1,1]}^{-1} f](x)\\
=&\,\big[ \Pi_{[-1,1]}^{-1} f\big]\star \Phi^\pm\star k_0^\pm(\cdot\mp 1)(x)\\
 =&\,\int_{\mathbb R} \big[\Pi^{-1}_{[-1,1]} f\big](x-y)(\pm y)\ind{\pm y} \,\dd y\\
=&\,\int_{0}^{x\pm 1} f(x-y)y\,\dd y=I_\pm^2f(x),
\end{align*}
using the convention $\int_0^{x-1} =-\int_{x-1}^0$ in the `$-$' case. Similar computations prove that for $f\in L^1[-1,1]$
\begin{equation}\label{eq:phi-1} 
\Conepm \Ipm  f=\RLonepm \Ipm  f=\pm I_\pm f. 
\end{equation}

\item We could equivalently define   $\Cpm $ as    $ \RLpm [f-f(\mp 1)]$, for $f,f(\pm 1)\in  W^{1,1}[-1,1]$ such that $\Cpm f\in X$, which can be proved a similar argument as for the proof of Lemma \ref{lem:diphi}.
\end{enumerate}
\end{remark}


We now prove the explicit representation in terms of scale functions of the domains $\D(\Gen,BC)$ in Table \ref{tab:ndo}.
\begin{lemma}  Assume \ref{H0}. Then the domain $\D(\Gen,BC)$  of the twelve nonlocal differential operators in Table \ref{explicitProcesses} can be represented as in  Table \ref{tab:ndo}.
\end{lemma} 
\begin{proof} Recall the definition of $\D(\Gen,BC)$ form Section \ref{sec:intro1}. The inclusion `$\supset$' is clear by direct computation using Remark \ref{rmk:ndo}-(iii) and Lemma \ref{lem:diphi}. To prove the inclusion `$\subset$'  we first to show that 
$f\in X$ and $\Gen^\pm f\in X$ imply that there exist $g\in X$ and constants $c,c_1,c_0,c_{-1}\in\mathbb R$, such that 
in the mixed Caputo case ($\Gen^\pm=\Cpm$) 
\begin{align*}
f(x)&=  \Ipm  g(x)+ c_{1}k^\pm_1(x) + c_{0}k^\pm_0(x)+ c ,\\
\Conepm  f(x)&=\pm I_\pm g(x) \pm c_{1}(1\pm x)\pm c_{0},\\
\Cpm  f(x)&=  g(x) + c_{1} ,
\intertext{meanwhile in the left Riemann–Liouville case ($\Gen^+=\RLl)$}
f(x)&= \Il  g(x)+c_{1} k^+_1(x)+c_{0}k^+_0(x)+c_{-1}k^+_{-1}(x),\\
\RLonel  f(x)&= I_+g(x)+c_{1} (1+x)+c_{0},\\
\RLl  f(x)&= g(x)+c_{1} ,
\end{align*}
and secondly we show the correct value of the constants due to the boundary conditions $\BC$. We start with the second claim.
For the case $\D (\Cl ,\mathrm{DN})$ (so that $X=L^1[-1,1]$),  $c=0$ is necessary to guarantee $f(-1)=0$,  and $c_0=-I_+[g+c_1](1)$ to obtain $\Conel  f(1)=0$, and so $c_1$ is free so we can let it be zero using \eqref{eq:Iphiofind}.
For the case $\D (\RLl ,\mathrm{N^*N})$ (so that $X=L^1[-1,1]$), $c_0=0$ and $c_1=-I_+g(1)/2$ to respectively guarantee $\RLonel f(-1)=0$  and $\RLonel f(1)=0$ , meanwhile $c_{-1}$ is free as $k_{-1}^+$ belongs to the kernel of $\RLl $. For the case $\D (\Cr ,\mathrm{DN})$ (so that $X=C_0(-1,1]$), we must have $c_0,c_1=0$ to respectively ensure that $\Coner  f(1)=0$ and  $\Cr f(-1)=0$, and so $c=-\Ir g(-1)$ to obtain $f(-1)=0$. The remaining nine cases are dealt with similarly and omitted. We only mention that for the $\mathrm{N^*R}$ cases we use Lemma \ref{lem:diphi}. \\
It remains to prove the first claim to obtain the two displays above, and due to their similarity we only prove the  Riemann–Liouville  case ($\Gen^+=\RLl)$. The third identity is obvious and second identity follows by applying $I_+$ to both sides in  $\RLl f = g+c_1$. To prove the first identity observe that $\RLl f = g+c_1$ implies that $\Il\RLl f = \Il g+c_1 k_1^+ $, and we compute
\begin{align*}
\Il\RLl f(x) &= \int_{-1}^x k_0^+(x-y -1)[\Phi^+\star f]''(y) \,\dd y\\
 &= \int_{-1}^x k_{-1}^+(x-y -1)  [\Phi^+\star f]'(y) \,\dd y - k_0^+(x  ) [\Phi\star f]'(-1)\\
  &= \frac{\dd}{\dd x} \int_{0}^{1+x} k_{-1}^+( y -1)  [\Phi^+\star f](x-y) \,\dd y- k_{-1}^+( x)  [\Phi^+\star f](-1)\\
  &\quad- k_0^+(x  ) [\Phi^+\star f]'(-1)\\
    &= \frac{\dd}{\dd x}  k_{-1}^+( \cdot -1) \star \Phi^+\star f(x)  - k_{-1}^+( x)  [\Phi^+\star f](-1)- k_0^+(x  ) [\Phi\star f]'(-1)\\
      &=   f(x)  - k_{-1}^+( x)  [\Phi^+\star f](-1) - k_0^+(x  ) [\Phi^+\star f]'(-1),
\end{align*}
with the last identity holding a.e., so that we obtained the existence of the constants $c_0$ and $c_{-1}$, where we used \eqref{eq:psikn} and that  $  [\Phi^+\star f]', I_+f\in W^{1,1}[-1,1]$, $\Phi^+\star f\in C^1[-1,1]$, which are consequences of the definition of $\RLl$.
\end{proof}

 \begin{table}[h]
\centering
\vline
\begin{tabular}{l l|}
  \hline
  \multicolumn{2}{c}{$\mathcal D(\Gen,\BC )=\left\{f\in X:f=\I g+c+\sum_{j=-1}^1c_jk_j,\,\,g,c_1\in X\right\}$}\vline\\ \hline

  \multicolumn{2}{c}{$X = L^1[-1,1]$}  
	\vline
	\\
1.	 & $\D (\Cl , \mathrm{DD}) = \left\{f\in X:\;f=\Il  g- \frac{\Il  g (1)}{k^+_0(1)}k^+_0,\;g\in X\right\}$\\
2.	 & $\D (\Cl , \mathrm{DN})=\left\{f\in X:\;f=\Il  g- I_+ g (1)k^+_0,\;g\in X\right\}$\\
3.	& $\D (\Cl , \mathrm{ND})=\left\{f\in X:\;f=\Il  g -\Il  g(1),\;g\in X\right\}$\\
4.	 & $\D (\Cl , \mathrm{NN})=\left\{f\in X:\;f=\Il  g-\frac{I_+g(1)}{2} k^+_1+	c, \;g\in X , c\in\R\right\}$\\
5.	 & $\D (\RLl , \mathrm{N^*D})=\left\{f\in X:\;f=\Il  g - \frac {\Il  g(1)}{k^+_{-1}(1)} k^+_{-1},	\;g\in X \right\}$\\
6.	 & $\D (\RLl , \mathrm{N^*N})=\left\{f\in X:\;f=\Il  g-\frac{I_+g(1)}{2} k^+_1+	c_{-1}k^+_{-1},	\;g\in X , c_{-1}\in\R\right\}$\\ 
  \hline
  \multicolumn{2}{c}{$X = C_0(\Omega)$}  
	\vline \\
1.	 & $\D (\Cr , \mathrm{DD}) = \left\{f\in X:\;f=\Ir  g-\frac{\Ir  g (-1)}{k^-_0(-1)}k^-_0,\;g\in X\right\}$\\
2.	 & $\D (\Cr , \mathrm{DN})=\left\{f\in X:\;f=\Ir  g- \Ir  g (-1),\;g\in X\right\}$\\
3.	 & $\D (\Cr , \mathrm{ND})=\left\{f\in X:\;f=\Ir  g -I_- g(-1)k^-_0, \;g\in X\right\}$\\
4.	 & $\D (\Cr , \mathrm{NN})=\left\{f\in X:\;f=\Ir  g-\frac{I_-g(-1)}{2}k^-_1 +c , \;g\in X, c\in\R  \right\}$\\
5.	 & $\D (\Cr , \mathrm{N^*D})=\left\{f\in X:\;f=\Ir  g+ \frac{\frac{\dd}{\dd x}\Ir g(-1)}{k_{-1}^-(-1)} k^-_0,	\;g\in X \right\}$\\
6.	& $\D (\Cr , \mathrm{N^*N})=\left\{f\in X:\;f=\Ir  g+\frac{\frac{\dd}{\dd x}\Ir g(-1)}{k^-_0(-1)} k^-_1+	c,	\;g\in X, c\in\R \right\}$\\
  \hline
\end{tabular}
\caption{\label{tab:ndo} Domain representation in terms of scale functions for the nonlocal operators $(\Gen,\BC)$ of Table \ref{explicitProcesses}.}
\end{table}  
We prove a basic smoothing property of $\Ipm $ that we used and will use several times.
\begin{lemma}\label{lem:diphi}
Assume \ref{H0}. Then the operators $\Ipm  :L^1[-1,1] \to W^{1,1}[-1,1]$ are continuous and 
\begin{equation}
\frac{\dd}{\dd x}\Ipm  f = \pm k_{-1}^\pm (\cdot\mp1)\star f.
\label{eq:diphi}
\end{equation}
Also the operators $\Ipm  :C[-1,1] \to C^1[-1,1]$ are continuous and \eqref{eq:diphi} holds.
\end{lemma}
\begin{proof} The last statement  follows by straightforward computations using $k_0^+\in C_0(-1,1]\cap C^1(-1,\infty)\cap W^{1,1}[-1,1]$.
To prove the first statement, easy calculations show that for $f\in C[-1,1] $ and $i\in\{0,-1\}$
\[
\|k_{i}^\pm (\cdot\mp1)\star f\|_{L^1{[-1,1]}} \le \|f\|_{L^1{[-1,1]}}\|k_i^\pm\|_{L^1{[-1,1]}},
\]  
and we can conclude by  \eqref{eq:diphi}, the density of $C[-1,1]$ in $L^{1}[-1,1]$ and  the linearity of the operators.
\end{proof}

\subsection{Density of $\D(\Gen,\mathrm{LR})$}
We show that all the operators in Table \ref{explicitProcesses} are densely defined. The  proof is a variation to the one provided in \cite[Theorem 5]{MR3720847}.
\begin{theorem}\label{thm:5}  Assume \ref{H0}. Then the domains of the nonlocal derivative operators $(\Gen, \BC)$ listed in Table \ref{tab:ndo} are dense. 
\end{theorem}
\begin{proof}
Recall that $C_0^\infty(-1,1)$ is dense in $L^1[-1,1]$ and $C_0^\infty(\Omega)$ is dense in $C_0(\Omega)$. Then, for operators on $X=L^1[-1,1]$, we show that for any $g\in C_0^\infty(-1,1)$ and $\epsilon>0$, there exists a function $h_\epsilon^+\in X$ such that $g_\epsilon^+:=g+h_\epsilon^+\in \mathcal D(\Gen^+,\mathrm{LR})$ and $\|g_\epsilon^+-g\|_X\to 0$ as $\epsilon\to 0$, meanwhile, for  $X=C_0(\Omega)$, we show that for any $g\in C_0^\infty(\Omega)$ and $\epsilon>0$,  there exists a function $h_\epsilon^-\in X$ such that $g_\epsilon^-:=g+h_\epsilon^-\in \mathcal D(\Gen^-,\mathrm{LR})$ and $\|g_\epsilon^--g\|_X\to 0$ as $\epsilon\to 0$. Define for  $1>\epsilon>\delta>0$ and constants $C(\epsilon),c(\epsilon)\in\mathbb R$, the function
\begin{equation}\label{eq:hepsilon}
h_\epsilon^+(x):= \left\{
\begin{aligned}
&0, &x\in[-1,1-\epsilon],\\
&C(\epsilon) k_2^+(x-2+\epsilon), &x\in(1-\epsilon,1-\delta],\\
&C(\epsilon)  k_2^+(x-2+\epsilon)-c(\epsilon)k_2^+(x-2+\delta), &x\in(1-\delta,1],
\end{aligned}
\right.
\end{equation}
and the function $h_\epsilon^-(x):=h_\epsilon^+(-x)$. Simple calculations show that $h_\epsilon^+\in  C^2_c(-1,1]$,  $h_\epsilon^-\in  C^2_c[-1,1)$, and so $\Cpm  h_\epsilon^\pm=\RLpm  h_\epsilon^\pm$ and the constants $\Cpm  h_\epsilon^\pm(\mp 1),\Conepm  h_\epsilon^\pm(\mp1)$ equal $0$. Note that   any $g\in C_0^\infty(-1,1)$  satisfies all the possible left boundary conditions for $\Gen^+$ (i.e. $\Cl g(-1),\Conel g(-1)$ and $\RLonel g(-1)$ are $0$). Also, any  $g\in C_0^\infty(\Omega)$  satisfies all the right boundary conditions for $\Gen^-$ (i.e. $\Cr g(1)=0$ when $g\in C_0^\infty[-1,1)$, and $\Coner g(1)=0$ for any $\Omega$). Therefore, we only need to choose $ C(\epsilon),c(\epsilon)$ and $\delta$ such that $g^+_\epsilon$ and $g^-_\epsilon$ satisfy  the right  and left boundary condition, respectively. Such choices of $ C(\epsilon),c(\epsilon),$ and $\delta$  are listed in Table \ref{tab:thm5}.  We only show the proof for  $\D(\Gen^+,\mathrm{LD})$, $\mathcal D(\Gen^+,\mathrm{LN})$ and $\mathcal D(\Gen^-,\mathrm{N^*R})$. The other proofs are similar and omitted. 
Let us first recall the identities $\RLl g=\Cl g$, and $\RLonel g=\Conel g$ for $g\in C_0^\infty(-1,1)$, and the inequality
\begin{equation}\label{eq2plusmth}
k_2^+(\epsilon-1)\le \frac{\epsilon^2}{2}\|k_0^+\|_{C[-1,\epsilon-1]} .
\end{equation}

Case $\D(\Gen^+,\mathrm{LD})$: 
observe that, using \eqref{eq:psikn},
\[
\Cl h_\epsilon^+(x)=C(\epsilon) (x-1+\epsilon)\mathbf1_{\{x>1-\epsilon\}}-c(\epsilon)(x-1+\delta)\mathbf1_{\{x>1-\delta\}},
\]
and letting $c(\epsilon)=\frac{C(\epsilon)\epsilon+\Cl g(1)}\delta$, we satisfy the D right boundary condition
\[
\Cl h_\epsilon^+(1)=C(\epsilon)\epsilon-c(\epsilon)\delta=-\Cl g(1).
\]
Then, the condition $h_\epsilon^+(1)=0$ holds if and only if $C(\epsilon)= c(\epsilon)\frac{k_2^+(\delta-1)}{k_2^+(\epsilon-1)}$, which rewrites as
\[
C(\epsilon)\left(\frac{k_2^+(\epsilon-1)\delta}{k_2^+(\delta-1)}-\epsilon\right)= \Cl g(1).
\]
We can select  $\delta=2\epsilon\frac{k_2^+(\delta-1)}{k_2^+(\epsilon-1)}$, and so we let $C(\epsilon)=\Cl g(1)/\epsilon$, and we can select such $\delta$ because \eqref{eq2plusmth} implies $\delta\mapsto\frac {k_2^+(\delta-1)}\delta\in C_0(0,1]$, and so there exists a $\delta<\epsilon$ such that
\[
\frac{1}{2}\frac{k_2^+(\epsilon-1)}{\epsilon} =\frac {k_2^+(\delta-1)}\delta.
\]
To show that $h_\epsilon^+\to 0$ uniformly as $\epsilon\to 0$, we use \eqref{eq2plusmth} to derive the inequalities
\begin{align*}
|c(\epsilon)k_2^+(x-2+\delta)|&\le\frac{2|\Cl g(1)|}{\delta}\frac{\delta^2}2\|k_0^+\|_{C[-1, \delta-1]}\le |\Cl g(1)|\,\epsilon \|k_0^+\|_{C[-1, \epsilon-1]},
\end{align*}
and 
\begin{align*}
|C(\epsilon)k_2^+(x-2+\epsilon)|&\le\frac{|\Cl g(1)|}\epsilon \frac{\epsilon^2}2\|k_0^+\|_{C[-1, \epsilon-1]}\le |\Cl g(1)|\,\epsilon \|k_0^+\|_{C[-1, \epsilon-1]}.
\end{align*}

Case $\D(\Gen^+,\mathrm{LN})$: observe that,  using \eqref{eq:psikn} and $C(\epsilon)=1$,
\[
\Conel h_\epsilon^+(x)=  \frac{(x-1+\epsilon)^2}2\mathbf1_{\{x>1-\epsilon\}}-c(\epsilon)\frac{(x-1+\delta)^2}2\mathbf1_{\{x>1-\delta\}},
\]
and letting $c(\epsilon)=\frac{ \epsilon^2/2+\Conel g(1)}{\delta^2/2}$, we satisfy the N right boundary condition
\[
\Conel h_\epsilon^+(1)= \epsilon^2/2-c(\epsilon)\delta^2/2=-\Conel g(1).
\]
 To show that $h_\epsilon^+\to 0$ uniformly as $\epsilon\to 0$, we  use \eqref{eq2plusmth} as before, obtaining
\begin{align*}
|c(\epsilon)k_2^+(x-2+\delta)|&\le \frac{\epsilon^2/2+|\Conel g(1)|}{\delta^2/2}\frac{\delta^2}2\|k_0^+\|_{C[-1, \delta-1]}\\
&\le(1+|\Conel g(1)|)\,\|k_0^+\|_{C[-1, \epsilon-1]}.
\end{align*}
 We are done as it is clear that  term in $h_\epsilon^+$ multiplied by $C(\epsilon)=1$ vanishes uniformly.  
Case: $\D(\Gen^-,\mathrm{N^*R})$: observe that 
\begin{align*}
\frac{\dd}{\dd x} h_\epsilon^-(x)&=  -k_1^+(-x-2+\epsilon)+c(\epsilon)k_1^+(-x-2+\delta),
\end{align*}
and letting $c(\epsilon)=\frac{ k_1^+(\epsilon-1)-\frac{\dd}{\dd x}g(-1)}{k_1^+(\delta-1)}$, we satisfy the ${\rm N^*}$ left boundary condition 
\[
\frac{\dd}{\dd x} h_\epsilon^-(-1)= -k_1^+(\epsilon-1)+c(\epsilon)k_1^+(\delta-1)=-\frac{\dd}{\dd x}g(-1).
\]
To see that $h_\epsilon^-\to 0$ uniformly as $\epsilon\to 0$,  observe that
\begin{align*}
|c(\epsilon)k_2^+(x-2+\delta)|&\le\Big( k_1^+(\epsilon-1)+\big|\frac{\dd}{\dd x}g(-1)\big|\Big)\epsilon, 
\end{align*} 
using $k_1^+(-1)=0$ and  
$k_2^+(x-2+\delta)\le (x-1+\delta)k_1^+(x-2+\delta)\le \delta k_1^+(\delta-1)$, due to   $k_1^+$ being non-decreasing and $x\ge 1-\delta$. We are done as it is clear that  term in $h_\epsilon^-$ multiplied by $C(\epsilon)=1$ vanishes uniformly.

\end{proof}

\begin{table}[h]
\centering
\vline
\begin{tabular}{ l l l l|}

	\hline
		  $\D (\Gen, \mathrm{LR})$  &  $C(\epsilon)$ &  $c(\epsilon)$ &  $\delta <\epsilon $   \\
	\hline
	
 
 $\D (\Gen^+, \mathrm{LD})$ & $C(\epsilon)=\frac{\Gen^\pm g(\pm1)}\epsilon $ & $c(\epsilon)=\frac{2\Gen^\pm g(\pm 1)}\delta$ & $\delta =2\epsilon \frac{k_2^+(\delta-1)}{k_2^+(\epsilon-1)}$  \\

$\D (\Gen^-, \mathrm{DR})$ & & &   \\
  \hline
 $\D (\Gen^+, \mathrm{LN})$ & $C(\epsilon)=1$ & $c(\epsilon)=\frac{\epsilon^2/2\pm \Gen^{\pm,{\rm N}}g(\pm 1)}{\delta^2/2}$ & $\text{any}$\\

 $\D (\Gen^-, \mathrm{NR})$ &  &  & \\
\hline
 $\D (\Gen^-, \mathrm{N^*R})$ & $C(\epsilon)=1$ & $c(\epsilon)=\frac{ k_1^+(\epsilon-1)-\frac{\dd}{\dd x}g(-1)}{k_1^+(\delta-1)}$ & $\text{any}$ \\

  \hline
\end{tabular}
\caption{\label{tab:thm5} Choices of $C(\epsilon),c(\epsilon)$ and $\delta$ for the proof of Theorem \ref{thm:5}, where $\Gen^{\pm,{\rm N}}\in \{\Conepm,\RLonel \}$.}

\end{table}

\subsection{Closedness of $(\Gen,\BC)$}
We prove that  all the operators in Table \ref{explicitProcesses} are closed and then identify an appropriate core.
\begin{theorem}\label{thm:7} Assume \ref{H0}. Then the operators $(\Gen,\mathrm{LR})$ in Table \ref{explicitProcesses} are closed.
\end{theorem}
\begin{proof}
The argument of \cite[Theorem 7]{MR3720847} can be used in our setting for all twelve cases and we omit it. We only observe that all the determined coefficients $c,c_1,c_0,c_{-1}$ in Table \ref{tab:ndo} are continuous for $g_n\to g$ in $X$ as a consequence of Lemma \ref{lem:diphi}.
\end{proof}

\begin{proposition}\label{prop:cores}
The set $\mathcal C(\Gen,\mathrm{LR})$ is a core for  $(\Gen,\mathrm{LR})$, where $\mathcal C(\Gen,\mathrm{LR})$ is  defined as $\D(\Gen,\mathrm{LR})$ in Table \ref{tab:ndo} for $g\in C_c^\infty(-1,1)$ if $X=L^1[-1,1]$ and for $g\in C_c^\infty(\Omega)\cap\{g-g(1)\in C_c[-1,1)\}$  if $X=C_0(\Omega)$.
\end{proposition}
\begin{proof} Let $f\in \D(\Gen,\mathrm{LR})$, and denote $ g+c_1=\Gen f\in X$ and $c,c_1,c_0,c_{-1}$ the coefficients for $f$ as in Table \ref{tab:ndo}. Recall that $C_c^\infty(-1,1)$ is dense in $L^1[-1,1]$ and $C_c^\infty(\Omega)\cap\{g-g(1)\in C_c[-1,1)\}$ is dense in  $C_0(\Omega)$. Then we can take a sequence $\{g_n\}_{n\in\mathbb N}\subset C_c^\infty(-1,1)$ for $X=L^1[-1,1]$  and $\{g_n\}_{n\in\mathbb N}\subset C_c^\infty(\Omega)\cap\{g-g(1)\in C_c[-1,1)\}$ for  $X=C_0(\Omega)$ such that $g_n\to g$ as $n\to\infty$ in $X$, and denote by $f_n$  the element in $\mathcal C(\Gen,\mathrm{LR})$ for each $g_n$ and denote the respective coefficients by $c(n),c_1(n),c_0(n),c_{-1}(n)$, fixing  $c(n)=c$ and $c_{-1}(n)=c_{-1}$ in cases 4 and 6 for any $X$ in Table \ref{tab:ndo}. Then, on the one hand,  $f_n\to f$     by continuity of $\Ipm $  and the coefficients depending on $g_n$ (Lemma \ref{lem:diphi}). On the other hand, due to $c_1(n)$ converging to $c_1$,    $\Gen f_n=g_n+c_1(n)\to g+c_1=\Gen f$, and we are done. 
\end{proof}

\subsection{Surjectivity of $(\LTp -\Gen)$}
In this section we want to construct the resolvents of all operators in Table \ref{explicitProcesses}. It is then natural to define the  \textit{ left (`$+$') and right (`$-$') L\'evy operators of Mittag-Leffler   type} for $g\in L^1[-1,1]$ and $\LTp>0$ as
\[
g\mapsto \Epmq  g:=\sum_{n=0}^\infty \LTp^n(\Ipm )^ng,
\]
with the convention that $(\Ipm )^0$ are identity operators.
\begin{remark}\label{rem:Zkyp} We named the operator $\Elq$ after the Mittag-Leffler function \cite{MR2884383} as we use it mostly  to solve nonlocal differential equations (see Lemma \ref{lem:gML}). But $\Elq 1$ is a well-known scale function in theory of L\'evy processes, sometimes called the adjoint $\LTp$-scale function \cite[Definition 2]{MR1995924}. Indeed, 
recall that $k_0^+(\cdot-1)$ equals $W$ in \cite[Chapter 8.2]{MR2250061}, so that   $(\Il )^n=W\star^{n}$ and from  \cite[Eq. (8.24)]{MR2250061} (or a Laplace transform argument) it follows that  
\[
\Elq g (x)=g(x) +\LTp\int_0^{x+1} W^{(\LTp)}(y)g(x-y)\,\dd y,
\]
where $W^{(\LTp)}= \sum_{n\ge 0}\beta^n W\star^{n+1}$ is the scale function with Laplace transform $(\sym(\xi)-\LTp)^{-1}$ (for large $\Re\xi>0$) and $\star^{n}$ denotes the $n$-th convolution power.
\end{remark}
\begin{lemma}\label{lem:gML}  Assume \ref{H0} and let $\LTp>0$.
 Then $\Epmq  :Y\to Y$ are bounded linear operators for  $Y$ being $L^1[-1,1]$ or $ C[-1,1]$,   $\Erq $  is bounded and linear for    $Y=C_0[-1,1)$ and  $\Elq $  is bounded and linear for    $Y=C_0(-1,1]$. Moreover, the following identities hold 
\begin{equation}\label{eq:MLrel}
\begin{aligned}
\Cpm  \Epmq  \Ipm  g & = \RLpm  \Epmq  \Ipm  g=g+\LTp\Epmq  \Ipm  g,\\
\Cpm  \Epmq  k_0^\pm &= \RLpm  \Epmq  k_0^\pm= \LTp\Epmq  k_0^\pm,\\
\Cpm  \Epmq  1 &= \LTp\Epmq  1,\\
\RLl  \Elq  k_{-1}^+ &= \LTp\Elq  k_{-1}^+, 
\end{aligned}
\end{equation}
where $g\in L^1[-1,1]$, and if  $g\in C[-1,1]$ then $\Erq \Ir g\in C^1[-1,1]$. 
\end{lemma}
\begin{proof}
Recall that 
$ 
\|k_0^+\|_{C[-1,1]} =\|k_0^-\|_{C[-1,1]} = k^+_0(1)=:K ,
$
and $|\Ipm  g(x)|\le 2K  \|g\|_X  $ for all $x\in[-1,1]$, $g\in X$, and any $X\in \{L^1[-1,1],C_0(\Omega)\}$.  For $g\in X$ and $n\ge2$, an induction argument yields for all $x\in[-1,1]$  
\begin{align*}
|(\Ipm )^n g(x)|&\le 2K\|g\|_X(\Ipm )^{n-1}(x)  \\
&= 2K\|g\|_X(\Ipm )^{n-2} I_\pm k_0^\pm (x) \\
&\le 2K\|g\|_X  KI_\pm(\Ipm )^{n-2}(x)\\
&\quad\vdots\\
&\le 2K\|g\|_XK^{n-1} I_\pm^{n-1} (x)\\
&= 2K\|g\|_X \frac{\left(K(1\pm x)\right)^{n-1}}{(n-1)!} .
\end{align*}
where we used \eqref{eq:Iphiofind} for the first equality and $I_\pm \Ipm =\Ipm  I_\pm$ in the second inequality.   As $\Ipm  g\in Y$ if $g\in Y$, by the above
\[
\|\Epmq  g\|_Y\le \|g\|_Y\left(1+2K\LTp e^{2K\LTp} \right),
\]
  and for any $g\in X$ it holds that $\Epmq  \Ipm  g = \Ipm  \Epmq   g$ and $\Epmq I_\pm g=I_\pm \Epmq   g$ using Lemma \ref{lem:diphi} and the continuity of $I_\pm:X\to X$. Then, because of \eqref{eq:ker}, \eqref{eq:inv} and \eqref{eq:phi-1}, the identities in \eqref{eq:MLrel} hold. The last statement is a consequence of Lemma \ref{lem:diphi}, which proves that $ \Ir \Erq g\in C^1[-1,1]$.





\end{proof}

\begin{theorem}\label{thm:9}  Assume \ref{H0}. Then, for each operator $(\Gen,\mathrm{LR})$ in Table \ref{explicitProcesses} and $\LTp >0$, the operator $(\LTp -\Gen):\D(\Gen,\mathrm{LR})\to X$ is   surjective.
\end{theorem}
\begin{proof}
 For each $g\in X$, we construct a function $\varphi\in \D(\Gen,\mathrm{LR})$ such that $(\Gen-\LTp)\varphi=g$ using the properties of the   L\'evy operators in Lemma \ref{lem:gML}.  We will use without mention the obvious identity $\Epmq f = f+\LTp\Epmq \Ipm f$. For $\D(\Cl ,\mathrm{DD})$, let $d=\Elq  \Il  g(1)/\Elq  k_0^+(1)$ and choose
\[
\varphi= \Elq \Il  g - d \Elq  k_0^+=\Il [\Elq (g-\LTp d k_0^+)]-d k_0^+.
\]
Then    $\Cl  \varphi= g+\LTp\varphi\in X$ by a direct computation using \eqref{eq:MLrel}, and $\varphi\in \D(\Cl ,\mathrm{DD}) $ by the domain representation of   Table \ref{tab:ndo},  given that 
\[
d= \frac{\Il [\Elq (g-\LTp d k_0^+)](1)}{k_0^+(1)}.
\] 
For $\D(\Cl , \mathrm{DN})$, let $d=\Elq  I_+ g(1)/[1+\LTp\Elq  k_1^+(1)]$ and choose
\[
\varphi= \Elq  \Il  g -d \Elq  k_0^+=\Il  [\Elq( g - \LTp d k_0^+)]-dk_0^+.
\]
 Then $\varphi\in X$, $\Cl  \varphi= g+\LTp\varphi\in X$ by a direct computation using \eqref{eq:MLrel}, and $\varphi\in \D(\Cl ,\mathrm{DN})$ by Table \ref{tab:ndo}, given that
\begin{align*}
I_+[\Elq  (g - \LTp dk_0^+)](1)
&=I_+\Elq  g(1) -\frac{\LTp\Elq  I_+ g(1)}{1+\LTp\Elq  k_1^+(1)}\Elq   k_1^+(1)=d.
\end{align*}
For $\D(\RLl , \mathrm{N^*N})$,  let $d=I_+\Elq g(1)/(\LTp I_+\Elq k^+_{-1}(1))$ and choose
\[
\varphi=\Elq \Il  g - d\Elq  k_{-1}^+=\Il  \Elq [ g -0 k_1^+- \LTp d k_{-1}^+]-0 k_1^+-dk_{-1}^+.
\]
Then $\varphi\in X$, $\RLl  \varphi= g+\LTp\varphi\in X$ by a direct computation using \eqref{eq:MLrel}, and $\varphi\in \D(\RLl ,\mathrm{N^*N})$ by Table \ref{tab:ndo}, given that
\[
 0 = I_+\Elq g(1) -\LTp d  I_+\Elq k_{-1}^+(1)  =  I_+\Elq [ g -0 k_1^+- \LTp d k_{-1}^+](1).
\]

 The remaining cases are similar and the respective functions $\varphi$ are given in Table \ref{tab:thm9}.
\end{proof}

\begin{table}[h]
\centering
\vline
\begin{tabular}{ l l  l|}

	\hline
		 $\D (A, \mathrm{LR})$  &  $\omega$ &  $d $   \\
	\hline
	

 
  $\D (\Cpm , \mathrm{DD})$ & $\omega=     k_0^\pm$ & $d =\frac{\Epmq  \Ipm  g(\pm1)}{\Epmq  k_0^\pm (\pm1)}$   \\
  \hline
  $\D (\Cl , \mathrm{DN})$ & $\omega=     k_0^\pm$ & $d =\frac{\Epmq  I_\pm g(\pm1)}{1+\LTp\Epmq  I_\pm  k_0^\pm(\pm1)}$   \\
	  $\D (\Cr , \mathrm{ND})$ &  &    \\
  \hline
 $\D (\Cl , \mathrm{ND})$ & $\omega=  1 $ & $d =\frac{\Epmq  \Ipm  g(\pm1)}{\Epmq 1 (\pm1)}$   \\
	$\D (\Cr , \mathrm{DN})$ & &  \\
  \hline
  $\D (\Cpm , \mathrm{NN})$ & $\omega=     k_1^\pm$ & $d =\frac{ \Epmq  I_\pm g(\pm1)}{\LTp(2+\LTp\Epmq   I_\pm k_1^\pm (\pm 1))}$   \\
  \hline
  $\D (\RLl , \mathrm{N^*D})$ & $\omega=     k_{-1}^+$ & $d =\frac{\Elq \Il g(1)}{\Elq k_{-1}^+(1)}$ \\
  \hline
  $\D (\Cr , \mathrm{N^*D})$ &$\omega=   k_0^-  $ & $d =\frac{\frac{\dd}{\dd x}\Erq \Ir g(-1)}{\frac{\dd}{\dd x}\Erq k_{0}^-(-1)}$ \\
  \hline
  $\D (\Cr , \mathrm{N^*N})$  &$\omega=   k_1^-  $ & $d =\frac{\frac{\dd}{\dd x}\Erq \Ir g(-1)}{\LTp \frac{\dd}{\dd x}\Erq k_{1}^-(-1)}$ \\
  \hline
	  $\D (\RLl , \mathrm{N^*N})$  &$\omega=   k_{-1}^+  $ & $d =\frac{I_+\Elq g(1)}{\LTp I_+\Elq k^+_{-1}(1)}$
	  \\
  \hline
\end{tabular}
\caption{\label{tab:thm9} For the proof of Theorem \ref{thm:9}, choices of the resolvent $(\LTp-G)^{-1}g =-\varphi$ for $g\in X$: $\varphi=\Epmq  [\Ipm  g-d \omega]$ if ${\rm LR}\in\{{\rm DR},{\rm LD}\}$ or  for $\D (\RLl , \mathrm{N^* N})$; $\varphi=\Epmq  [\Ipm  g-\LTp d \omega]-d$ for $\D (\Cr , \mathrm{N^*N})$ or $\D (\Cpm , \mathrm{NN})$.}

\end{table} 
\subsection{Duality}
In this section we establish the duality relation between the backward and  forward generators of Table \ref{explicitProcesses}. At the end we discuss how to use   known resolvent measures to obtain a proof of Table \ref{explicitProcesses} for the processes that do not feature a left fast-forwarding boundary condition, i.e. $\Y{DD}$, $\Y{N^*D}$, $\Y{D N}$ and $\Y{N^*N}$.
 
Recall that for a linear operator $A:\mathcal D(A) \subset X\to X$ and a subspace $\SuS\subset X$ where $X$ is a Banach space, we denote by $A |_\SuS$ the part of $A$ in $\SuS$, i.e. $\{(f,Af)\in \mathcal D(A)\times X:f,    Af\in \SuS \}$. For two operators $A:\mathcal D(A) \subset X\to X$ and $\widetilde A:\mathcal D(\widetilde A) \subset X\to X$, the inclusion $A\subset \widetilde A$ means that the graph of $A$ is a subset of the graph of $\widetilde A$. We recall that the dual spaces of $C_0(\Omega)$ and $L^1[-1,1]$ are respectively $C_0(\Omega)^*$, the space of finite Borel signed measures on $\Omega$ (with total variation norm) \cite[Theorems 2.18 and 6.19]{MR924157}, 
 and $L^1[-1,1]^*=L^\infty[-1,1]$. We refer to \cite[Chapter 1.10]{MR710486} for standard definitions of duals of densely defined operators and duals of strongly continuous contraction semigroups.
\begin{proposition}
\label{adjointfractionalderivatives} Assume \ref{H0} and let $(\Gen^+, \BC)$ on $L^1[-1,1]$ and $(\Gen^-, \BC)$ on $C_0(\Omega)$ be as in Table \ref{explicitProcesses}, and denote by $(\Gen^+, \BC)^*$ and $(\Gen^-, \BC)^*$ their dual operators, respectively. Then
\[
(\Gen^-, \BC) = (\Gen^+, \BC)^* \Big|_{C_0(\Omega)}\quad\text{and}\quad(\Gen^+, \BC) = (\Gen^-, \BC)^* \Big|_{L^1[-1,1]}.
\]
\end{proposition}

\begin{proof} 
First observe that all dual operators a well-defined by Theorem \ref{thm:5}.
Using the explicit representation in Table \ref{tab:ndo} of the respective domains we first show that operators satisfy $(\Gen^\pm,\BC)\subset (\Gen^\mp,\BC)^*$; i.e. we show that for $\omega_+\in \D(\Gen^+,\BC)$ and $\omega_-\in\D(\Gen^-,\BC)$,
$$\Delta:=\int_{-1}^1 \omega_-(x)\Gen^+\omega_+(x)\,\dd x-\int_{-1}^1\omega_+(x)\Gen^-\omega_-(x)\,\dd x=0.$$ 
In order to simplify the readability of the next computation we  abuse  notation by writing $k^\pm_{i}(\cdot\mp 1)=k^\pm_{i}(\cdot)$ when $k^\pm_{i}$ is being convolved, so that, for example, $ \Ipm  f=k_{0}^\pm\star f$ (cf. Definition \ref{def:ndo}). 
Let $$\omega_-=\Ir  g_-+a_-k^-_{1}+b_-k^-_{0}+c_-\in \D(\Gen^-,\BC)$$ and
$$\omega_+=I^\alpha_+ g_++a_+k^+_{1}+b_+k^+_{0}+c_++d_+k^+_{-1}\in \D(\Gen^+,\BC)$$
for one of the six cases of Table \ref{tab:ndo}.
Using $\int_{-1}^1 f \Il  g= \int_{-1}^1g\Ir f$ and $\int_{-1}^1k_{1}^+ =\int_{-1}^1k_{1}^-$ to cancel four terms in the second equality, and $\int_{-1}^1 k_{i}^\pm\star f=k_{i+1}^\pm\star f(\pm 1)$ and $\int_{-1}^1 k_{i}^\pm f=k_{i}^\mp\star f(\mp 1)$ in the third equality, we obtain (omitting the dependence on $x$ in the second equality)
\begin{equation*}\begin{split}\Delta=&\int_{-1}^1\omega_-(x)\left(g_+(x)+a_+\right)-\omega_+(x)\left(g_-(x)+a_-\right)\,\dd x\\
=&\int_{-1}^1 \left(\Ir g_-+b_-k^-_{0}+c_-\right)a_+-\left(\Il  g_++b_+k^+_{0}+c_++d_+k^+_{-1}\right)a_-\,\dd x\\
&+\int_{-1}^1 \left(a_-k^-_{1}+b_-k^-_{0}+c_-\right)g_+-\left(a_+k^+_{1}+b_+k^+_{0}+c_++d_+k^+_{-1}\right)g_-\,\dd x\\
=&\,a_+k^-_{1}\star g_-(-1)-a_-k^+_{1}\star g_+(1)+b_-a_+k^-_{1}(-1)-b_+a_-k^+_{1}(1)+c_-a_+2\\
&-a_-c_+2-a_-d_+k^+_{0}(1)+a_-k^+_{1}\star g_+(1)+b_-k^+_{0}\star g_+(1)+c_-I_+g_+(1)\\
&-a_+k^-_{1}\star g_-(-1)-b_+k^-_{0}\star g_-(-1)-c_+I_-g_-(-1)-d_+k^-_{-1}\star g_-(-1)\\
=&\,b_-\left(a_+k^-_{1}(-1)+k^+_{0}\star g_+(1)\right)-b_+\left(a_-k^+_{1}(1) +k^-_{0}\star g_-(-1)    \right)\\
&+c_-\left(a_+2+I_+g_+(1)\right)-c_+\left(a_-2+I_-g_-(-1) \right)-d_+\left(a_-k^+_{0}(1)+k^-_{-1}\star g_-(-1)\right).
\end{split}
\end{equation*}
Then for each of the six cases of Table \ref{tab:ndo}  one can verify that $\Delta=0$. For example, for $(\Gen^\pm,{\rm ND})$, we have $a_-,c_-,a_+,b_+,d_+ =0$, $b_-=-I_-g_-(-1)$ and $c_+=-\Il g_+(1)$, so that
\[
\Delta= b_-k^+_{0}\star g_+(1)-c_+I_-g_-(-1)  = -I_-g_-(-1)\, k^+_{0}\star g_+(1)+\Il g_+(1)I_-g_-(-1)  = 0,
\]
or for $(\Gen^\pm,{\rm N^*N})$, we have $b_- ,b_+,c_+ =0$, $d_+,c_-\in\mathbb R$, $a_-=[\Ir g_-]'(-1)/k^-_0(-1)$ and  $a_+=-I_+g_+(1)/2$, so that (using Lemma \ref{lem:diphi})
\[
\Delta=c_-\left(\frac{-I_+g_+(1)}{2}2+I_+g_+(1)\right)-d_+\left(\frac{[\Ir g_-]'(-1)}{k^-_0(-1)}k^+_{0}(1)+k^-_{-1}\star g_-(-1)\right) = 0.
\]
Thus we proved that $(\Gen^-, \BC) \subset (\Gen^+, \BC)^*|_{C_0(\Omega)} $ and $ (\Gen^+, \BC) \subset (\Gen^-, \BC)^*|_{L^1[-1,1]}$. 

As a consequence, 
$$I - (\Gen^-, \BC) \subset \big( I- (\Gen^+, \BC)^*\big)\Big|_{C_0(\Omega)}$$ and  $$ I- (\Gen^+, \BC) \subset \big(I- (\Gen^-, \BC)^*\big)\Big|_{L^1[-1,1]},$$ 
where the identity operator on the respective spaces is denoted by $I$. By Theorem \ref{thm:9}, 
 $I - (\Gen^-, \BC)$ is surjective which implies that $I - (\Gen^+, \BC)^*$ is injective, and  thus $(I - (\Gen^+, \BC)^*)|_{C_0(\Omega)}$ is also injective. This yields 
 $$
 I - (\Gen^-, \BC) = \big( I - (\Gen^+, \BC)^*\big)\Big|_{C_0(\Omega)},
 $$ since if operators $T, S$ are such that $T \subset S$, $T$ is surjective and $S$ is injective, then $T=S$. Hence, $(\Gen^-, \BC) = (\Gen^+, \BC)^* |_{C_0(\Omega)}$. A similar argument holds for the pair $(\Gen^+, \BC)$ and $(\Gen^-, \BC)^*|_{L^1[-1,1]}$ . 
\qquad \end{proof}

\begin{corollary}  Assume \ref{H0}. Then, for each operator $(\Gen,\mathrm{LR})$ in Table \ref{explicitProcesses} and $\LTp >0$, the operator $(\LTp-(\Gen,\BC)):\D(\Gen,\BC)\to X$ is a bijection.
\end{corollary}
\begin{proof}
By Theorem  \ref{thm:9} $(\LTp-(\Gen^-,\BC))$ is surjective, therefore its dual $(\LTp-(\Gen^-,\BC))^*$ is injective, and so does $(\LTp-(\Gen^-,\BC))^*|_{L^1[-1,1]}$. By Proposition \ref{adjointfractionalderivatives}  $(\LTp-(\Gen^-,\BC))^*|_{L^1[-1,1]}=(\LTp-(\Gen^+,\BC))$, and by Theorem \ref{thm:9} we conclude that $(\LTp-(\Gen^+,\BC))$ is a bijection. The same argument works for $(\LTp-(\Gen^-,\BC))$. 
\end{proof}

\begin{corollary}\label{cor:duality} Assume \ref{H0} and recall  Table \ref{explicitProcesses}. Then the operator $(\Gen^-, \BC)  $ generates a positive strongly continuous contraction semigroup  on $C_0(\Omega)$ if and only if $(\Gen^+, \BC)  $ generates a positive strongly continuous contraction semigroup on $L^1[-1,1]$.  
\end{corollary}
\begin{proof} For the ``only if'' direction recall that $L^1[-1,1]$ is a closed subspace of $C_0(\Omega)^\ast$ and denote by $P$ the strongly continuous semigroup generated by $(\Gen^-, \BC)$ and by $P^*$ its dual semigroup. Also denote by  $(\Gen^-,{\rm LR})^*$ the dual operator of the densely defined operator $(\Gen^-,{\rm LR})$.
 By  \cite[Theorem  1.10.4]{MR710486}, $P^\ast|_{\SuS^\ast}$ is a strongly continuous semigroup on $\SuS^\ast$, where $\SuS^\ast$ is the closure  in $C_0(\Omega)^\ast$ of the domain of $(\Gen^-,{\rm LR})^*$, and the generator of $P^\ast|_{\SuS^\ast}$ is $(\Gen^-,LR)^\ast|_{\SuS^*}$. Moreover, $P^\ast|_{\SuS^\ast}$ is a contraction semigroup by \cite[Lemma 1.10.1]{MR710486}. 
 By Proposition \ref{adjointfractionalderivatives}  $(\Gen^-,{\rm LR})^*|_{L^1[-1,1]}=(\Gen^+,LR)$ so that Theorem \ref{thm:5} implies    $L^1[-1,1]\subset \SuS^\ast$. Then  $(\LTp -(\Gen^+,LR)):\mathcal{D}(\Gen^+,LR)\to  L^1[-1,1]$ is a bijection for any $\LTp>0$  
 so that   
 $$
 \int_0^\infty e^{-\LTp t}P^*_t (\cdot)\,\dd t: L^1[-1,1]\to \mathcal{D}(\Gen^+,LR)\subset  L^1[-1,1], 
 $$
 and therefore $P^* L^1[-1,1]\subset  L^1[-1,1]$ by \cite[Theorem 1.8.3]{MR710486}, so  we obtained that $P^* |_{L^1[-1,1]}$ is a strongly continuous contraction semigroup on $L^1[-1,1]$ with generator  $(\Gen^+,LR)$.  Positivity holds because for all  non-negative $f\in C_0(\Omega), \, g\in L^1[-1,1]$ and $ t>0$
 \[
 0\le \int_\Omega P_t f \,g\,\dd x = \int_\Omega f P_t^* g\,\dd x ,
 \]
 so that for any $0\le f\in L^\infty[-1,1]$, by the theory of mollifiers, we can find a sequence $ 0\le f_n \in  C_0(\Omega)$ converging to $f$ a.e. and  with $\sup_n \|f_n\|_{L^\infty[-1,1]}<\infty$ and by Dominated Convergence Theorem we obtain $0\le  \int_\Omega f P_t^* g\,\dd x$. We omit the ``if'' direction as it follows by the same strategy but using  that  $C_0(\Omega)$ is a closed subspace of $L^1[-1,1]^\ast$ and a slightly simpler positivity argument.  
 
\end{proof}

\begin{remark}\label{rmk:knowntwosided}
We recall that the resolvent measures of the processes $\Y{DD}$, $\Y{N^*D}$, $\Y{D N}$ and $\Y{N^*N}$ are known, and they can be computed from \cite[Theorems 8.7, 8.11.i, 8.11.ii]{MR2250061} and \cite[Theorem 1]{MR1995924}, respectively (recalling that in law $\Y{D N}$ equals  $\Y{D N^*}$ and $\Y{N^*N}$ equals the two-sided reflection). These known results combined with   Remark \ref{rem:Zkyp} and Table \ref{tab:thm9} confirm that in Theorem \ref{thm:9} we constructed the resolvent of these processes and their inverse, thus  $( \Cr,{\rm DD})$, $ (\Cr,{\rm N^*D})$, $ (\Cr,{\rm DN})$  and $ (\Cr,{\rm N^*N})$ are       the   Feller generators of the processes $\Y{DD}$, $\Y{N^*D}$, $\Y{D N}$ and $\Y{N^*N}$, respectively. Then Corollary \ref{cor:duality} yields the respective forward generators  $ (\Cl,{\rm DD})$, $ (\RLl,{\rm N^*D})$, $ (\Cl,{\rm DN})$  and $ (\RLl,{\rm N^*N})$, and the proof of   Table \ref{explicitProcesses} for these four cases is complete. For example, in the case $\Y{N^*D}$, $(\LTp-(\Cr,{\rm N^*D}))^{-1} g$ from Table \ref{tab:thm9} rewrites as 
\begin{align*}
&\,\frac{[ \Erq \Ir g]'(-1)}{[  \Erq k_0^-]'(-1)}W^{(\LTp)}(1-x)-W^{(\LTp)}(-\cdot)\star g(x)\\
&\,=\frac{W^{(\LTp)}(1-x)}{[W^{(\LTp)}]'(2)}\int_{-1}^1 [W^{(\LTp)}]'(y+1)g(y) \,\dd y-\int_{x}^1 W^{(\LTp)}(y-x)g(y) \,\dd y,
\end{align*} 
where we used $ \Erq \Ir g(x) = W^{(\LTp)}(-\cdot)\star g(x)$ and  $\Erq k_0^-(x) = W^{(\LTp)}(1-x)$,  which can be proved using $f\star^n (-\cdot)=f(-\cdot)\star^n$. Hence we identified the resolvent measure of \cite[Theorem 8.11.i]{MR2250061} adapted to the interval $[-1,1]$ (note that by the discussion after this theorem there is no atom at the reflecting boundary).
\end{remark}


\section{Gr\"unwald type approximations and interpolation} 
This section is first dedicated to constructing our approximation scheme to nonlocal derivatives (following \cite{MR923707,MR1269502}), studying several related algebraic and Post–Widder convergence properties and identifying the corresponding compound Poisson process $Y^h$. We then embed $Y^h$ into six discrete valued Feller process on the bounded domain $[-1,1]$ via the interpolation technique of \cite{MR3720847} (with the study of the path properties appearing later in Part II). These six processes will approximate the six processes of Table \ref{explicitProcesses}. In the last (technical) subsection, we construct suitable interpolated approximations of the scale functions $k_i$, $i=1,0,-1$. This is necessary to prove the convergence in Theorem \ref{thm:TK}-(ii), as these scale functions appear in most domains of the limit generators in Table \ref{tab:ndo} and convergence fails without ``interpolating''.

\subsection{Gr\"unwald type coefficients} 
In this section we study the coefficients that describe: (i)   a natural finite   difference scheme   approximating our nonlocal operators  (Section \ref{sec:convconvq}) and scale functions  (Section \ref{sec:approx_vartheta}); (ii)     finite difference schemes/discrete Markov processes approximating the spectrally positive process with L\'evy symbol $\sym$ and its restrictions to $[-1,1]$  (Section \ref{sec:interpmat}). It is important to remark that these coefficients depend implicitly on the approximating parameter $h$, contrarily to the Gr\"unwald case \cite{MR3720847}. However, we   tame the dependence on $h$  relying on Post–Widder convergence properties that follow from \ref{H0}. In four of the theorems of Section  \ref{sec:case},   we rely on the stronger assumption \ref{H1} to obtain the required convergence rate. These convergence properties are proved in Lemma \ref{lem:PWthms}.    
\begin{definition} Under assumption \ref{H0}, we define the \textit{Gr\"unwald type  coefficients }for $j\ge 0$ and $h>0$
\begin{equation}\label{eq:GCdef}
\begin{split}\Gru_{j,h}&=\frac{(-1)^j}{j!}\frac{1}{h^j}\sym^{(j)}(1/h),
\\ 
\Gruone_{j,h}&=\frac{(-1)^j}{j!}\frac{1}{h^j}\left(\frac{\sym}{\xi}\right)^{(j)}(1/h),\\
\calG^{k_{1}}_{j,h}&=\frac{(-1)^j}{j!}\frac{1}{h^j}\left(\frac{1}{\xi\sym}\right)^{(j)}(1/h),
\\
\calG^{k_{0}}_{j,h}&=\frac{(-1)^j}{j!}\frac{1}{h^j}\left(\frac{1}{\sym}\right)^{(j)}(1/h),
\\
\calG^{k_{-1}}_{j,h}&=\frac{(-1)^j}{j!}\frac{1}{h^j}\left(\frac{\xi}{\sym}\right)^{(j)}(1/h),\\
\calG^{k_{-2}}_{j,h}&=\frac{(-1)^j}{j!}\frac{1}{h^j}\left(\frac{\xi^2}{\sym}\right)^{(j)}(1/h),
\end{split}
\end{equation}
and  $\Gru_{-j,h},\Gruone_{-j,h},\calG_{-j,h}^{k_{i}}=0$ for all $j\in\mathbb N, \,i\in\{1,0,-1,-2\}$.  We will ease notation by writing $\Gru_{j,h}=\Gru_j,\,\Gruone_{j,h}=\Gruone_j$ and $\calG^{k_{i}}_{j,h}=\calG_j^{k_{i}},\,i\in\{-2,-1,0,1\},$ when the value of $h$ is clear.
\end{definition} 
\begin{remark}$\,$
\begin{enumerate}[(i)]
\item As $\sym$ is analytic on $\mathbb C\cap\{\Re \xi>0\}$, for all  $|\xi|<1$ is holds that
\begin{equation}\label{eq:hatomegas}
\begin{split}
\sym\left(\frac{1-\xi}h\right)=\sum_{j=0}^\infty \Gru_{j,h} \xi^j,\quad \frac{\sym\left(\frac{1-\xi}h\right)}{\frac {1-\xi}h}=\sum_{j=0}^\infty \Gruone_{j,h} \xi^j, \\
\frac{\frac h{1-\xi}}{\sym\left(\frac{1-\xi}h\right)}=\sum_{j=0}^\infty \calG^{k_{1}}_{j,h} \xi^j,\quad  
\frac{\frac{1-\xi}h}{\sym\left(\frac{1-\xi}h\right)}=\sum_{j=0}^\infty \calG^{k_{-1}}_{j,h} \xi^j,\\
 \frac{1}{\sym\left(\frac{1-\xi}h\right)}=\sum_{j=0}^\infty \calG^{k_{0}}_{j,h} \xi^j,\quad  \frac{\left(\frac{1-\xi}{h}\right)^2}{\sym\left(\frac{1-\xi}h\right)}=\sum_{j=0}^\infty \calG^{k_{-2}}_{j,h} \xi^j.
\end{split}
\end{equation}

\item To see why the coefficients $\{\Gru_{j,h}:j\in\mathbb N_0,h>0\}$ are natural candidates, observe that $e^{\xi h}\sym((1-e^{-\xi h})/h)\to \sym(\xi)$ as $h\to0$, where $e^{\xi h}\sym((1-e^{-\xi h})/h)$ is the L\'evy symbol of the finite difference scheme $\Crh$ defined in Section \ref{sec:convconvq}  and $\sym(\xi)$ the L\'evy symbol of $\Mr$. Indeed $e^{\xi h}\sym((1-e^{-\xi h})/h)= \int_\mathbb R (e^{-\xi  y}-1)\phi_h(\dd y)$ for the L\'evy measure $\phi_h(\dd y)=\sum_{j\neq 1} \Gru_{j,h}\delta_{(j-1)h}(\dd y)$ and $\Crh g(x)= \int_\mathbb R (g(x+y)-g(x))\phi_h(\dd y)$, where we used \eqref{eq:grun_phi} and $\delta_{x}(\dd y)$ denotes the Dirac delta measure at $x$.

  \item In the fractional case $\phi(\dd y)=y^{-1-\alpha}/\Gamma(-\alpha)\dd y$, $\alpha\in(1,2)$, $\sym(\xi)=\xi^\alpha$, and the coefficients defined in \eqref{eq:GCdef} reduce to the standard Gr\"unwald coefficients \cite{MR3720847}. Indeed, for each $j\ge0$, $\Gru_{j,h}=h^{-\alpha}(-1)^j\binom\alpha j$, $\Gruone_{j,h}=h^{1-\alpha}(-1)^j\binom{\alpha-1} j$ and $\calG^{k_{i}}_{j,h}=h^{\alpha+i}(-1)^j\binom{-\alpha -i}j$ for $i\in\{1,0,-1,-2\}$.
  \end{enumerate}
\end{remark}

\begin{lemma}\label{lem:grun_phi}  Assume \ref{H0}. Then, for each $h>0$, the coefficients $\{\Gru_{j,h}:j\in\mathbb N_0\}$ satisfy 
\begin{equation}
\text{$-\Gru_{1,h},\Gru_{0,h},\Gru_{j,h}>0$, for $j\ge 2$, and $-\Gru_{1,h}=\sum_{j\neq 1}\Gru_{j,h}$,}
\label{eq:grun_phi}
\end{equation}
and the coefficients $\{\Gruone_{j,h}:j\in\mathbb N_0\}$ satisfy 
\begin{equation}
\text{$\Gruone_{0,h},-\Gruone_{j,h}>0$, for $j\ge 1$, and $-\Gruone_{0,h}=\sum_{j\ge 1}\Gruone_{j,h}$.}
\label{eq:grun_phi-1}
\end{equation}
Moreover,  for all $j\ge0 $, $i\in\{-1,0,1\}$ and $h>0$  we have the identities 
\begin{equation}\label{eq:phi+1tophi}
\Gruone_{j,h}= h\sum_{m=0}^j \Gru_{m,h}\quad \text{and}\quad \calG^{k_i}_{j,h}= h\sum_{m=0}^j \calG^{k_{i-1}}_{m,h},
\end{equation}

and the inequalities 
 \begin{equation}
\calG^{k_i}_{j,h}>0.
\label{eq:grun_ki}
\end{equation}

\end{lemma}
\begin{proof} It is clear that  $\Gru_{0,h}>0$, as $\sym(x)>0$ for all $x\in(0,\infty)$, and the remaining inequalities in \eqref{eq:grun_phi} follow from  the definitions in \eqref{eq:GCdef} combined with \ref{H0} and
\[
\sym^{(1)}(\xi)=\int_0^\infty (1-e^{-\xi y})y\,\phi({\rm d}y)
\]
being Bernstein, by  \cite[Theorem 3.2]{MR2978140}. Then by the Monotone Convergence Theorem and \eqref{eq:convhomega},  we let $\xi\uparrow 1$ with $|\xi|<1$ in the first identity in \eqref{eq:hatomegas} and we proved the identity in \eqref{eq:grun_phi}. The identities in  \eqref{eq:phi+1tophi} are an immediate consequence of the identities in \eqref{eq:hatomegas} and Cauhcy's product rule. Combining \eqref{eq:phi+1tophi} and \eqref{eq:grun_phi} we obtain the inequalities in \eqref{eq:grun_phi-1}, and recalling that
\[ 
\frac{\sym(x)}x = \int_0^\infty (1-e^{-x y})\phi(y,\infty)\,\dd y\to 0\quad\text{as }x\to0,
\]
we can apply the Monotone Converge Theorem for  $\xi\uparrow 1$ with $|\xi|<1$ in the second identity of the first line of \eqref{eq:hatomegas}, and we proved the identity in \eqref{eq:grun_phi-1}. The inequalities in \eqref{eq:grun_ki} follow from Lemma \ref{thm:ker}, which proves that $k_i^+\ge 0$ and $\int_{0}^\infty e^{-x y}k_{i}^+(y-1)\,{\rm d}y=x^{-i}/\sym(x) >0$ for $x\in(0,\infty)$, so that
\[
\calG^{k_i}_{j,h}= \frac1{h^jj!}\int_{0}^\infty e^{-\frac1h y}y^jk_{i}^+(y-1)\,{\rm d}y>0.
\]
\end{proof}


\begin{lemma}\label{lem:convid} Assume \ref{H0}. Then the following discrete convolution identities hold
\begin{equation}
\sum_{j=0}^m \calG_{m-j}^{k_{-1}} \Gru_j =
\left\{\begin{split}&0,  &m\ge 2,\\
&1/h,  & m=0,\\
&-1/h,  & m=1,\\
\end{split}\right.
\label{eq:dconvk-1}
\end{equation} 
\begin{equation*}
\sum_{j=0}^m  \calG_{m-j}^{k_{0}} \Gru_j  = \sum_{j=0}^m  \calG_{m-j}^{k_{-1}} \Gruone_j = 
\left\{\begin{split}&0,  &m\ge 1,\\
&1,  & m=0,
\end{split}\right.
\end{equation*}

\begin{equation}
\sum_{j=0}^m  \calG_{m-j}^{k_{1}} \Gru_j =\sum_{j=0}^m  \calG_{m-j}^{k_{0}} \Gruone_j=
h, \quad m\ge 0,
\label{eq:dconvk1}
\end{equation}

\begin{equation*}
\sum_{j=0}^m  \calG_{m-j}^{k_{1}} \Gruone_j =
(m+1)h^2, \quad m\ge 0.
\end{equation*}
\end{lemma}
\begin{proof}
We know that for all $|\xi|<1$
\[
\frac{1}{h}-\frac{1}{h}\xi =\sym((1-\xi)/h)\frac{(1-\xi)/h}{\sym((1-\xi)/h)}=\sum_{m=0}^\infty\left( \sum_{j=0}^m \calG_{m-j}^{k_{-1}} \Gru_j\right)\xi^m,
\]
using Cauchy's product rule, and \eqref{eq:dconvk-1} is proved. The remaining identities follow by the same argument and we omit the proof.
\end{proof}

\begin{lemma}[Post–Widder]\label{lem:PWthms} Recall the definitions of the  Gr\"unwald type coefficients in \eqref{eq:GCdef}. 
If \ref{H0} holds, then, for any $\delta>0$, $i\in\{-1,0,1\}$ and $j\in\{0,1,2,...\}$, as $m\to\infty$
$$ 
\frac{m}{x+1}\calG_{m-j,\frac{x+1}m}^{k_{i}}\to k_{i}^+(x),\quad\text{ in }\quad C[-1+\delta,1]. $$ 
If \ref{H1}   holds, then above statement holds also for $i=-2$. Concerning pointwise limits, if \ref{H0} holds, then  as $h=2/(n+1)\to 0$  
\begin{align}
\Gru_{n+1,h}&\to 0, \quad \sum_{j=0}^n\Gruone_{j,h}\to \Phi(2),\quad  \sum_{j=0}^n\Gru_{j,h}\to -\frac{\phi((2,\infty))+\phi([2,\infty))}{2},  \label{eq:GClimphi}\\
\frac{\calG^{k_{-1}}_{n+1,h}}{\calG^{k_{1}}_{n+1,h}}&\to \frac{k^+_{-1}(1)}{k_{1}^+(-1)},  \quad\text{and} \quad  \frac{\calG^{k_{-1}}_{n,h}}{\calG^{k_{0}}_{n,h}}\to \frac{k^+_{-1}(1)}{k_{0}^+(-1)}.\label{eq:GClimk-1k1}\\ 
\intertext{If \ref{H1}  holds, then as $h=2/(n+1)\to 0$}
\frac{\calG^{k_{-2}}_{n+1,h}}{\calG_{n,h}^{k_{0}}}&\to \frac{k^+_{-2}(1)}{k^+_{0}(1)} , \quad  \frac{\calG^{k_{-2}}_{n+1,h}}{\calG_{n,h}^{k_{-1}}}\to \frac{k^+_{-2}(1)}{k^+_{-1}(1)}, \quad\text{and} \label{eq:GClimk-2k0}\\
  \frac{\calG^{k_{0}}_{n-1,h}}{h}&=k^+_0(1)+O(h) \label{eq:GClimk-2k0Tricky}.
\end{align}

\end{lemma} 

\begin{proof}   
Recalling Propositions \ref{thm:ker} and \ref{prop:y^2ker2}, the first statement is an immediate consequence of \cite[Theorem VII.5a]{MR0005923}. Indeed note that for any $x\in[-1+\delta,1]$
\begin{align*}
\frac{m}{x+1}\calG_{m-j,\frac {x+1}m}^{k_{i}}&=\frac{1}{(m-j)!}\left(\frac{m}{x+1}\right)^{m-j+1}\int_0^\infty e^{-\frac m{x+1}y}y^m y^{-j}k_{i}^+(y-1)\,\dd y\\
&=\left[(x+1)^j\frac{\prod_{l=0}^{j-1}(m-l) }{m^j}\right]\\
&\quad\times \frac{1}{m!}\left(\frac{m}{x+1}\right)^{m+1}\int_0^\infty e^{-\frac m{x+1} y}y^m y^{-j}k_{i}^+(y-1)\,\dd y,
\end{align*}
 and the first bracket converges uniformly to $(x+1)^j$, while the the second converges uniformly to $(x+1)^{-j}k_{i}^+(x)$ because  $y^{-j}k_i^+(y-1)$ satisfies the four conditions of \cite[Theorem VII.5a]{MR0005923}.  We now prove the pointwise limits. We compute for $n\ge3$
 
\begin{align*}
\Gru_{n+1,h}&=\frac{(-1)^{n+1}}{(n+1)!}\frac{\sym^{(n+1)}(1/h)}{h^{n+1}}\\
&=\frac{1}{(n+1)!}\left(\frac {n+1}{2}\right)^{n+1}\int_{(0,\infty)} e^{-\frac{n+1}{2}z}z^{n+1} \,\phi(\dd z) \\
&=\frac{1}{(n+1)!}\left(\frac {n+1}{2}\right)^{n+1}\int_{(0,\infty)} \int_{(0,z)} \left[e^{-\frac{n+1}{2}y}y^{n+1} \right]'\dd y\,\phi(\dd z) \\
&=\frac{1}{(n+1)!}\left(\frac {n+1}{2}\right)^{n+1}\int_0^\infty  \left[e^{-\frac{n+1}{2}y}y^{n+1}\right]'\phi(y,\infty)\,\dd y \\
&=-\frac{1}{(n+1)!}\left(\frac {n+1}{2}\right)^{n+2}\int_0^\infty  e^{-\frac{n+1}{2}y}y^{n+1}\phi(y,\infty)\,\dd y \\
&\quad+\frac{1}{(n+1)!}\left(\frac {n+1}{2}\right)^{n+2}\int_0^\infty e^{-\frac{n+1}{2}y}y^{n+1}2y^{-1}\phi(y,\infty)\,\dd y \\
&\to -\frac{\phi((2,\infty))+\phi([2,\infty))}{2}+\frac22\frac{\phi((2,\infty))+\phi([2,\infty))}{2}=0,
\end{align*}
as $n\to \infty$, by \cite[Theorem VII.3c.1]{MR0005923} where we used Fubini's Theorem in the fourth equality. 
Using  $\sym/\xi =\int_0^\infty (1-e^{-\xi y})\,\phi(y,\infty)\,\dd y$ and \eqref{eq:phi+1tophi} we can compute
\begin{align*}\nonumber
\sum_{j=0}^n\Gru_j&=\frac{1}{n!}\frac1h\left(\frac{-1}{h}\right)^{n} \left[\sym/\xi\right]^{(n)}(1/h)\\ \nonumber
&=-\frac{1}{n!}\left(\frac{n+1}{2}\right)^{n+1}\int_{0}^\infty e^{-\frac{n+1}{2}y} y^{n}\phi(y,\infty)\,\dd y\\ \nonumber
&=-2\frac{1}{(n+1)!}\left(\frac{n+1}{2}\right)^{n+2}\int_{0}^\infty e^{-\frac{n+1}{2}y} y^{n+1}y^{-1}\phi(y,\infty)\,\dd y\\ 
&\to -\frac{\phi([2,\infty))+\phi((2,\infty))}{2}, 
\end{align*}
as $n\to \infty$ by \cite[Corollary VII.3c.1]{MR0005923}. The proof for the second limit in  \eqref{eq:GClimphi} is just like the one for the third limit and we omit it. By the first part of the current lemma we immediately obtain    \eqref{eq:GClimk-1k1} and   \eqref{eq:GClimk-2k0}. To prove   \eqref{eq:GClimk-2k0Tricky} under \ref{H1}-(ii), we only need to apply \cite[Theorem 4.1 with $\mu=p=1$]{MR923707}. To see this, observe that, by the proof of Proposition \ref{prop:y^2ker2} $-k_{-2}(\cdot-1)$ is completely monotone, thus $k_{0}(\cdot-1)$ is Bernstein and so we can extend it analytically on $\Re \xi>0$ and it is continuous at the origin \cite[Proposition 3.6]{MR2978140}, yielding $k_0(\xi-1)=O(1)$ as $\xi\to 0$. Then by the remark below \cite[condition (1.5)]{MR923707} we know \cite[condition (1.5)]{MR923707} holds. As the Post–Widder approximation corresponds to $\delta(\xi)=1-\xi$ in  \cite{MR923707}, \cite[condition (1.10)]{MR923707} holds for $p=1$. Under assumption \ref{H1}-(i), \eqref{eq:GClimk-2k0Tricky} follows by  Proposition  \ref{prop:H1O(h)} observing that $k_0(x-1)\in C^2(0,\infty)$ by Proposition \ref{prop:y^2ker2} and it is 2-regular, which follows by \eqref{eq:2mon} and \eqref{eq:3mon} (consequences of 2-regularity of $\Phi$).  
\end{proof}
\begin{remark}
In the proof above we used the Post–Widder $O(h)$ convergence proved in the end of   Proposition \ref{prop:H1O(h)}. As far as we know this result is new, and we refer to   \cite{MR923707,MR1158482} and references therein for related results.
\end{remark}

%

\subsection{Convergence of the convolution quadrature}\label{sec:convconvq}
For each $h>0$, we define the approximating convolution quadratures for a suitable $g\in C(\mathbb R)$ 
\begin{align}\label{eq:phiconvq}
\Cpmhzero  g(x)&=\sum_{j=0}^\infty \Gru_{j,h} g(x\mp jh), &x\in\mathbb R,\\ \nonumber
\Conepmhzero g(x)&=\sum_{j=0}^\infty \Gruone_{j,h} g(x\mp jh), &x\in\mathbb R.
\end{align}
 We   define the  shift operator $S_{a}g(x)=g(x+a)$ and for suitable $g\in C(\mathbb R)$ we define
\[
\Cpmh g(x):=S_{\pm h}\Cpmhzero g(x)= \sum_{j=0}^\infty \Gru_{j,h} g(x\mp(j\pm 1)h),\quad x\in\mathbb R.
 \]
The following result is an immediate consequence of \cite[Theorem 3.1]{MR1269502}.
\begin{proposition}\label{thm:L94} Assume \ref{H0}. Let $g\in  C^m[a,b]$ such that $g^{(j)}(a)=0$ for $j=0,1,...,m-1$, where $m\ge 5$. Then there exist constants $C,h_0>0$  such that 
	\begin{equation}\label{eq:L94}
	\left|\Cpmhzero  [g(\pm\cdot)](\pm x)-\Mpm [g(\pm\cdot)](\pm x)\right|\le C h \int_a^b |g^{(m)}(y)|\,\dd y,
	\end{equation}
	 for all $\pm x\in[a,b]$ and $h\in(0,h_0].$
\end{proposition}
\begin{proof} Consider the  canonical   extension of $g$ to $C^m(\mathbb R)$ (Remark \ref{rmk:canon}).
 Define $ \tilde g :=S_{ a}g$. Then  $\tilde g\in C^m[0,\infty)$,  $\tilde g^{(j)}(0)=0$ for $j=0,1,...,m-1$, and $\tilde g=0$ on $(-\infty,0)$. Also, borrowing the notation $K(\partial_t)$ and  $K(\partial_t^h)$   from \cite[(2.2) and (3.1) with $k_m=\Phi$ ]{MR1269502}, with $x=\tilde x+a$,
\begin{align*}
\Ml g(x)&=\frac{\dd^2}{\dd x^2}\int_{a}^{x} \Phi(x-y)g(y)\,\dd y\quad\quad (a<x<\infty)\\
&=\int_{0}^{\tilde x} \Phi(y)\tilde g''(\tilde x-y)\,\dd y\quad\quad (0<\tilde x<\infty)\\
&=K(\partial_t)\tilde g(\tilde x),\\
\intertext{and }
\Clhzero  g(x)&=\sum_{n=0}^\infty \Gru_{n} g(x-nh)\quad\quad (a<x<\infty)\\
&=\sum_{n=0}^\infty \Gru_{n}\tilde g(\tilde x-nh)\quad\quad (0<\tilde x<\infty)\\
&=K(\partial_t^h)\tilde g(\tilde x).
\end{align*}
 Because $|s^2\widehat \Phi(s)|\le M |s|^2$ for all $\Re s\ge1 $, \cite[Theorem 3.1 with $\mu=2,p=1,\delta(\xi)=1-\xi,T=b-a$, $K(s)=s^2\widehat \Phi(s)$]{MR1269502} yields the existence of  $C,h_0>0$ such that
 \[|K(\partial_t^h)\tilde g(\tilde x)-K(\partial_t)\tilde g(\tilde x)|\le C h \int_0^{\tilde x} |\tilde g^{(m)}(y)|\,\dd y,\quad \tilde x\in[0,b-a],\,h\le h_0,\]
 which rewrites as \eqref{eq:L94} for the `$+$' case. The `$-$' case follows by the two identities $\Ml f(x)=\Mr [f(-\cdot)](-x)$ and  $\Clhzero  g( x)=\Crhzero  [g(-\cdot)](- x)$.
 
\end{proof}

\begin{corollary}\label{cor:L94} Assume \ref{H0}. Let $f\in C^m[-1,1]$ with $f^{(j)}(-1)=0$ for $j=0,1,...,m-1$, $m\ge6$. Then 
		\begin{equation*}
	\Cpmh  \Ipm  f(\pm \cdot)\to \Cpm  \Ipm  f(\pm \cdot)=f(\pm \cdot) \quad \text{as }h\to  0, \quad \text{in }C[-1,1].
	\end{equation*}
	\end{corollary}
\begin{proof}
	Consider the  canonical extension of $f$ to $C^m(\mathbb R)$ (Remark \ref{rmk:canon}). Direct differentiation shows that $g:=\Il f$ and $g'$ satisfy the conditions of Proposition \ref{thm:L94} for $a=-1$ and $b=2$. Also note that $ \Clhzero  [g']=[\Clhzero  g]'\in C_0(-1,2]$ and $ \Cl  [g']=[\Cl g]'\in C_0(-1,1]$. Then, by Proposition \ref{thm:L94}, $\Clhzero g\to \Cl  g$ in $C^1[-1,2]$. We can now conclude with the inequalities for all $x\in [-1,1]$
\begin{align*}
	|S_h  \Clhzero g(x)- \Cr  g(x)|&\le |S_h\Clhzero g(x)- \Clhzero  g(x)|+|\Clhzero g(x)- \Cl  g(x)|\\
	&\le h\|\Clhzero g\|_{C^1[-1,2]}+\|\Clhzero g- \Cl  g\|_{C[-1,1]}.
\end{align*}
The `$-$' case follows immediately by $\Cl f(x)=\Cr [f(-\cdot)](-x)$ and  
$$\Clh   \Il  f( x)=\Crh  [\Il  f(-\cdot)](- x)=\Crh  [\Ir  \left[f(-\cdot)]\right](- x).$$
\end{proof}

When taking the limit $\lim_{h\to0} \Conepmhzero \Ipm  g=I_\pm g$ we  will need the exact first order approximation term, which we derive  below.

\begin{lemma}\label{lem:firtorderapprox} 
Assuming \ref{H0}, if $g\in C_0^{m}(-1,1]$ for some $m\ge 6$, then the following identity holds   for every $x\in [-1,1]$ and $h\in(0,1]$
\begin{align*}
\Conepmhzero \Ipm  [g(\pm \cdot)] (x)&=I_\pm [g(\pm \cdot)](x) +h F_\pm[g(\pm \cdot)](x)+O(h^2),
\end{align*}
and $F_+[g]\in C_0(-1,1]$ and  $F_-[g(-\cdot)]\in C_0[-1,1)$,
where \begin{equation}
F_\pm[g(\pm y)](x) =-2^{-1}\left[[y\phi(y,\infty)]^\pm \star k_{-1}^{\pm}(\cdot\mp1)\star g'(\pm \cdot)\right](x),
\end{equation}
defining $[y\phi(y,\infty)]^\pm(x)=\pm x\phi(\pm x,\infty)\mathbf 1_{\{\pm x>0\}}$. 
\end{lemma}
\begin{proof}
See Appendix \ref{app:firtorderapprox}.
\end{proof}

\begin{remark} Note that in the fractional case $y\phi(y,\infty)= -y^{1-\alpha}/\Gamma(1-\alpha)=-(1-\alpha)\Phi(y)$, so that $2 F_+[g ](x)=-(\alpha-1)g(x)$, in agreement with the version for $\alpha-1$ of \cite[Remark 4.5.6]{H14}. Also, note that $\Conepmhzero f\to \pm \Conepm f$,   because $\Conel f(x)=-\Coner [f(-\cdot)](-x)$ but $\Conelhzero f(x)=\Conerhzero [f(-\cdot)](-x)$.
\end{remark}


\subsection{Interpolation matrices}\label{sec:interpmat}
 
Our aim in  this section is two fold. We first show that the convolution quadrature $\Crh$
generates a Feller semigroup
on $C_0(\R)$ for each $h>0$, and we denote the corresponding discrete valued compound Poisson process by $Y^h$.  
Secondly, recalling the interpolation matrices framework from  \cite{MR3720847}, we embed   $Y^h$ into six Feller semigroups on $C_0(\Omega)$ and we identify their dual semigroups on $L^1[-1,1]$. Later in Part II we  identify the path properties of these embedded processes on grid points in terms of pathwise restrictions  of $Y^h$ to $[-1,1]$ (for the same   six pathwise restrictions of $Y$ listed in Table \ref{explicitProcesses}).\\
Divide $\R$ up into grids of length $h$ and consider  the space of two-sided sequences going to zero at infinity, $C_0(\mathbb Z)$. We define the projection  operator
$\Pi:C_0(\R)\to C([0,1);C_0(\mathbb Z))$ via 
$$\left(\Pi f\right)_j(\lambda)=f((j+\lambda)h)$$
where $j\in\mathbb Z$, $f\in C_0(\R)$, and let $\Pi^{-1}$ and $\lambda$ (which takes values in $[0,1)$) respectively be the natural extensions of $\Pi^{-1}_{n+1}$ and $\lambda$   in Definition \ref{definitiongridcoordinatefunctions} from the interval $[-1,1]$ to  $\mathbb R$.  Then  we can write
\begin{equation}\label{eq:Gh}
\Crh f=\Pi^{-1}G_h\Pi f,
\end{equation}
with
\begin{equation*}
G_h=
  \begin{pmatrix} 
     \ddots  & \ddots & \ddots & \ddots  &  \ddots& \ddots & \ddots  \\
     \ddots &  \Gru_1 & \Gru_2 & \ddots & \Gru_{n-1}  & \Gru_{n}  & \ddots   \\
     \ddots & \Gru_0 & \Gru_1 &    \ddots &\Gru_{n-2} & \Gru_{n-1}&  \ddots \\
     \ddots & 0&\Gru_0&\ddots    & \ddots &\ddots & \ddots   \\
    \ddots  & \ddots&\ddots&\ddots&\Gru_1&\Gru_2& \ddots  \\
    \ddots & 0 &\ddots&0&\Gru_0&\Gru_1&  \ddots \\
    \ddots  & \ddots  & \ddots & \ddots  &\ddots & \ddots& \ddots 
  \end{pmatrix}.
\end{equation*}
Note that the entries of $G_h$ are negative on the main diagonal and non-negative everywhere else and that the row sums and column sums add to zero by \eqref{eq:grun_phi}. Then $G_h$ is a transition rate matrix on $\mathbb Z$, and  $\Crh $ is the (bounded) generator of a compound Poisson processes $(Y^h_t)_{t>0}$. In particular, if $Y^h_0=x$, then $Y_t^h\in x+h\mathbb Z$ for all $t\ge 0$.  We call this process the \emph{Gr\"unwald type process}. We state this fact below.

 \begin{proposition}\label{prop:Ghgen}
\label{prop:Ghfree} Assume \ref{H0}. Then $(\Crh   ,C_0(\mathbb R)  )$  is the generator  of a   Feller semigroup on $C_0(\mathbb R)$, for each $h>0$, and the corresponding process $Y^h$ is a  compound Poisson process.
\end{proposition}
Restricting this process to a finite domain is at the heart of this article. Philosophically, boundary conditions should only influence the process at the boundary; i.e., if the process moves across a boundary, it can be restarted somewhere (or killed). We will therefore restrict ourselves to finite state processes with transition rate matrix being the central square of $G_h$, where we can modify the first and last row and column to suit a particular boundary condition. For the six cases of Table \ref{explicitProcesses}; i.e., killing, fast-forwarding, or reflecting at the boundary, we modify the countable state transition matrix $G_h$ such that the resulting finite state process is obtained by killing, fast-forwarding, or reflecting at the respective boundary. \\

Recall that the $(i,j)$-th entry of a (backward) rate matrix represents the rate at which particles move from state $i$ to state $j$. As in \cite{MR3720847}, the combinations of boundary conditions are defined  using the generic matrix
\begin{equation}\label{Glr}
G_{n\times n}^{\mathrm{LR}}=
 \begin{pmatrix} 
       b_1^l & b_2^l & \cdots & b_{n-1}^l & b_n    \\
    \Gru_0 & \Gru_1 &    \cdots &\Gru_{n-2} & b_{n-1}^r \\
     0&\Gru_0&\ddots    & \vdots &\vdots    \\
    \vdots&\ddots&\ddots&\Gru_1&b_2^r \\
    0 &\cdots&0&\Gru_0& b_1^r\\
  \end{pmatrix}.
\end{equation}
with the coefficients $\{b_j^l,b_j^r, b_n: j=1,2,...,n-1\}$  determined for each boundary condition in Table \ref{mainTable}. The intuition for the coefficients corresponding to killing and reflecting boundary conditions is clear and we refer to  \cite{MR3720847} for a discussion. The choice of rates for fast-forwarding at the left boundary is less intuitive and we compute it explicitly in Part II. We now embed these candidate discrete states processes into Feller semigroups on $C_0(\Omega)$, which requires a great deal of care (see Example \ref{exp:killB}).   We perform this embedding by employing the interpolated transition matrices theory developed in \cite{MR3720847}.

\begin{example}\label{exp:killB} The fact that killing (and fast-forwarding) our compound Poisson processes yields a process that is not Feller is the technical difficulty that leads us to implement the idea of interpolated transition matrices.    Here we explain this difficulty and how we resolve it in the simple case of a killed Brownian motion. First recall that we want to identify a   finite difference approximation that   characterises in the limit the generator of a Feller semigroup (Theorem \ref{thm:TK}-(ii)) and its corresponding process (Theorem \ref{thm:TK}-(iii)), and thus obtaining at once a full picture of the boundary conditions and their probabilistic meaning. Recall also that $C^2_0(0,\infty)\cap C_c(0,T]$  is a core for the Feller generator $(\partial^2,C^2_0(0,\infty))$  of the Brownian motion killed upon leaving $(0,\infty)$ (easily deductible from \cite[Example 3.17, p. 88]{T20}). We  denote by $B$ and $B^{{\rm D}}$ the Brownian motion and its killed version, respectively. Let $\partial^2_hf(x)=(f(x-h)-2f(x)+f(x+h))/h^2$ be the second order finite difference scheme converging to $\partial^2$, which generates a continuous compound Poisson (Feller) process $B^h$ that converges to $B$, with convergences in the sense of Theorem \ref{thm:TK}. Denote by $B^{h,{\rm D}}$ the process obtained by killing $B^h$ the first time it leaves $(0,\infty)$. Then $B^{h,{\rm D}}$ still convergences to $B^{{\rm D}}$ for every starting point in the Skorokhod sense (by a careful application of the CMT), but we lose the convergence in the sense of generators Theorem \ref{thm:TK}-(ii). This is because  $B^{h,{\rm D}}$ is not a Feller process. Indeed  killing provokes a discontinuity with respect to initial conditions  at $ 0$. This is because if $B^{h,{\rm D}}_0=0$ the process is instantaneously killed, meanwhile for any $B^{h,{\rm D}}_0>0$ the process   performs half of the first jumps away from 0 (or one could say that first hitting time of $0$ does not vanish as $B^h_0\downarrow 0$). This translates to the pointwise generator of $B^{h,{\rm D}}$, say $\partial^{2,{\rm D}}_h$, not preserving the set $C_0(0,\infty)$. Indeed,   $\partial^{2,{\rm D}}_h f(x)=\partial^2_h f(x)$ for $x>h$, but  $\partial^{2,{\rm D}}_hf(x)=(-2f(x)+f(x+h))/h^2$ for all $x\in(0,h]$ implying that
\[ 
\lim_{x\downarrow 0} \partial^2_hf(x)=\frac{f( h)}{h^2},
\] 
and thus $\partial^{2,{\rm D}}_h C_0(0,\infty)\not\subset C_0(0,\infty) $, i.e. the Feller property is lost. The idea introduced in \cite{MR3720847} is to \textit{interpolate the coefficients close to the   boundary }to guarantee  the Feller property, i.e.  define the operator $G^{2,{\rm D}}_h=\partial^2_h$ on $[h,\infty)$, and 
\[
G^{2,{\rm D}}_hf(x)=\frac{ -2f(x)+\lambda(x)f(x+h)}{h^2}, \quad x\in[0,h), \]
for some $\lambda:[0,h)\to [0,1)$ continuous, increasing, with $\lambda(0)=0$ and $\lambda(h-)=1$. Then $G^{2,{\rm D}}_h C_0(0,\infty) \subset C_0(0,\infty)$, $G^{2,{\rm D}}_h$ defines a Feller generator, and on grid points $h\mathbb N_0$ the generated process is   $B^{h,{\rm D}}$. In this case, the convergence of the interpolated generators (in the sense of Theorem \ref{thm:TK}-(ii)) is straightforward, because   $B^{{\rm D}}$ allows $ C^2_0(0,\infty)\cap C_c(0,T]$ as a core. But for our restrictions of L\'evy processes, any core of the generator of $\Y{LR}$ must contain functions that are not smooth at the boundary (see Table \ref{tab:ndo}). This results in the interpolating function $\lambda$ creating a singularity as we let $h\to 0$, and finding suitable test  functions in the approximations is the main difficulty to overcome for the proofs in Section \ref{sec:case}. \\
Similarly,   fast-forwarding a compound Poisson processes does not always yield a Feller processes, as we show in Part II. Intuitively, suppose we fast-forward $B^h$ outside of $(0,\infty)$, then, if $B^h_0= h$ the process can only jump to the right of $h$, but for any $ B^h_0\in(h,2h)$ half of the first jumps will fall  to the left of $h$ (on $B^h_0-h$), provoking a discontinuity with respect to the initial condition, and thus we loose the Feller property.  
\end{example}

We now introduce some technical notation from \cite{MR3720847}.

\begin{definition}
\label{definitiongridcoordinatefunctions}
Let $n \in \mathbb{N}$, then we divide the interval $[-1,1]$ into $n+1$ grids, each of width $h=\frac{2}{n+1}$,  such that the first $n$ grids are half open (on the right) while the $(n+1)^{\mathrm{st}}$ (last) grid is closed. 
\begin{itemize}
\item \emph{Grid Number:} Let $\iota : [-1,1] \to \{1,2,\ldots  , n+1\}$ denote the grid number of $x$,  \[
\iota (x) = \left\lfloor \frac{x+1}{h}\right\rfloor +1
\]
with $\iota(1) = n+1$ where $\left\lfloor \frac{x+1}{h}\right\rfloor$ denotes the largest integer not greater than $\frac{x+1}{h}$. 
\item \emph{Location within Grid:} Let $\lambda: [-1,1]  \to [0,1]$ denote the location within the grid of $x$ given by \[\lambda(x) = \frac{x+1}{h} - \big( \iota(x)-1\big),\] observing that $\lambda(1)=1$.  Also, let $\barlambda(x)=1-\lambda(x)$.
\item \emph{Grid Projection Operator:}
Let $L^1\left([-1,1];\mathbb{R}^{n+1}\right)$ denote the space of vector-valued integrable functions $v: [-1,1] \rightarrow \mathbb{R}^{n+1}$. The projection operator $\Pi_{n+1}:L^1[-1,1]\to L^1\left([0,1];\mathbb{R}^{n+1}\right)$ is defined by \[\left(\Pi_{n+1} f\right)_j(\lambda) =f((\lambda + j-1)h-1), \quad f \in L^1[-1,1],\]
where $\lambda \in [0,1]$ and $j \in \{1,2,\ldots , n+1\}$. 
\item \emph{Grid Embedding Operator:} The grid embedding operator $\Pi^{-1}_{n+1}$ embeds integrable functions defined on the grids onto $L^1[-1,1]$;  $$\left(\Pi^{-1}_{n+1} v\right)(x)=v_{\iota(x)}(\lambda (x)).$$ 
\end{itemize}
\end{definition}

Armed with the transition rate matrices \eqref{Glr} for the finite state processes we build the interpolation matrices. That is, we   define the functions $D^{l,r}$ and $N^{l,r}$ for the respective six cases and we choose them according to  Table \ref{mainTable} to obtain the interpolation matrices

\begin{equation}
\label{interpolationmatrixGBC}
  G^{\BC}_{n+1}(\lambda)=\begin{pmatrix}
    b_1^l& D^l(\lambda) b_2^l&\cdots& \cdots&D^l(\lambda)b_n & 0\\
    N^l(\lambda) \Gru_0  & \lambda \Gru_1 +\barlambda b_1^l &\cdots&\cdots&\lambda b^r_{n-1}+\barlambda b_{n-1}^l & N^r(\lambda) b_n\\
    0 & \Gru_0  & \cdots&\Gru_{n-3} &\vdots & \vdots \\
    \vdots & \ddots &\ddots& \vdots &\vdots&\vdots\\
   \vdots &  &\ddots&\Gru_0  & \lambda b_1^r+\barlambda \Gru_1  & N^r(\lambda) b_2^r\\
    0 &\cdots &\cdots& 0& D^r(\lambda) \Gru_0 & b_1^r
  \end{pmatrix}
\end{equation}
recalling that $\barlambda =1-\lambda$, leading to the transition operators
\begin{equation}\label{G-hBC}G_{-h}^{\BC}f(x)  = \Big[ G_{n+1}^{\BC}(\lambda (x)) (\Pi_{n+1}f)(\lambda (x))\Big]_{\iota (x)}
\end{equation}
on $C_0(\Omega)$ and
\begin{equation}\label{G+hBC}G_{+h}^{\BC}f(x)  = \left[ \left(G_{n+1}^{\BC}(\lambda (x))\right)^T (\Pi_{n+1}f)(\lambda (x))\right]_{\iota (x)}
\end{equation}
on $L^1[-1,1]$.

We conclude the section by obtaining positive strongly continuous contraction semigroups from the bounded operators $G^{\mathrm{LR}}_{\mp h}$ on the bounded domain $[-1,1]$.
 \begin{lemma}
\label{Ghdissipative}
The bounded operators $G^{\mathrm{LR}}_{\mp h}$ defined in \eqref{G-hBC} and \eqref{G+hBC} via Table \ref{mainTable} are generators  of positive strongly continuous contraction semigroups on $C_0(\Omega)$ (i.e. Feller semigroups) and on $L^1[-1,1]$, respectively. 
\end{lemma} 
\begin{proof}
Easy computations using \eqref{eq:grun_phi}, \eqref{eq:grun_phi-1}, Table \ref{mainTable} and Definition \ref{definitiongridcoordinatefunctions}, prove that for each $\lambda$ all matrices $G^{\BC}_{n+1}(\lambda)$ define transition rate matrices; i.e.,   the row sums and the diagonal entries are non-positive and the off-diagonal entries are non-negative. Then   \cite[Proposition 15]{MR3720847} proves the lemma.
\end{proof}


\begin{table}
\begin{tabular}{|c|c|c|c|}
  \hline
  Case  & $b^{l,r}$ &$D^{l,r}$& $N^{l,r}$ \\ 
  \hline 
  $\mathrm{DR}$ & $b_i^l= \Gru_i$ & $D^l(\lambda)=\frac{ \lambda \Gru_0\calG_1^{k_0} }{\lambda \Gru_0 \calG_1^{k_0} +\barlambda }$ & $N^l=\mathbf{1}$\\
  $\mathrm{NR}$ & $b_i^l=-\sum_{j=0}^{i-1}\Gru_{j}$ & $D^l=\mathbf{1}$ & $N^l(\lambda)=\lambda$\\
  $\mathrm{N^*R}$ & $b_1^l=\sum_{j=0}^1\Gru_{j} ,b_i^l=\Gru_i ,i\ge2$ & $D^l=\mathbf{1}$ & $N^l(\lambda)=\lambda$\\
  $\mathrm{LD}$ & $b_i^r=\Gru_i$, $b_n=b_n^l$ & $D^r(\lambda)=\frac{ \barlambda  \Gru_0 \calG_1^{k_0} }{\barlambda \Gru_0 \calG_1^{k_0} +\lambda}$&$N^r=\mathbf{1}$\\
  $\mathrm{LN}$ & $b_i^r=-\sum_{j=0}^{i-1}\Gru_{j} , b_n=-\sum_{i=0}^{n-1}b_i^l$& $D^r=\mathbf{1}$&$N^r(\lambda)=\barlambda $\\
\hline
\end{tabular}
\caption{\label{mainTable}Table of boundary weights and interpolating functions used to build the interpolation matrix \eqref{interpolationmatrixGBC}, assuming $b_0^l=0$ for NR and ${\rm N^*R}$.}
\end{table}
\subsection{The approximating functions $\vartheta_{\pm h}^{k_{i}}$}\label{sec:approx_vartheta}
To apply Trotter–Kato Theorem we need, for $f$ in a core of $(\Gen^\pm,\BC)$, a sequence $f_h\to f$ such that the interpolated bounded generators $G^{\BC}_{\pm h} f_h$ converge strongly to $(\Gen^\pm,\BC)f$. As (essentially) all functions in  $\D(\Gen^\pm,\BC)$ feature at least one of the scale functions $k_i^\pm$, $i\in\{-1,0,1\}$ (Table \ref{tab:ndo}), we   construct  natural approximations  $\vartheta^{k_i}_{\pm h}$  as  interpolations of the Post–Widder coefficients of Lemma \ref{lem:PWthms}. In Lemma \ref{lem:varthetaDD} we show two facts: the desired convergence $\vartheta^{k_i}_{\pm h}\to k^\pm_i$; and that the interaction of $\vartheta^{k_i}_{\pm h}$ with  $\Cpmh$ mimics the one of     $k_i^\pm$ with  $\Cpm $ away from the boundaries (for example, compare \eqref{eq:partialphithetah0} and \eqref{eq:ker}). It turns out that we need to carefully weight the interpolation structure of each $\vartheta^{k_i}_{\pm h}$ in order to compensate for the singularities arising in the limit at the first (last) interval of the grid, due to the interpolation structure of $G^{\mathrm{LR}}_{\ph}$ ($G^{\mathrm{LR}}_{\mh}$). This delicate selection of interpolating coefficients in Table \ref{tab:varthetadetails}  essentially derives from reverse engineering the statements of Lemmata \ref{lem:varthetaDD}, \ref{lem:varthetaDD2},  and \ref{lem:NRk-1}. 
\begin{definition}\label{def:varthetadetails}
Recall \eqref{eq:GCdef} and Definition \ref{definitiongridcoordinatefunctions}. For each $i\in\{0,1,-1\}$, we define $\vartheta^{k_{i}}_{\cdot\, h}:(-\infty,1]\to\mathbb R$ as 
\begin{equation*}
\vartheta^{k_{i}}_{\cdot\, h}(x) =\left\{ \begin{split} &\frac1h \Big( \big(1- \theta (\lambda) \big) \mathcal{G}_{\iota( x)-2-\tau}^{k_i}+ \theta (\lambda) \mathcal{G}_{\iota( x)-1-\tau}^{k_i}\Big), &\iota( x) \neq 1,\\
&\text{see Table \ref{tab:varthetadetails}}, &\iota(x)=1, \\
&0, &x\in (-\infty,-1),
\end{split}\right.
\end{equation*} 
where the parameters $\theta(\lambda)
$ and $\tau$  are defined in Table \ref{tab:varthetadetails}.  Then,  we define
\begin{align*}
\vartheta^{k_{i}}_{\mh}(x)&=\vartheta^{k_{i}}_{\cdot\, h}(-x) &  \text{if }X=C(\Omega),\\
\vartheta^{k_{i}}_{\ph}(x)&=\vartheta^{k_{i}}_{\cdot\, h}(x)&  \text{if }X= L^1[-1,1],
\end{align*} 
and     $\vartheta^{0}_{\ph}(x)=1\; \mathrm{if} \;\iota (x) \neq 1$, $ \vartheta^{0}_{\ph}=\lambda \; \mathrm{if} \;\iota (x) =1$ and $\vartheta^{0}_{\ph}=0$ on $(-\infty,-1)$. 

\begin{table}
\centering
\vline
\begin{tabular}{c|c|c|}
  \hline
  $\,$  & $L^1[-1,1]$ & $C_0(\Omega)$ \\ 
\hline
$i=1$ & $\begin{matrix} \theta (\lambda)= 1,\; \tau =1\hfill\\
                                               \vartheta^{k_1}_{\cdot\, h} (x) = - h^{-1}\barlambda  \mathcal{G}^{k_1}_{0} ,\; \mathrm{if} \;\iota (x) =1 \end{matrix}$ & $\begin{matrix} \theta (\lambda)= \lambda,\;\tau =1 \hfill\\
                                               \vartheta^{k_1}_{\cdot\, h} (x) = - h^{-1} \barlambda  \mathcal{G}^{k_1}_{0} ,\; \mathrm{if} \; \iota (x) =1 \end{matrix}$\\
\hline
$i=0$ & $\begin{matrix} \theta (\lambda)= \lambda,\;\tau =0 \hfill\\
                                               \vartheta^{k_0}_{\cdot\, h} (x) =-\frac{\Gru_{0}}{h\Gru_{1}}(\barlambda  \mathcal{G}^{k_{0}}_{0} +\lambda \mathcal{G}^{k_0}_{1})
											
																							,\; \mathrm{if} \;\iota (x) =1 \end{matrix}$ & $\begin{matrix} \theta (\lambda)= \lambda,\;\tau =0 \hfill\\
                                               \vartheta^{k_{0}}_{\cdot\, h}(x) = h^{-1} \lambda \mathcal{G}^{k_0}_{0} ,\; \mathrm{if} \; \iota (x) =1 \end{matrix}$\\
\hline

$i=-1$ & $\begin{matrix} \theta (\lambda)= \frac{\lambda}{\lambda+\barlambda h\Gru_{0} \calG^{k_{-1}}_{1} }, \;\tau =0\hfill\\
                                               \vartheta^{k_{-1}}_{\cdot\, h} (x) = h^{-1} \theta (\lambda ) \mathcal{G}^{k_{-1}}_{0} ,\; \mathrm{if} \;\iota (x) =1 \end{matrix}$ & \\
\hline
\end{tabular}
\caption{\label{tab:varthetadetails}
Definitions of the parameters $\theta(\lambda),\tau$ and of $\vartheta^{k_{i}}_{\cdot\, h} $ for $ \iota( x) = 1$ . }
\end{table}

\end{definition}

\begin{remark}[Canonical extensions]\label{rmk:canon} 
In the proofs of the theorems of Section \ref{sec:caseC} (resp. Section  \ref{sec:caseL}) we sometimes need to smoothly extend our functions to the interval $[-1-h,- 1)$ (resp. $(1, 1 + h]$). The \textit{canonical extension for the approximating functions} $\vartheta^{k_{i}}_{\pm h}$    is defined as in Definition \ref{def:varthetadetails} for $\iota(x)\neq1$ by extending $\vartheta^{k_{i}}_{\cdot\, h}$ to $(1,1+h]$ by extending  the definitions of $\iota$ and $\lambda$ to  $\iota (x) = \left\lfloor \frac{x+1}{h}\right\rfloor +1$  for $x\in(1, 1 + h) $, $\iota(1 + h) = n + 2$ and $\lambda(x) = \frac{x+1}{h} - ( \iota(x)-1) $ for $x\in(1, 1 + h]$, and then letting $\vartheta^{k_{i}}_{\cdot\, h}=0$ on $( 1+h,\infty)$.    Meanwhile, the \textit{canonical extension for a generic function }$g\in C^m[a,b]$, $m\in \mathbb N$ is the     extension   to $ g\in  C^m(\mathbb R)$ by the $m$-th order Taylor expansions at $a$ and $b$ of $g$, namely $ g(x)=\sum_{j=0}^m g^{(j)}(b)(x-b)^j/j!$ for $x>b$, and $ g(x)=\sum_{j=0}^m g^{(j)}(a)(x-a)^j/j!$ for $x<a$. If $g\in C^\infty[a,b]$ we apply the canonical extension for $m=7$.
\end{remark}

\begin{remark}\label{rmk:Dphi_h}
In analogy with Remark \ref{rem:zeroext}, if $g:[-1,1]\to\mathbb R$  we   understand the operators $\Cpmh  $ and $\Conepmhzero  $ applied to    $\Pi^{-1}_{[-1,1]}g$, unless we explicitly apply the canonical extension of Remark \ref{rmk:canon}. In all these cases the test functions will satisfy $g(  x)=0$ for $\pm x\in(-\infty,-1)$, so that 
\begin{align*} 
\Cpmh  g( x)&=\sum_{j=0}^{\iota(\pm x)} \Gru_{j} g(x\mp (j-1)h), &\pm x\in[-1,1),\\
\Conepmhzero g( x)&=\sum_{j=0}^{\iota(\pm x)-1} \Gruone_{j} g(  x\mp jh), &\pm x\in[-1,1),
\end{align*}
and if $g(\mp 1)=0$, then the above identities also holds for $\pm x= 1$.  As if $g(-1)\neq 0$,   because   $\iota(1)=n+1\neq n+2$, we would have 
\[  
\Cpmh  g( \pm1)=\sum_{j=0}^{n+2} \Gru_{j} g(\pm 1\mp (j-1)h) \neq \sum_{j=0}^{\iota(\pm 1)} \Gru_{j} g(\pm 1\mp (j-1)h).
\] 
\end{remark}
\begin{remark}\label{rmk:coeftheta}
Note that   $\theta(\lambda)$ for $i=-1$ in Table \ref{tab:varthetadetails} satisfies $0\le\theta(\lambda)\le 1$, and $\theta(\lambda)(x)=0$ and $\theta(\lambda)(x-)=1$ on grid points $x$, which is a consequence of  $h\Gru_{0} \calG^{k_{-1}}_{1} > 0$, by  \eqref{eq:grun_phi} and \eqref{eq:grun_ki}.

\end{remark}

\begin{lemma}\label{lem:varthek0}
Assume \ref{H0}. Then   $\vartheta^{k_{i}}_{\mh}\to k_{i}^-$, as $h\to0$ in $C_0[-1,1)$ for $i\in\{0,1\}$,  and  $\vartheta^{k_{i}}_{\ph}\to k_{i}^+$, as $h\to0 $ in $L^1[-1,1]$ for $i\in\{-1,0,1\}$. Also,  the following identities hold
\begin{align}\label{eq:partialphithetah0}
\Crh \vartheta^{k_{0}}_{\mh}(x)&=0,\quad &x\in  [-1 +h ,1-h],\\
\label{eq:partialphithetah1}
\Crh \vartheta^{k_{1}}_{\mh}(x)&=1 - \frac{\lambda}h\calG^{k_1}_0\Gru_{\iota(-x)},\quad &x\in  [-1  +h,1-h],\\ 
\label{eq:partialphi-1thetah1}
 \Conerhzero \vartheta^{k_{1}}_{\mh}(x)&= (\iota(-x)-2+\barlambda )h -\frac{ \lambda}h \calG^{k_1}_0 \Gruone_{\iota(-x)-1},\quad &x\in (-1 ,1-2h],\\
\label{eq:partialphi-1thetah0}
 \Conerhzero \vartheta^{k_0}_{\mh}(x)&=1, &x\in [-1 ,1-h],\\
\label{eq:partialphithetah0+}
\Clh \vartheta^{k_{0}}_{\ph}(x)&= \Gru_{\iota(x)} \left(-\frac\lambda h  \calG^{k_0}_{0} +\frac{-\Gru_0}{h\Gru_1}\left(\barlambda  \calG^{k_0}_{0}+ \lambda \calG^{k_0}_{1}\right)\right), &x\in [-1+h,1  -h],\\
\label{eq:partialphi-1thetah0+}
 \Conelhzero \vartheta^{k_{0}}_{\ph}(x)&=1+\Gruone_{\iota(x)-1} \left(-\frac\lambda h \calG^{k_0}_{0}+ \frac{-\Gru_0}{h\Gru_1}\left(\barlambda  \calG^{k_0}_{0}+ \lambda \calG^{k_0}_{1}\right)\right),&x\in [-1+h,1),\\
\label{eq:partialphi-1thetah-1+}
 \Conelhzero \vartheta^{k_{-1}}_{\ph}(x)&=0,&x\in [-1+2h,1],\\
\label{eq:partialphithetah1+}
\Clh \vartheta^{k_{1}}_{\ph}(x)&=1-\frac{\barlambda }h\Gru_{\iota(x)} \calG^{k_1}_{0},&x\in [-1,1 -h],\\
\label{eq:partialphi-1thetah1+}
 \Conelhzero \vartheta^{k_{1}}_{\ph}(x)&=(\iota(x)-1)h-\frac{\barlambda }h\Gruone_{\iota(x)-1} \calG^{k_1}_{0},&x\in [-1+h,1),\\
\label{eq:partialphithetah-1+}
\Clh \vartheta^{k_{-1}}_{\ph}(x)&=0,  &x\in [-1+2h,1  -h],
\end{align}
and for the canonical extensions of $\vartheta^{k_{0}}_{\mh}$, $i\in\{0,1\}$, for $x\in [-1 , 1-h]$
\begin{equation}\label{eq:diffvarth}
\begin{split}&\vartheta^{k_{0}}_{\mh}(x)-\vartheta^{k_{0}}_{\mh}(x-h)=- \calG^{k_{-1}}_{\iota(-x)-1} -\barlambda h \calG^{k_{-2}}_{\iota(-x)},\\
&\vartheta^{k_{1}}_{\mh}(x)-\vartheta^{k_{1}}_{\mh}(x-h)=- \calG^{k_{0}}_{\iota(-x)-2} -\barlambda h \calG^{k_{-1}}_{\iota(-x)-1}.
\end{split}
\end{equation}
\end{lemma} 
\begin{proof} Recall Definition \ref{def:varthetadetails}, Remarks \ref{rmk:canon} and  \ref{rmk:coeftheta} and that $h=2/(n+1)$. A direct check confirms that for each $h>0$, $\vartheta^{k_{i}}_{\mh}\in C_0[-1,1)$ for $i\in\{0,1\}$, and $\vartheta^{k_{i}}_{\ph}\in X$ in the case $X=L^1[-1,1]$ for $i\in\{-1,0,1\}$. We now prove the convergences. 
For all cases, for an arbitrary $\epsilon>0$, let $\delta>0$ such that $\|k_i^+\|_{C[-1,-1+\delta]}\le \epsilon$ for both $i\in\{0,1\}$.  Then  for each $i\in\{-1,0,1\}$ and all $\pm x\in [-1+\delta,1]$ and $h<\delta$, recalling that $\vartheta^{k_{i}}_{\ph}(-x)=\vartheta^{k_{i}}_{\mh}(x)$, we have 
\begin{align*}
&\,\left|\vartheta^{k_{i}}_{\pm h}(x)-k^\pm_i(x)\right|\\
\le &\,\left|\frac1h\calG^{k_i}_{\iota(x)-2-\tau,h}-k^+_i(x)\right|+ \left|\frac1h\calG^{k_i}_{\iota(x)-1-\tau,h}-k_i^+(x)\right|\\
\le&\,  \left|\frac1h\calG^{k_i}_{\iota(x)-2-\tau,h}-k_i^+((\iota(x)-1)h-1)\right|+  \left|\frac1h\calG^{k_i}_{\iota(x)-1-\tau,h}-k_i^+((\iota(x)-1)h-1)\right|\\
&\,+2\left|k_i^+((\iota(x)-1)h-1)-k_i^+(x)\right|,
\end{align*}
and recalling that    $x=h(\iota(x)-1)-1+h\lambda(x)$, we conclude that the last term is $o(1)$ as $h\to 0$ uniformly in $ x\in [-1+\delta,1]$   by uniform continuity of $k_i^+$ on $[-1+\delta,1]$, and the first two terms are also $o(1)$ as $h\to 0$ uniformly in $ x\in [-1+\delta,1]$ by Lemma \ref{lem:PWthms}. It remains to show the convergence for  $\pm x\in [-1,-1+\delta)$. For $i=1$ and any $X$ or $i=0$ and $X=C_0(\Omega)$, recalling that $\calG^{k_{i}}_{m,h}\le \calG^{k_{i}}_{m+1,h}$ for any $m\in\mathbb N$ by \eqref{eq:phi+1tophi} and \eqref{eq:grun_ki}, we have  for any $\pm x\in[-1,-1+\delta)$
\begin{align*}
|\vartheta^{k_{i}}_{\pm h}(x)-k^\pm_i(x)|&\le  \frac1h\calG^{k_i}_{\iota(x)-2-\tau,h}+  \frac1h\calG^{k_i}_{\iota(x)-1-\tau,h}+\|k_i^+\|_{C[-1,-1+\delta]}\\
&\le  \frac1h\calG^{k_i}_{\iota(-1+\delta)-2-\tau,h}+  \frac1h\calG^{k_i}_{\iota(-1+\delta)-1-\tau,h}+\epsilon,
\end{align*} 
which is bounded by $5\epsilon$ for small $h$ due to   Lemma \ref{lem:PWthms}.  For $i=0$ and $X=L^1[-1,1]$ we use the above inequalities on $[-1+h,-1+\delta)$ meanwhile on the first  interval $[-1,-1+h)$, we use  \eqref{eq:GCdef} and \eqref{eq:convhomega} to obtain
\begin{align*}
|\vartheta^{k_{0}}_{\ph}(x)-k^+_0(x)|&\le \left|-\frac{\Gru_0}{h\Gru_1}\left(\barlambda \calG^{k_0}_0  +\lambda\calG^{k_0}_1 \right)\right|+k_0^+(h-1)\\
&= \frac{\sym(1/h)}{\sym'(1/h)}\left(\barlambda  \frac{1}{\sym(1/h)}  +\lambda\frac{\sym'(1/h)}{h(\sym(1/h))^2} \right)+k_0^+(h-1)\\
&=\barlambda  \frac{1}{\sym'(1/h)}  +\lambda\frac{1}{h\sym(1/h)}+k_0^+(h-1)\to 0 ,\quad \text{as }h\to 0.
\end{align*}
For the case $i=-1$, we compute using \eqref{eq:phi+1tophi} and \eqref{eq:grun_ki}
\begin{align*}
\int_{-1}^{-1+\delta}|\vartheta^{k_{-1}}_{\ph}(x)-k^+_{-1}(x)|\,\dd x &\le \frac2h\sum_{j=0}^{\iota(-1+\delta)-1} \calG^{k_{-1}}_{j,h} h + \|k_0^+\|_{C[-1,-1+\delta]}\\
&= \frac2h \calG^{k_0}_{\iota(-1+\delta)-1,h} + \|k_0^+\|_{C[-1,-1+\delta]},
\end{align*}
and we conclude by Lemma \ref{lem:PWthms} as above. The proofs of \eqref{eq:partialphithetah0}-\eqref{eq:diffvarth} are easily obtained by    rewriting each function to recover a convolution  term (to apply Lemma \ref{lem:convid}) plus a remainder, and the proof can be found the Appendix \ref{sec:app_lem:varthek0}.  
\end{proof}
\begin{remark}\label{rmk:canon_extra}
If $\Cpmh $ is applied to the canonical extension   of $\vartheta^{k_{0}}_{\pm h},\,\vartheta^{k_{1}}_{\pm h}$ or $\vartheta^{k_{-1}}_{\ph}$ (Remark \ref{rmk:canon}), then \eqref{eq:partialphithetah0} and \eqref{eq:partialphithetah-1+}  hold for $\pm x\in (1-h,1]$, meanhwile  \eqref{eq:partialphithetah1}, \eqref{eq:partialphithetah0+} and \eqref{eq:partialphithetah1+} hold on $\pm x\in (1-h,1)$.
\end{remark}
We conclude this section with some key properties of the approximating functions, showing that the interaction of $\vartheta^{k_{i}}_{\pm h}$ with $ \Gen^{{\rm LR}}_{\pm h}$   can resemble the one with $\Cpmh $, partially thanks to the specific interpolation weights of the $\vartheta^{k_{i}}_{\pm h}$'s. (The explicit entries of the matrices $\Gen^{\mathrm{LR}}_{n+1}$ involved  in the next proofs can be found in Section \ref{sec:caseC}.)

\begin{lemma}\label{lem:varthetaDD}
Assume \ref{H0}. Then  
\begin{equation}\label{eq:partialthetah0}
G^{\mathrm{DD}}_{\mh}\vartheta^{k_{0}}_{\mh}(x)=0,\quad 1\le\iota(-x)\le n,
\end{equation}
\end{lemma}
\begin{proof} For $2\le\iota(-x)\le n$, 
  $G^{\mathrm{DD}}_{\mh}\vartheta^{k_{0}}_{\mh}(x)=\Crh \vartheta^{k_{0}}_{\mh}(x)=0$ by \eqref{eq:partialphithetah0}. For $\iota(-x)=1$, we compute
\begin{align*}
G^{\mathrm{DD}}_{\mh}\vartheta^{k_{0}}_{\mh}(x)&=D^r(\lambda)\Gru_0  \vartheta^{k_{0}}_{\mh}(x-h) + \Gru_1 \vartheta^{k_{0}}_{\mh}(x)\\
&=\frac1h\left[-\frac{ \barlambda  \Gru_1  \calG_0^{k_0} }{\Gru_0 (\barlambda \calG_1^{k_0} +\lambda \calG_0^{k_0})}\Gru_0  \left( \lambda\calG^{k_{0}}_{0}+\barlambda \calG^{k_{0}}_{1}\right)+\Gru_1   \barlambda \calG^{k_{0}}_{0} \right]\\
&=0.
\end{align*}

\end{proof}

\begin{lemma}\label{lem:varthetaDD2} Assume \ref{H0}. Then, as $h\to0$,
\begin{equation}\label{eq:DDk0+}
G^{\mathrm{DD}}_{\ph}\vartheta^{k_{0}}_{\ph}(x)=
\left\{\begin{split}&0, &1\le\iota(x)\le n-1,\\
&(D^r(\lambda)-1)\Gru_{0} \vartheta^{k_{0}}_{\ph}(x+h)+o(1), &\iota(x)=n,\\
&- \Gru_{0}\frac1h \left(\barlambda  \calG^{k_0}_{n}+ \lambda \calG^{k_0}_{n+1}\right)+o(1), & \iota(x)=n+1.
\end{split}\right.
\end{equation}

\end{lemma}
\begin{proof} 
For $\iota(x)=1$
\begin{align*}
G^{\mathrm{DD}}_{\ph}\vartheta^{k_{0}}_{\ph}(x)&=\Gru_{1} \vartheta^{k_{0}}_{\ph}(x)+ \Gru_{0} \vartheta^{k_{0}}_{\ph}(x+h)\\ 
&=\Gru_{1} \frac1h \frac{-\Gru_0}{\Gru_1} \left(\barlambda  \calG^{k_0}_{0}+ \lambda \calG^{k_0}_{1}\right)+\frac1h \Gru_{0}\left(\barlambda  \calG^{k_0}_{0}+ \lambda \calG^{k_0}_{1}\right)
=0.
\end{align*}
For $2\le\iota(x)\le n-1$, using \eqref{eq:partialphithetah0+} in the second identity and the definitions \eqref{eq:GCdef} in the third,
\begin{align}\nonumber
G^{\mathrm{DD}}_{\ph}\vartheta^{k_{0}}_{\ph}(x)&=\sum_{j=0}^{\iota(x)-1} \Gru_{j} \vartheta^{k_{0}}_{\ph}(x-(j-1)h)+  D^l(\lambda) \Gru_{\iota(x)} \vartheta^{k_{0}}_{\ph}(x-(\iota(x)-1)h) \\ \nonumber
&=-\frac\lambda h \Gru_{\iota(x)} \calG^{k_0}_{0} + D^l(\lambda) \Gru_{\iota(x)} \vartheta^{k_{0}}_{\ph}(x-(\iota(x)-1)h) \\ \nonumber
&=-\frac\lambda h \Gru_{\iota(x)} \calG^{k_0}_{0}+\lambda \Gru_{\iota(x)} \frac{ \calG^{k_0}_0}{h} =0.
\end{align}
For $\iota(x)=n$,  using \eqref{eq:partialphithetah0+} again,
\begin{align*}\nonumber
&\,G^{\mathrm{DD}}_{\ph}\vartheta^{k_{0}}_{\ph}(x)\\
=&\,D^r(\lambda)\Gru_{0} \vartheta^{k_{0}}_{\ph}(x+h)+\sum_{j=1}^{n-1} \Gru_{j} \vartheta^{k_{0}}_{\ph}(x-(j-1)h)+  D^l(\lambda) \Gru_{n} \vartheta^{k_{0}}_{\ph}(x-(n-1)h) \\ \nonumber
=&\,(D^r(\lambda)-1)\Gru_{0} \vartheta^{k_{0}}_{\ph}(x+h)-\frac\lambda h \Gru_{n} \calG^{k_0}_{0} +  D^l(\lambda) \frac{\Gru_{n}}h \frac{-\Gru_0}{\Gru_1}\left(\barlambda  \calG^{k_0}_{0}+ \lambda \calG^{k_0}_{1}\right)\\ \nonumber
=&\,(D^r(\lambda)-1)\Gru_{0} \vartheta^{k_{0}}_{\ph}(x+h)+o(1),
\end{align*}
where the  $o(1)$ term is a consequence of \eqref{eq:GClimphi} and \eqref{eq:convhomega} with  $  \calG^{k_0}_{0}/h=   (h\sym(1/h))^{-1}=o(1)$ 
and  
\[
\frac1h\frac{-\Gru_0}{\Gru_1} \left(\barlambda  \calG^{k_0}_{0}+ \lambda \calG^{k_0}_{1}\right) = \frac{\sym(1/h)}{\sym'(1/h)}\left(\barlambda  \frac{1}{\sym(1/h)}+ \lambda \frac1h \frac{\sym'(1/h)}{(\sym(1/h))^2}\right)=o(1).
\]
 For $\iota(x)=n+1$, using   \eqref{eq:partialphithetah0+}   for the canonical extension (Remark \ref{rmk:canon_extra}\footnote{Which holds for $x=1$ for the current summation, by observations as in Remark \ref{rmk:Dphi_h}.})
\begin{align}\nonumber
G^{\mathrm{DD}}_{\ph}\vartheta^{k_{0}}_{\ph}(x)&=\frac{- \Gru_{0}}h \left(\barlambda  \calG^{k_0}_{n}+ \lambda \calG^{k_0}_{n+1}\right) -\frac\lambda h \Gru_{n+1} \calG^{k_0}_{0}=\frac{- \Gru_{0}}h \left(\barlambda  \calG^{k_0}_{n}+ \lambda \calG^{k_0}_{n+1}\right)+o(1).
\end{align}
\end{proof}

\begin{lemma}\label{lem:NRk-1} 
Assume \ref{H0}. Then $G^{\mathrm{N^*R}}_{\ph} \vartheta^{k_{-1}}_{\ph}(x)=0$ for $1\le\iota(x)\le n-1$.
\end{lemma}
\begin{proof}
From \eqref{eq:dconvk-1} we have the three identities
\begin{equation}\label{eq:num1}
\Gru_0 \calG^{k_{-1}}_0=\frac1h,\quad\Gru_0 \calG^{k_{-1}}_1+\Gru_1 \calG^{k_{-1}}_0= -\frac1h,\quad \Gru_0\calG^{k_{-1}}_2+\Gru_2\calG^{k_{-1}}_0 =-\Gru_1\calG^{k_{-1}}_1,
\end{equation}
and as $\theta= \lambda/(h\Gru_0 \calG^{k_{-1}}_1\barlambda +\lambda )$ we easily obtain
\begin{align}\label{eq:num3}
(1-\theta)\calG^{k_{-1}}_0&=\frac{\barlambda }\lambda \theta \calG^{k_{-1}}_1,\\ \label{eq:num2}
\frac\lambda h(1-\theta)&=\barlambda \Gru_0\calG^{k_{-1}}_1\theta,
\end{align}
Thus, for $\iota(x)=1$, using \eqref{eq:num1} and \eqref{eq:num3},
\begin{align*}
G^{\mathrm{N^*R}}_{\ph} \vartheta^{k_{-1}}_{\ph}(x)&=\lambda \Gru_0\vartheta^{k_{-1}}_{\ph}(x+h)+(\Gru_0+\Gru_1)\vartheta^{k_{-1}}_{\ph}(x)\\
&=\lambda \Gru_0\frac1h \left((1-\theta)\calG^{k_{-1}}_0+\theta \calG^{k_{-1}}_1\right)+(\Gru_0+\Gru_1)\theta\calG^{k_{-1}}_0\frac1h\\
&=\lambda \Gru_0\frac1h \frac{\theta}\lambda\calG^{k_{-1}}_1+(\Gru_0+\Gru_1)\theta\calG^{k_{-1}}_0\frac1h\\
&=\frac\theta h\left[\Gru_0  \calG^{k_{-1}}_1+\Gru_0\calG^{k_{-1}}_0+\Gru_1\calG^{k_{-1}}_0\right]\\
&=0.
\end{align*}
 For $\iota(x)=2$, using \eqref{eq:num1} for the third and fifth identity, and  \eqref{eq:num2} for the last identity, 
\begin{align*}
&\, hG^{\mathrm{N^*R}}_{\ph} \vartheta^{k_{-1}}_{\ph}(x)\\
=&\,h\left[\Gru_0\vartheta^{k_{-1}}_{\ph}(x+h)+(\lambda\Gru_1+\barlambda (\Gru_1+\Gru_0))\vartheta^{k_{-1}}_{\ph}(x)+\Gru_2\vartheta^{k_{-1}}_{\ph}(x-h)\right]\\
=&\,\Gru_0\left((1-\theta)\calG^{k_{-1}}_1+\theta \calG^{k_{-1}}_2\right)+(\Gru_1+\barlambda \Gru_0)\left((1-\theta)\calG^{k_{-1}}_0+\theta \calG^{k_{-1}}_1\right)+\Gru_2\theta\calG^{k_{-1}}_0\\
=&\,(1-\theta)\Gru_0\calG^{k_{-1}}_1+(\Gru_1+\barlambda \Gru_0)\left((1-\theta)\calG^{k_{-1}}_0+\theta \calG^{k_{-1}}_1\right)-\theta\Gru_1\calG^{k_{-1}}_1\\
=&\,(1-\theta)\left(\Gru_0\calG^{k_{-1}}_1+\barlambda \Gru_0\calG^{k_{-1}}_0+\Gru_1\calG^{k_{-1}}_0\right)+\barlambda \theta \Gru_0\calG^{k_{-1}}_1 \\
=&\,-\frac\lambda h(1-\theta)+\barlambda \theta \Gru_0\calG^{k_{-1}}_1 \\
=&\,0.
\end{align*}

Recalling the matrices \eqref{interpolationmatrixGBC_N*D} and \eqref{interpolationmatrixGBC_N*N}, for $3\le\iota(x)\le n-1$, $G^{\mathrm{N^*R}}_{\ph} \vartheta^{k_{-1}}_{\ph}=\Clh \vartheta^{k_{-1}}_{\ph}$, which equals 0 by \eqref{eq:partialphithetah-1+}.
\end{proof}

\section{Convergence of the interpolated difference schemes}\label{sec:case}
This section proves the Trotter–Kato convergence of Theorem \ref{thm:TK}-(ii) for the backward operators $(\Gen^-,\BC)$ (Section \ref{sec:caseC}) and the forward operators $(\Gen^+,\BC)$ (Section \ref{sec:caseL}) from Table \ref{explicitProcesses}. In more detail, for any $f$ in the core $\mathcal C(\Gen^\pm,\BC)$ (Proposition \ref{prop:cores}) we construct a sequence $f_h\in X$ for $h=2/(n+1)$, such that  as $h\to 0$
\[f_h\to f\quad \text{and}\quad \Gen^{\BC}_{\pm h}f_h \to (\Gen^\pm,\BC)f \quad \text{ in $X$}.
\]
All these proofs follow the same idea. Namely,  we first prove without too much trouble that  $\|\Gen^{\BC}_{\pm h}f_h- \Gen^\pm f \|_{X_h}\le \|\Cpmh \Ipm g-g\|_{X}+ o(1) $ as $h\to 0$, for some smooth $g$ (to apply Corollary \ref{cor:L94}), where $X_h$ is $C[h-1,1-h]$ or $L^1[h-1,1-h]$ if $X$ is $C_0(\Omega)$ or $L^1[-1,1]$, respectively. Secondly, we treat the  convergence   close to the boundaries $[-1,h-1)$ and $(1-h,1]$. The boundary convergence often requires us to control the divergent term $\Gru_{0,h}=\psi(1/h)$, so that we work hard to compensate  it with a $O(h^2)$ term, recalling \eqref{eq:convhomega}. This  difficult boundary control has to depend on the smoothness of the limiting function $f$, which is dictated by the scale functions $k_i$. It turns out that it is natural to require continuous differentiability of $k_{-1}$   in four cases, to access the asymptotics \eqref{eq:GClimk-2k0} and \eqref{eq:GClimk-2k0Tricky}.  Hence we assume \ref{H1} in these four cases (see Table \ref{tab:Hs}). However,  we only assume \ref{H0} when treating any left fast-forward boundary condition (i.e. ${\rm NR}$), which is the main result of our work. (We collected in Table \ref{tab:convergencerates} the orders of convergence proved in the proofs of the theorems of this section.)

\begin{table}
\centering
\vline
\begin{tabular}{l|l|}
	\hline
  $X = C_0(\Omega)$ & $X = L^1[-1,1]$ \\
	\hline
  $ (\Cr , \mathrm{DD}):\quad$\ref{H1} & $ (\Cl , \mathrm{DD}):\quad$\ref{H0} \\
  \hline
  $ (\Cr , \mathrm{DN}):\quad$\ref{H0} & $ (\Cl , \mathrm{DN}):\quad$\ref{H0}\\
  \hline
 $ (\Cr , \mathrm{ND}):\quad$\ref{H0} & $ (\Cl , \mathrm{ND}):\quad$\ref{H0}\\
  \hline

	  $ (\Cr , \mathrm{NN}):\quad$\ref{H0}
	  & 
	 $ (\Cl , \mathrm{NN}):\quad$\ref{H0}\\
  \hline
 $ (\Cr , \mathrm{N^*D}):\quad$\ref{H1} & $
	 (\RLl , \mathrm{N^*D}):\quad$\ref{H1}\\
  \hline
 $ (\Cr , \mathrm{N^*N}):\quad$\ref{H1} & $
	 (\RLl , \mathrm{N^*N}):\quad$\ref{H0}\\
  \hline
\end{tabular}
\caption{\label{tab:Hs}
Assumptions for the approximation theorems of Section \ref{sec:case}. }

\end{table}

\begin{table}[h]
\centering
\vline
\begin{tabular}{l|l|}
	\hline
  $X = C_0(\Omega)$ & $X = L^1[-1,1]$ \\
	\hline
  $ (\Cr , \mathrm{DD}):\quad h^2\sym(1/h)$ & $ (\Cl , \mathrm{DD}):\quad  h^2\sym(1/h)$ \\
  $ (\Cr , \mathrm{DN}):\quad h^2\sym(1/h)$  & $ (\Cl , \mathrm{DN}):\quad h$\\
 $ (\Cr , \mathrm{ND}):\quad  1/[h\sym(1/h)]$ 	& $ (\Cl , \mathrm{ND}):\quad  h^2\sym(1/h)$\\
  $ (\Cr , \mathrm{NN}):\quad  1/[h\sym(1/h)]$  & $ (\Cl , \mathrm{NN}):\quad h$\\
 $ (\Cr , \mathrm{N^*D}):\quad h^2\sym(1/h)$  & $ (\RLl , \mathrm{N^*D}):\quad  h^2\sym(1/h)$\\
$ (\Cr , \mathrm{N^*N}):\quad h^2\sym(1/h)$ & $ (\RLl , \mathrm{N^*N}):\quad 1/[h\sym(1/h)]$\\  
  \hline
\end{tabular} 
\caption{\label{tab:convergencerates} List ``$(\Gen, \BC):w(h)$'' of the   convergence $\|\Gen_hf_h-\Gen f\|_X=O(w)$ as $h\to 0$ proved in the approximation theorems of Section \ref{sec:case}. In the stable-like case  ($1<\alpha<2$) the orders reduce to  $h^2\sym(1/h)\sim h^{2-\alpha}$ and $ 1/[h\sym(1/h)]\sim h^{\alpha-1}$. For the case $ (\RLl , \mathrm{N^*D})$ order of convergence is proved under the extra condition that $\phi$ allows a density with right and left limits at 2 (see Remark \ref{rem:thm12}). }
\end{table}

\subsection{The case $X=C_0(\Omega)$}\label{sec:caseC} 
The proofs of the theorems in this section can be found in Appendix \ref{app:caseC}. (See Table \ref{tab:fhexplicitC} for the definition of all the approximating    functions $f_h$.)
 
Plugging in the values in Table \ref{mainTable}  in the approximating interpolation matrix \eqref{interpolationmatrixGBC}, we obtain
\begin{equation}
\label{interpolationmatrixGBC_DD}
  G^{\mathrm{DD}}_{n+1}(\lambda)=\begin{pmatrix}
    \Gru_1& D^l(\lambda)\Gru_2&\cdots& \cdots&D^l(\lambda)\Gru_n & 0\\
     \Gru_0 &  \Gru_1 &\cdots&\cdots& \Gru_{n-1}  &  \Gru_n\\
    0 & \Gru_0 & \cdots&\Gru_{n-3}&\vdots & \vdots \\
    \vdots & \ddots &\ddots& \vdots &\vdots&\vdots\\
   \vdots &  &\ddots&\Gru_0 & \Gru_1 &\Gru_2\\
    0 &\cdots &\cdots& 0& D^r(\lambda) \Gru_0 & \Gru_1
  \end{pmatrix}.
\end{equation}

Let us observe that the weight 
\[
D^l(\lambda)=\frac{ \lambda \Gru_0 \calG_1^{k_0} }{\lambda\Gru_0\calG_1^{k_0} +\barlambda  }=-\frac{ \lambda \Gru_1 \calG_0^{k_0} }{\Gru_0(\lambda\calG_1^{k_0} +\barlambda  \calG_0^{k_0})} 
\]
   is chosen so that $G^{\mathrm{DD}}_{\mh}\vartheta^{k_{0}}_{\mh}(x)=0$ when $x\in(1-h,1]$ (see proof of Lemma \ref{lem:varthetaDD}). Moreover, it is clear that $D^l(\lambda)\in C[-1,-1+h)$, $D^l(0),\,D^l(1)=0$, $D^l(1-)=1$, and $0\le D^l(\lambda)\le 1$.  In the fractional case $\Gru_0 =h^{-\alpha}$ and $\calG_1^{k_0}=\alpha h^\alpha $, so that $D^l(\lambda)$ agrees with its counterpart in \cite{MR3720847}.

\begin{theorem}[The case $(\Cr ,\mathrm{DD})$]\label{thm:conv_C_DD}
Assume \ref{H1}. If $f=\Ir  g+\Ir  g(-1) k_{0}^-/ k_{0}^-(-1)$ and 
\[
f_h=\Ir  g+b_h\vartheta^{k_{0}}_{\mh}
\]
for each $h>0$, with $g\in C_c^\infty(-1,1)$  and $b_h=-\Ir  g(-1)/\vartheta^{k_{0}}_{\mh}(-1)$, then $f_h\to f$ and $G^{\mathrm{DD}}_{\mh}f_h\to g $ as $h\to0$ in $C_0(-1,1)$. 
\end{theorem}

 Plugging in the values in Table \ref{mainTable}  in the approximating interpolation matrix \eqref{interpolationmatrixGBC}, we obtain
\begin{equation}  
\label{interpolationmatrixGBC_DN}
  G^{\mathrm{DN}}_{n+1}(\lambda)=\begin{pmatrix}
    \Gru_1& D^l(\lambda)\Gru_2&\cdots& \cdots&-D^l(\lambda)\sum_{j=0}^{n-1}\Gru_j & 0\\
     \Gru_0 &  \Gru_1 &\cdots&\cdots& -\lambda\sum_{j=0}^{n-2}\Gru_j  +\barlambda \Gru_{n-1}& -\barlambda  \sum_{j=0}^{n-1}\Gru_j\\
    0 & \Gru_0 & \cdots&\Gru_{n-3}&\vdots & \vdots \\
    \vdots & \ddots &\ddots& \vdots &\vdots&\vdots\\
   \vdots &  &\ddots&\Gru_0 & \mathcal -\lambda\Gru_0+\barlambda \Gru_1 & -\barlambda \sum_{j=0}^{1}\Gru_j\\
    0 &\cdots &\cdots& 0&  \Gru_0 &-  \Gru_0
  \end{pmatrix}.
\end{equation}
\begin{theorem}[The case $(\Cr ,\mathrm{DN})$]\label{thm:conv_C_DN} Assume \ref{H0}. 
If $f=\Ir  g-\Ir  g(-1)$ and   for each $h>0$
$$
f_h=\Ir  [g-g(1)]+b_h\vartheta^{k_{1}}_{\mh}-\Ir  g(-1),\quad b_h= \frac{g(1)k^-_{1}(-1)}{\vartheta^{k_1}_{\mh}(-1)},
$$ with 
 $g\in C_c^\infty(-1,1]$ such that $g-g(1)\in C_c[-1,1)$, then $f_h\to f$ and $G^{\mathrm{DN}}_{\mh}f_h\to g $ as $h\to0$ in $C_0(-1,1]$.
\end{theorem}

 Plugging in the values in Table \ref{mainTable}  in the approximating interpolation matrix \eqref{interpolationmatrixGBC}, we obtain
\begin{equation}
\label{interpolationmatrixGBC_ND}
  G^{\mathrm{ND}}_{n+1}(\lambda)=\begin{pmatrix}
    -\Gru_0& -\sum_{j=0}^{1}\Gru_j &\cdots& \cdots&-\sum_{j=0}^{n-1}\Gru_j & 0 \\
     \lambda\Gru_0 &  \lambda\Gru_1+\barlambda (-\Gru_0) &\cdots&\cdots& \lambda\Gru_{n-1}  +\barlambda (-\sum_{j=0}^{n-2}\Gru_j )& - \sum_{j=0}^{n-1}\Gru_j\\
    0 & \Gru_0 & \cdots&\Gru_{n-3}&\vdots & \Gru_{n-1} \\
    0 & 0& \Gru_0  &\cdots&\vdots & \Gru_{n-2} \\
    \vdots & \ddots &\ddots& \vdots &\vdots&\vdots\\
   \vdots &  &\ddots&\Gru_0 & \Gru_1 & \Gru_2\\
    0 &\cdots &\cdots& 0&  D^r(\lambda)\Gru_0 & \Gru_1
  \end{pmatrix}.
\end{equation}
\begin{theorem}[The case $(\Cr ,\mathrm{ND})$]\label{thm:conv_C_ND} Assume \ref{H0}.  Let $f=\Ir  g-I_- g(-1) k_0^-$ and 
$$
f_h=\Ir  g+b_h\vartheta^{k_{0}}_{\mh}+e_h ,
$$
for each $h>0$, where 
$
 b_h= -I_-g(-1) - hF_-[g](-1), 
$
\begin{equation*} 
e_h(x)=\left\{\begin{split}& -\barlambda \calG^{k_0}_0 d, & x\in[-1,-1+h],\\
&0,& x\in(-1+h,1],
\end{split}\right. 
\end{equation*}
 $F_-$ is defined   in Lemma \ref{lem:firtorderapprox}, $d=g(-1)$ and  $g\in C_c^\infty[-1,1)$. Then $f_h\to f$ and $G^{\mathrm{ND}}_{\mh}f_h\to g $ as $h\to0$ in $C_0[-1,1)$.
\end{theorem}

 Plugging in the values in Table \ref{mainTable}  in the approximating interpolation matrix \eqref{interpolationmatrixGBC},  we obtain
\begin{equation}
\label{interpolationmatrixGBC_NN}
  G^{\mathrm{NN}}_{n+1}(\lambda)=\begin{pmatrix}
    -\Gru_0& -\sum_{j=0}^{1}\Gru_j &\cdots& \cdots&\sum_{j=0}^{n-2}(\sum_{i=0}^{j}\Gru_i) & 0 \\
     \lambda\Gru_0 &  \lambda\Gru_1+\barlambda (-\Gru_0) &\cdots&\cdots& -\sum_{j=0}^{n-2}\Gru_j & \barlambda\sum_{j=0}^{n-2}(\sum_{i=0}^{j}\Gru_i)\\
    0 & \Gru_0  & \Gru_1&\cdots& \barlambda  \Gru_{n-2} -\lambda\sum_{j=0}^{n-3}\Gru_j  & -\barlambda \sum_{j=0}^{n-2}\Gru_j \\
    0 & \cdots& \ddots   &\cdots&\vdots & \vdots \\
   \vdots &  &\ddots&\Gru_0 & \barlambda \Gru_1+\lambda (-\Gru_0) &- \barlambda \sum_{j=0}^{1}\Gru_j\\
    0 &\cdots &\cdots& 0&  \Gru_0 & -\Gru_0
  \end{pmatrix}.
\end{equation}
\begin{theorem}[The case $(\Cr ,\mathrm{NN})$]\label{thm:conv_C_NN} Assume \ref{H0}.  Let $f=\Ir  g-I_- g(-1) k^-_1/2+c$, $c\in\mathbb R$, and 
\[
f_h=\Ir  [g-g(1)] +b_h\vartheta^{k_{1}}_{\mh}+e_h +c ,
\]
for each $h>0$, where 
\begin{equation*}
e_h(x)=\left\{\begin{split}& -\barlambda \calG^{k_0}_0 d_h, & x\in[-1,-1+h],\\
&0,& x\in(-1+h,1],
\end{split}\right. 
\end{equation*}
$$
b_h=\left[g(1) -\frac{I_-g(-1)}2 - \frac h2F_-[g-g(1)](-1)\right]\frac{n+1}{n},\quad d_h=  g(-1) -\frac{I_-g(-1)}2\frac{n+1}{n},
$$
 with $g\in C^\infty[-1,1]$ such that $g-g(1)\in C_c[-1,1)$, and $F_- $ is defined in Lemma \ref{lem:firtorderapprox}. Then $f_h\to f$ and $G^{\mathrm{NN}}_{\mh}f_h\to g-I_-g(-1)/2 $ as $h\to0$ in $C[-1,1]$.
\end{theorem} 

 Plugging in the values in Table \ref{mainTable}  in the approximating interpolation matrix \eqref{interpolationmatrixGBC}, we obtain
\begin{equation}
\label{interpolationmatrixGBC_N*D}
  G^{\mathrm{N^*D}}_{n+1}(\lambda)=\begin{pmatrix}
    \Gru_1+ \Gru_0&  \Gru_2 &\cdots&\cdots& \Gru_n & 0 \\
     \lambda\Gru_0 & \lambda \Gru_1+\barlambda  (\Gru_1+ \Gru_0) &\cdots&\cdots&  \Gru_{n-1}&\Gru_n\\
    0 & \Gru_0 & \cdots&\Gru_{n-3}&\vdots &\Gru_{n-1} \\
    \vdots & \ddots &\ddots& \vdots &\vdots&\vdots\\
   \vdots &  &\ddots&\Gru_0 & \Gru_1 & \Gru_2\\
    0 &\cdots &\cdots& 0&  D^r(\lambda)\Gru_0 & \Gru_1
  \end{pmatrix}.  
\end{equation}

\begin{theorem}The case $(\Cr ,\mathrm{N^*D})$\label{thm:conv_C_N*D} Assume  \ref{H1}.
Let $f=\Ir  g+[\Ir  g]'(-1) k^-_0/k^-_{-1}(-1)$, and 
$$
f_h=\Ir  g +b_h\vartheta^{k_{0}}_{\mh},
$$ 
for each $h>0$, where
$$
b_h=\frac{[\Ir  g]'(-1)}{\calG^{k_{-1}}_{n}/h},
$$
 with $g\in C^\infty_c[-1,1)$. Then $f_h\to f$ and $G^{\mathrm{N^*D}}_{\mh}f_h\to g $ as $h\to0$ in $C_0[-1,1)$.
\end{theorem}

 Plugging in the values in Table \ref{mainTable}  in the approximating interpolation matrix \eqref{interpolationmatrixGBC}, we obtain
\begin{equation}
\label{interpolationmatrixGBC_N*N}
  G^{\mathrm{N^*N}}_{n+1}(\lambda)=\begin{pmatrix}
    \Gru_1+ \Gru_0&  \Gru_2 &\cdots&\cdots& -\sum_{j=0}^{n-1}\Gru_j & 0 \\
     \lambda\Gru_0 &  \lambda\Gru_1+\barlambda ( \Gru_1+ \Gru_0) &\cdots&\cdots& \barlambda  \Gru_{n-1}+\lambda(-\sum_{j=0}^{n-2}\Gru_j)& -\barlambda \sum_{j=0}^{n-1}\Gru_j\\
    0 & \Gru_0 & \cdots&\Gru_{n-3}&\vdots &-\barlambda \sum_{j=0}^{n-2}\Gru_j \\
    \vdots & \ddots &\ddots& \vdots &\vdots&\vdots\\
   \vdots &  &\ddots&\Gru_0 & \barlambda \Gru_1+\lambda(-\Gru_0) & -\barlambda \sum_{j=0}^1\Gru_j\\
    0 &\cdots &\cdots& 0&  \Gru_0 & -\Gru_0
  \end{pmatrix}.
\end{equation}
\begin{theorem}[The case $(\Cr ,\mathrm{N^*N})$]\label{thm:conv_C_N*N} Assume   \ref{H1}.
Let \[f=\Ir  g+ [\Ir  g]'(-1) k^-_1/k^-_{0}(-1)+c\]  and 
$$
f_h=\Ir  [g-g(1)] +b_h\vartheta^{k_{1}}_{\mh}+c,
$$ 
for each $h>0$, where
$$
b_h= g(1)+\frac{[\Ir  g]'(-1)}{\calG^{k_{0}}_{n-1}/h},
$$
 with $g\in C^\infty[-1,1]$ such that $g-g(1)\in C_c[-1,1)$ and $c\in\mathbb R$. Then $f_h\to f$ and $G^{\mathrm{N^*N}}_{\mh}f_h\to g+[\Ir  g]'(-1)/k^-_{0}(-1) $ as $h\to0$ in $C[-1,1]$.
\end{theorem}

\subsection{The case $X=L^1[-1,1]$}\label{sec:caseL}
The proofs of the theorems in this section can be found in Appendix \ref{app:caseL}. (See Table \ref{fhexplicit1} for the definition of all the approximating    functions $f_h$.)


\begin{theorem}[The case $(\Cl ,\mathrm{DD})$]\label{thm:conv_L_DD}
Assume   \ref{H0}. Let $f=\Il  g-\Il  g(1)k^+_0/k^+_0(1)$ and 
\[
f_h=\Il  g+b_h\vartheta^{k_{0}}_{\ph},
\]
for each $h>0$, with $g\in C_c^\infty(-1,1)$  and $b_h=- \Il  g(1)/\vartheta^{k_{0}}_{\ph}(1)$. Then $f_h\to f$ and  $G^{\mathrm{DD}}_{\ph}f_h\to g $ as $h\to0$ in $L^1[-1,1]$.
\end{theorem}

Note that for the dual problem $(\Cr ,\mathrm{DD})$ on $X=C_0(-1,1)$, we required the stronger assumption \ref{H1} in place of \ref{H0} (the same comment applies to Theorem \ref{thm:conv_L_N*N}).


\begin{theorem}[The case $(\Cl ,\mathrm{DN})$]\label{thm:conv_L_DN}
Assume   \ref{H0}. Let $f=\Il  g-I_+ g(1)k^+_0$, and $$f_h=\Il  g+b\vartheta^{k_{0}}_{\ph}+e_h, $$ with $g\in C_c^\infty(-1,1)$, $b=-I_+ g(1)$ and
\begin{equation}
e_h(x)=\left\{\begin{split}& 0, & x\in[-1,1-h),\\
&-\lambda \left( \Il g(x)+\vartheta^{k_{0}}_{\ph}(x)\right),& x\in[1-h,1].
\end{split}\right.  
\end{equation}
 Then $f_h\to f$ and  $G^{\mathrm{DN}}_{\ph}f_h\to g $ as $h\to0$ in $L^1[-1,1]$.
\end{theorem}

\begin{theorem}[The case $(\Cl ,\mathrm{ND})$]\label{thm:conv_L_ND}
Assume   \ref{H0}. Let $f=\Il  g-\Il  g(1)$, and 
\[
f_h=\Il  g+b\vartheta_{\ph}^0,
\] 
with $g\in C_c^\infty(-1,1)$ and $b=-\Il g(1)$. Then $f_h\to f$ and  $G^{\mathrm{ND}}_{\ph}f_h\to g $ as $h\to0$ in $L^1[-1,1]$.
\end{theorem}

   

\begin{theorem}[The case $(\Cl ,\mathrm{NN})$]\label{thm:conv_L_NN}
Assume \ref{H0}. Let $f=\Il  g-I_+ g(1)k^+_1/2+c$,  and 
\[
f_h=\Il  g+b \vartheta^{k_1}_{\ph}+c\vartheta^{0}_{\ph}+e_h,
\] 
with $c\in\mathbb R$, $b=- I_+g(1)/2$,  $g\in C_c^\infty(-1,1)$ and
\begin{equation*} e_h(x)=\left\{\begin{split}& 0, & x\in[-1,1-h),\\ 
&-\lambda (\Il g(x)+b\vartheta_{\ph}^{k_1}(x)+c),& x\in[1-h,1].
\end{split}\right.
\end{equation*}
 Then $f_h\to f$ and  $G^{\mathrm{NN}}_{\ph}f_h\to g-I_+ g(1)/2 $ as $h\to0$ in $L^1[-1,1]$.
\end{theorem}

In the next two theorems we use the approximating function $\vartheta^{k_{-1}}_{\ph}$, which features the weight $\theta$ in its definition (see Table \ref{tab:varthetadetails}). Similarly as for  the weight $D^l(\lambda)$, the weight $\theta$ is chosen so that $G^{\mathrm{N^*R}}_{\ph}\vartheta^{k_{-1}}_{\ph}$ vanishes on the first interval $[-1,-1+h)$. It then turns out $G^{\mathrm{N^*R}}_{\ph}\vartheta^{k_{-1}}_{\ph}$ also vanishes on $[-1+h,-1+2h)$ (see Lemma \ref{lem:NRk-1}).
\begin{theorem}[The case $(\RLl ,\mathrm{N^*D})$]\label{thm:conv_L_N*D}
Assume   \ref{H1}. Let $f=\Il  g-\Il  g(1)k^+_{-1}/k^+_{-1}(1)$,  and 
\[
f_h=\Il  g+b_h \vartheta^{k_{-1}}_{\ph},
\]    
with $b_h= -\Il g(1)/\vartheta^{k_{-1}}_{\ph}(1)$  and $g\in C_c^\infty(-1,1)$. Then $f_h\to f$ and  $G^{\mathrm{N^*D}}_{\ph}f_h\to g$ as $h\to0$ in $L^1[-1,1]$.
\end{theorem} 

  
\begin{theorem}[The case $(\RLl ,\mathrm{N^*N})$]\label{thm:conv_L_N*N}
Assume   \ref{H0}. Let 
$
f=\Il  g- I_+ g(1)k^+_{1}/2+c_{-1}k^+_{-1},
$
  and  
\[
f_h=\Il  g+b \vartheta^{k_{1}}_{\ph}+c_{-1}\vartheta^{k_{-1}}_{\ph}+e_h,
\] 
with  $g\in C_c^\infty(-1,1)$, $b=-I_+g(1)/2$, $c_{-1}\in\mathbb R$ and
\begin{equation*} e_h(x)=\left\{\begin{split}& 0, & x\in[-1,1-h),\\ 
&-\lambda \left(\Il g(x)+b\vartheta_{\ph}^{k_1}(x)+c_{-1}\vartheta_{\ph}^{k_{-1}}(x)\right),& x\in[1-h,1]. 
\end{split}\right. 
\end{equation*}
Then $f_h\to f$ and  $G^{\mathrm{N^*N}}_{\ph}f_h\to g-I_+g(1)/2$ as $h\to0$ in $L^1[-1,1]$.
\end{theorem}

\section{Convergence of the semigroups}
 The aim of this section is to combine the results from Sections \ref{sec:2}, \ref{sec:interpmat} and \ref{sec:caseL} to show that the operators in Table \ref{explicitProcesses} generate positive strongly continuous contraction semigroups on the respective Banach spaces, and that they are dual to each other. The argument is structured as follows. The results of Section \ref{sec:caseL} prove that all six operators $(\Gen^+,\BC)$ are dissipative and thus, in combination with the results of Section \ref{sec:2},  they generate strongly continuous contraction semigroups (by Lumer–Phillips Theorem). Then  Section  \ref{sec:caseL} and  Trotter–Kato Theorem prove the strong convergence of the approximating  semigroups, which are positive (Lemma \ref{Ghdissipative}), and thus the limit semigroups are positive. Finally, Corollary \ref{cor:duality} proves the claim for the backward semigroups $(\Gen^-,\BC)$.
 
\begin{theorem}[Trotter–Kato type approximation]
\label{trotterkatotypetheorem} Assume \ref{H1}. Then the operators  of Table \ref{explicitProcesses} generate
 positive strongly continuous contraction semigroups on $X$. Moreover, the semigroup generated by $G_{\pm, h}^{\BC}$ converges strongly (and uniformly for $t\in[0,t_0]$ for any $t_0>0$) to the semigroup generated by $(\Gen^\pm, \BC)$ on $X$. The above statement holds under the weaker condition \ref{H0} for all operators in Table \ref{explicitProcesses} apart from $ (\Cr , \mathrm{DD}),\, (\Cr , \mathrm{N^*D}),\, (\Cr , \mathrm{N^*N})$ and $ (\RLl , \mathrm{N^*D})$.
\end{theorem}
\begin{proof}
Under assumption \ref{H1}, by the twelve theorems of Section \ref{sec:case}, for each $f\in \mathcal C(\Gen,\mathrm{LR})$ there exists a sequence $\left\{f_h\right\} \subset X$ such that $f_h\to f$ and $\Gen_h f_h \to \Gen f$ in $X$, $\Gen_h\in\{\Gen^{\rm{LR}}_{\pm h}\}$. In view of Lemma \ref{Ghdissipative}, each $G_h$ is dissipative; that is, $\|(\LTp-\Gen_h)f_h\|_X\ge \LTp \|f_h\|_X$ for all $f_h\in X$ and all $\LTp>0$. Thus, in view of the twelve theorems of Section \ref{sec:case}, letting $h \to 0$  we obtain $\|(\LTp-\Gen)f\|_X\ge \LTp \|f\|_X$ for all $f\in \mathcal C(\Gen,\mathrm{LR})$, and  thus  for all $f\in \mathcal D(\Gen,\mathrm{LR})$ by Proposition \ref{prop:cores}, i.e., each $(\Gen,\BC )$ is dissipative. Furthermore, in view of Theorems \ref{thm:5}, \ref{thm:7} and \ref{thm:9}, all 
$(\Gen, \BC)$ are densely defined closed operators such that  the range of $ (1 -\Gen)$ is $X$. Hence, the operators $(\Gen,\BC )$ generate strongly continuous contraction semigroups by the Lumer–Phillips Theorem \cite[Theorem 1.4.3]{MR710486}. The second statement and the positivity of the  semigroups generated by  $(\Gen,\BC )$ follow by   Theorem \ref{thm:TK}, in view of the twelve theorems of Section \ref{sec:case}, Lemma \ref{Ghdissipative} and Proposition \ref{prop:cores}. The last statement is clear as the theorems of Section \ref{sec:case} for the eight cases concerned hold under \ref{H0}.
\qquad \end{proof}

\begin{corollary}\label{cor:C0semi}
Assume \ref{H0}. Then all the operators in Table \ref{explicitProcesses} generate positive strongly continuous contraction semigroups on the respective Banach spaces $X$.  
\end{corollary}
\begin{proof} For all operators in Table \ref{explicitProcesses} apart from $(\Cr,{\rm N^*D})$ and $(\RLl,{\rm N^*D})$,  apply Theorem \ref{trotterkatotypetheorem} to obtain the result for the (forward) operators  $(\Gen^+,LR)$. Then the ``if'' direction of Corollary \ref{cor:duality} proves the statement for the respective (backward) operators  $(\Gen^-,LR)$. To obtain the result for $(\Cr,{\rm N^*D})$ and $(\RLl,{\rm N^*D})$, recall that by Remark \ref{rmk:knowntwosided}    $(\Cr,{\rm N^*D})$ generates a positive strongly continuous contraction semigroup, so then we can conclude with the   ``only if'' direction of Corollary \ref{cor:duality}.   
 
\end{proof}
\begin{remark} In Corollary \ref{cor:C0semi}, to cover under \ref{H0} the operators $(\Cr,{\rm N^*D})$ and $(\RLl,{\rm N^*D})$, we relied on the known results mentioned in Remark \ref{rmk:knowntwosided}. This is because $(\Cr,{\rm N^*D})$ and $(\RLl,{\rm N^*D})$ are the only pair of dual operators that required \ref{H1} in Theorem \ref{trotterkatotypetheorem} (see Table \ref{tab:Hs}).
\end{remark}
\begin{corollary}[Duality]\label{cor:adjointsemi}  
Assume \ref{H0} and recall Corollary \ref{cor:C0semi}.  Then the restriction to $L^1[-1,1]$ of the dual semigroup generated by $(\Gen^-,LR)$ equals the semigroup generated by $(\Gen^+,LR)$. Conversely, the restriction to $C_0(\Omega)$ of the dual semigroup generated by $(\Gen^+,LR)$ equals the semigroup generated by $(\Gen^-,LR)$.
\end{corollary}
\begin{proof} This is immediate from    Corollaries \ref{cor:C0semi} and   \ref{cor:duality}. 
 
\end{proof}


\section*{Acknowledgments} We would like express our gratitude to Florin Avram, Andreas Kyprianou and Christian Lubich for several interesting discussions and their crucial help with finding relevant literature for this work.

\appendix
\section*{Appendix}
\section{Proofs related to assumption \ref{H1}-(i)}\label{app:y^2ker2}
 
\begin{proof}[of Proposition \ref{prop:y^2ker2} under \ref{H1}-(i)] Define the Fourier transform $\mathcal Ff(x)=\widehat f(ix)/\sqrt{2\pi}$ and its inverse $\mathcal F^{-1}f(x)=\mathcal Ff(-x)$. Also, by distribution we always mean tempered distribution and by $L^1(\mathbb{R})$ we mean the Banach space of complex-valued Lebesgue integrable functions on $\mathbb R$. We first claim that for each $c>0$, $x\mapsto(\xi^2/\sym(\xi))''\in L^1(\mathbb{R})$, where $\xi=c+ix$, so that
\[
w_c:=\mathcal F^{-1}[(\xi^2/\sym)'']\in C_0(\mathbb{R}). 
\]
 First observe that \ref{H1}-(i) implies that
  $x\mapsto 1/\sym(c+ix)\in L^1(\R)$.
A direct calculation yields
\begin{equation}\label{eq:regbound}
\begin{split}
\frac{\dd^2}{\dd\xi^2}\frac{\xi^2}{\sym(\xi)}= \frac1{\sym(c+ix)}&\left(2-4\frac{(c+i x)\sym'(c+ix)}{\sym(c+ix)}\right. \\
& \left. +2\left(\frac{(c+ix)\sym'(c+ix)}{\sym(c+ix)}\right)^2-\frac{(c+ix)^2\sym''(c+ix)}{\sym(c+ix)}\right).
\end{split}
\end{equation}
By assumption $\Phi$ is 2-regular, i.e. 
$$|\zeta\widehat\Phi'(\zeta)|+|\zeta^2\widehat\Phi''(\zeta)|\le M|\widehat\Phi(\zeta)|$$
for some $M>0$ and all $\Re(\zeta)>0$. As $\widehat\Phi(\zeta)=\sym(\zeta)/\zeta^2$, we have
$$|\zeta||\sym'(\zeta)/\zeta^2-2\sym(\zeta)/\zeta^3|\le M|\sym(\zeta)/\zeta^2|,$$
or 
\begin{equation}\label{eq:2mon}
|\zeta\sym'(\zeta)/\sym(\zeta)|\le M+2.
\end{equation} 
 Similarly,
\begin{equation}\label{eq:3mon}|\zeta^2\sym''(\zeta)/\sym(\zeta)|\le 5M+14.
\end{equation} 
The claim is now proved because the second factor on the right hand side of \eqref{eq:regbound} is bounded and
  $x\mapsto 1/\sym(c+ix)\in L^1(\R)$. \\
  
  We know that in the sense of distributions $\mathcal F w_c(x) = (\xi^2/\sym)''$, and so it remains to show that  $e^{cy}w_c(y)/\sqrt{2\pi}$ is independent of $c$, supported on $[0,\infty)$ and of sub-exponential growth. Then, recalling that $\xi=c+ix$ and using Lemma \ref{thm:ker} in the third identity,   in the sense of distributions, 
\begin{align}\nonumber
\mathcal F w_c(x) &= (\xi^2/\sym)'' =2\left(\xi/\sym\right)'+\xi(\xi/\sym)''\\ \nonumber
&=-2\int_0^\infty e^{-(c+ix)y}yk_{-1}(y-1)\,\dd y +(c+ix)\int_0^\infty e^{-(c+ix)y}y^2k_{-1}(y-1)\,\dd y\\ \label{eq:FTterm}
&=\mathcal F [\sqrt{2\pi}e^{-cy}(-2+c y)yk^+_{-1}(y-1)](x)+ix\mathcal F G(x),
\end{align}
and observe that the first Fourier transform in \eqref{eq:FTterm} is a   distribution. Also, the    second term in \eqref{eq:FTterm} equals the   distribution $\mathcal F G'$, where the   distribution $G'$ is the distributional derivative of the    distribution  $G(y)=\sqrt{2\pi}e^{-cy} y^2k_{-1}^+(y-1)$. Indeed, if $s$ is a Schwartz function, using standard properties of   distributions \cite[Chapter 7]{MR1157815}
\begin{align*}
ix\mathcal F G(s)&=\mathcal F G(ix \mathcal F^{-1}\mathcal Fs)=-\mathcal F G( \mathcal F^{-1}(\mathcal Fs)')=- G( (\mathcal Fs)')=G '(\mathcal Fs)= \mathcal FG'(s).
\end{align*}
Then,  by uniqueness of Fourier transforms of   distributions  \cite[Theorem 7.15]{MR1157815}
\[
w_c = [\sqrt{2\pi}e^{-cy}(-2+c y)yk^+_{-1}(y-1)]+G'.
\]
This implies that $w_c$ 
is   supported on $[0,\infty)$,  because $G'(s)=-\int_{\mathbb R} G (y)s'(y)\,\dd y=-\int_{[0,\infty)} G(y)s'(y)\,\dd y$ is supported on $[0,\infty)$,
and also $G'$ allows the continuous  density
\[
G'(y)=\sqrt{2\pi}e^{-cy}(2-cy)yk_{-1}^+(y-1)+w_c(y).
\]
Therefore $G\in C^1[0,\infty)$, 
and for  every  $x>0$
\begin{align*}
x^2k_{-1}(x-1)=e^{cx}\frac{G(x)}{\sqrt{2\pi}}=\int_0^xe^{c(x-y)}(2-cy)yk_{-1}(y-1)+e^{cx}\frac{w_c(y)}{\sqrt{2\pi}}\,\dd y.
\end{align*}
Then $x\mapsto x^2k_{-1}(x-1)\in C^{1}[0,\infty)$ 
with  derivative
\begin{align*}
\frac{\dd}{\dd x}x^2k_{-1}(x-1)&=cx^2k_{-1}(x-1) + (2-cx)xk_{-1}(x-1)+e^{cx}\frac{w_c(x)}{\sqrt{2\pi}} \\
&= 2xk_{-1}(x-1) + e^{cx}\frac{w_c(x)}{\sqrt{2\pi}}.
\end{align*}
Finally, rearranging the above identity, we obtain that   $w(x):=e^{cx}w_c(x)/\sqrt{2\pi}$ for   $x>0$ is independent of $c$, $w\in C_0(0,T]$ for each $T>0$ and $w$ is of sub-exponential growth because $x\mapsto e^{-cx} w(x)=w_c(x)/\sqrt{2\pi}\in C_0(\mathbb R)$  for all $c>0$.  Thus, relabelling $k_{-2}(x-1)=w(x)/x^2$ for $x>0$ and $ k_{-2}^+=\Pi^{-1}_{(-\infty,-1)}k_{-2}$,   we obtained \eqref{eq:y^2k-2lap}. \\
\end{proof}

\begin{proof}[of inequality \eqref{eq:Abelianest}]
From the stable-like assumption we have that for all $\epsilon>0$ there exists a $\delta>0$ such that   $ \frac{C}{\alpha-1}(1-\epsilon)x^{1-\alpha}\le \Phi(x) \le \frac{C}{\alpha-1}(1+\epsilon)x^{1-\alpha}$ for all $x\le\delta$ and recall that $\Phi$ is non-increasing. Now
\begin{equation*}
    \begin{split}
        |\widehat\Phi(x+iy)|& \ge \max\left\{\frac1{|y|}\left|\int_0^\infty\sin(t) e^{-xt/|y|}\Phi(t/|y|)\,\dd t\right|,\left|\int_0^\infty\cos(|y|t) e^{-xt}\Phi(t)\,\dd t\right|\right\}. 
    \end{split}
\end{equation*}
Hence we can find a small $\epsilon\in(0,1)$ and a corresponding $\delta>0$ such that  for $|y|>2\pi/\delta$,
\begin{equation*}
    \begin{split}
        |\widehat\Phi(x+iy)|& \ge 
        \frac{C}{(\alpha-1)|y|}\left|(1-\epsilon)\int_0^\pi\sin(t) e^{-xt/|y|}\left(t/|y|\right)^{1-\alpha}\,\dd t\right.\\
        &+\left.(1+\epsilon)\int_{\pi}^{2\pi}\sin(t) e^{-xt/|y|}\left(t/|y|\right)^{1-\alpha}\,\dd t\right|\\
        & \ge Me^{-x\pi/|y|}|y|^{\alpha-2},
    \end{split}
\end{equation*}
for some $M>0$ dependent only on $C,\,\alpha,\epsilon$ and $\delta$. Also, we have
\begin{equation*}
    \begin{split}
        |\widehat\Phi(x+iy)|&\ge
        \left|\int_0^{\frac{\pi}{3|y|}\wedge \delta}\cos(|y|t) e^{-xt}\Phi(t)\,\dd t+\int_{\frac{\pi}{3|y|}\wedge \delta}^{\infty} \cos(|y|t)e^{-xt}\Phi(t)\,\dd t\right|\\
        & \ge  \frac{C(1-\epsilon)}{2(\alpha-1)} \int_0^\infty e^{-xt}t^{1-\alpha}\,\dd t-\int_{\frac{\pi}{3|y|}\wedge \delta}^{\infty}e^{-xt}\left(\frac{C(1-\epsilon)}{2(\alpha-1)}t^{1-\alpha}+\Phi(t)\right)\,\dd t\\
        & \ge  \frac{C(1-\epsilon)\Gamma(2-\alpha)}{2(\alpha-1)} x^{\alpha-2}-C\frac{3+\epsilon}{2(\alpha-1)}\left(\frac{\pi}{3|y|}\wedge\delta\right)^{1-\alpha}\int_{\frac{\pi}{3|y|}\wedge\delta}^{\infty}e^{-xt}\,\dd t\\
        & \ge  \frac{C(1-\epsilon)\Gamma(2-\alpha)}{2(\alpha-1)} x^{\alpha-2}-C\frac{3+\epsilon}{2(\alpha-1)}\left(\frac{\pi}{3|y|}\wedge\delta\right)^{1-\alpha}x^{-1}\\
          & = \frac{C}{2(\alpha-1)} \frac1x\left((1-\epsilon)\Gamma(2-\alpha)x^{\alpha-1}-(3+\epsilon)\left(\frac{\pi}{3|y|}\wedge\delta\right)^{1-\alpha}\right)\\
        & \ge \frac{C(1-\epsilon)\Gamma(2-\alpha)}{4(\alpha-1)} x^{\alpha-2}
    \end{split}
\end{equation*}
where the last inequality holds for all $\beta x\ge |y|\vee 1$, where $\beta=\beta_1\wedge \beta_2 $ depends only on $C,\,\alpha,\,\epsilon$ and $\delta,$ where $\beta_1$ and $\beta_2$ are chosen as follows. Let $\beta_1=( \frac{(1-\epsilon)\Gamma(2-\alpha)}{2(3+\epsilon)(3/\pi)^{\alpha-1}})^{\frac{1}{\alpha-1}}$, so  that if $\frac{\pi}{3|y|}\le \delta$, it holds that \begin{align*}
         & (3+\epsilon)(3/\pi)^{\alpha-1}|y|^{\alpha-1}\le\frac12 (1-\epsilon)\Gamma(2-\alpha)x^{\alpha-1} \iff 
  |y|\le \beta_1 x,
\end{align*}
meanwhile let $\beta_2=( \frac{(1-\epsilon)\Gamma(2-\alpha)}{2(3+\epsilon)\delta^{1-\alpha}})^{\frac{1}{\alpha-1}}$,  so  that if $\frac\pi{3|y|}>\delta$
\begin{align*}
         & (3+\epsilon) \delta^{1-\alpha}\le\frac12 (1-\epsilon)\Gamma(2-\alpha)x^{\alpha-1}\iff
  1\le \beta_2 x.
\end{align*}
Hence we proved that for $z=x+iy$  \begin{equation}   \label{eq:bourdsbetaxy}
    |\sym(z)|=|z^2\widehat\Phi(z)|\ge \begin{cases}|z|^2Me^{-x\pi/|y|}|y|^{\alpha-2}\ge b|z|^\alpha,& \mbox{ for } |y|\ge\beta x\vee 2\pi/\delta,\\
|z|^2\widetilde Mx^{\alpha-2}\ge b |z|^\alpha, &\mbox {for }\beta x\ge |y|\vee 1,\end{cases}
\end{equation}
for some $\widetilde M,b>0$ that depend only on $C,\,\alpha,\,\epsilon$ and $\delta$. To conclude we show that $|\sym(z)|\ge b|z|^\alpha$, if $|z|\ge \sqrt{2}d$ with $\Re z>0$ where $d=\max\{2\pi/\delta,2\pi/(\delta\beta),1,1/\beta\}$. First observe that it must be that $x\vee |y|\ge d$. So, if $x\wedge |y|\ge d $, then $\beta x\ge 2\pi/\delta$ and $|y|\ge 1$ and we can apply \eqref{eq:bourdsbetaxy}. If $  x\ge d>|y| $, then: if $\beta\ge 1$   it must be that $\beta x\ge |y|\vee 1$, and so we can apply \eqref{eq:bourdsbetaxy};   if $\beta ,|y|< 1$, then using $\beta x\ge 1$ we can apply \eqref{eq:bourdsbetaxy}; and if $\beta < 1\le |y|$, observing that $\beta x\ge 2\pi/\delta$, we can apply \eqref{eq:bourdsbetaxy}. The remaining scenario is that  $x< d\le |y|$, so if $\beta x\ge |y|$ we can apply \eqref{eq:bourdsbetaxy} as $|y|\ge1$, otherwise $|y|> \beta x$ but we also have $|y|\ge2\pi/\delta$ and so \eqref{eq:bourdsbetaxy} still applies.
\end{proof}

\begin{proposition}\label{prop:H1O(h)}
Let $f\in C^2(0,\infty)$ be 2-regular with $w(\xi)=\int_0^\infty e^{-\xi x} f(x)\,\dd x$ such that $\int_1^\infty \bar w_{\frac12}(r)\,\dd r<\infty$, where
\[\bar w_{\frac12}(r):=\sup\{|\xi w(\xi)|:|\xi|>r,\, \Re\xi>2^{-1}\}/r.\]
  Let $\calG_{n,h}$ be such that
$$w\left(\frac{1-e^{-\xi h}}h\right)=\sum_{n=0}^\infty \calG_{n,h} e^{-\xi nh}$$ and define   the linear interpolation of $n\mathcal  G_{n,h}$ at the grid points $nh$ via
$$F_h(x):=(x/h-n)(n+1)\calG_{n+1,h}+(n+1-x/h)n\calG_{n,h}$$ for $n=\lfloor x/h\rfloor$. 
Then as $h\to0$
$$\frac{F_h(x)- xf(x)}{h}\to -\frac{\dd}{\dd x}xf(x)+\frac12\frac{\dd^2}{\dd x^2}(x^2f(x)),$$
uniformly in $[0,C]$ for any $C>0$. In particular, 
\[\frac{\frac1h\calG_{n,h}- f(x)}h\to\left.\left( -\frac{\dd}{\dd x}xf(x)+\frac12\frac{\dd^2}{\dd x^2}(x^2f(x))\right)\right/x \]for $x=nh$.
\end{proposition}

\begin{proof} First observe that the assumptions on $f$ imply that
 \begin{equation}\label{eq:convlambdawint}\bar w(r):=\sup\big\{|w(\xi)|:|\xi|>r,\, \Re\xi>2^{-1}\big\}\,\,\,\,\text{satisfies}\, \int_1^\infty \bar w (r)\,\dd r<\infty
\end{equation}
   and  \begin{equation}\label{eq:convlambdaw}\lim_{r\to\infty}\sup\{|\xi w(\xi)|:|\xi|\ge r,\,  \Re\xi\ge 2^{-1} \}=0.
\end{equation}  
Also, as $-e^{-\xi h}w'\left(\frac{1-e^{-\xi h}}h\right)=-\frac{d}{d\xi}w\left(\frac{1-e^{-\xi h}}h\right)=\sum_{n=0}^\infty nh\calG_{n,h} e^{-\xi nh}$, multiplication with $\frac{e^{\xi h}}{h^2}\left(\frac{1-e^{-\xi h}}{\xi}\right)^2$, the Laplace transform of the hat function centred at zero of width $2h$ and height $1/h$, gives the Laplace transform of $F_h$ (extended to 0 on $(-\infty,0)$):
$$\widehat F_h(\xi)=-\left(\frac{1-e^{-\xi h}}{\xi h}\right)^2w'\left(\frac{1-e^{-\xi h}}{ h}\right).$$

We are going to show that the Laplace transform of $ F_h(x)$ converges  to the Laplace transform $ xf(x)$ in $L^1(1+i\mathbb R)$ with the correct first order error term, yielding the desired result by Laplace inversion. That is, we are going to show that as $h\to 0$
\[
\int_{1+i\R}\left|\frac{\widehat F_h(\xi)+w'(\xi)}h -\xi w'(\xi)-\frac12\xi^2w''(\xi)\right|\,\dd \xi\to 0.
\]

First note that the Taylor expansion of $h\mapsto\widehat F_h(\xi)$ around $h=0$ for each $\xi$ shows that the integrand converges pointwise to zero.  Indeed $\widehat F_0(\xi)= -w'(\xi)$, and  $\frac{\dd}{\dd h}\widehat F_0(\xi)$ is obtained from \eqref{eq:dxFh}. 
However, there is no dominating integrable function. We therefore consider two cases: $\xi=1+ik$ with $|k|\le \frac{\pi}{2h}$ and $|k|>\frac{\pi}{2h}$ and show that the respective integrals converge to zero. 

Let $|k|\le\frac{\pi}{2h}$ and recall that from now on $\xi=1+ik$. Using Taylor, for each $\xi$
\[
\widehat F_h(\xi)+w'(\xi)=h\frac{\partial}{\partial \tilde h}\widehat F_{\tilde h }(\xi)\quad \text{for some $0<\tilde h  \le h$.}\]  
Now,
\begin{equation}\label{eq:dxFh}
\begin{split}
\frac{\dd}{\dd h}\widehat F_h(\xi)=&\left(\frac{1-e^{-\xi h}}{ h}\right)w'\left(\frac{1-e^{-\xi h}}{ h}\right)\frac{e^{h\xi}-1-h\xi}{h^2\xi^2/2}e^{-h\xi}\\
&+\frac12\left(\frac{1-e^{-\xi h}}{ h}\right)^2w''\left(\frac{1-e^{-\xi h}}{ h}\right) \frac{e^{h\xi}-1-h\xi}{h^2\xi^2/2}e^{-h\xi}.
\end{split}
\end{equation}

As $f$ is 2-regular, there exists a constant $M$ such that for all $\Re \zeta> 0$
$$|\zeta w'(\zeta)|+|\xi^2 w''(\zeta)|\le M |w(\zeta)|.
$$ 
As   for all $ |k|\le \frac{\pi}{2h}$ and $h$ small we have \[\Big|\frac{1-e^{-\tilde h  (1+ik)}}{\tilde h }\Big|\ge\Big| \frac{\sin(\tilde h  k)}{\tilde h  }\Big|e^{-\tilde h }\ge\Big| k\Big(1-\frac{\tilde h  ^2k^2}6\Big)\Big|\frac12\ge \Big| k\Big(1-\frac{\pi^2}{24}\Big)\Big|\frac12\] and   $\Re\frac{1-e^{-\tilde h  (1+ik)}}{\tilde h  }\ge  \frac{1-\cos(\tilde h  k)e^{-\tilde h } }{\tilde h  }  \ge\frac12$, we obtain for all $|k|\le \frac{\pi}{2h}$
\begin{align*}
    \left|\frac{d}{d\tilde h}\widehat F_{\tilde h }(1+ik)\right|&\le M' w\Big(\frac{1-e^{-\tilde h  (1+ik)}}{\tilde h }\Big)\\
    &\le  M'\sup_{\substack{\zeta:\,\Re \zeta\ge1/2 \,\,\text{and}\\ \,\,\,\,|\zeta|\ge |k|(1-\pi^2/24)\frac12}}| w(\zeta)|= M'\bar w( |k|(1-\pi^2/24)2^{-1}),
\end{align*}
 for some $M'>0$, independent of $\tilde h$ and all $ |k|\le \frac{\pi}{2h}$. As $\bar w$ also dominates $|\zeta w'|$ and $|\zeta^2w''|$, by \eqref{eq:convlambdawint} and the Dominated Convergence Theorem 
$$\int_{1-i\frac{\pi}{2h}}^{1+i\frac{\pi}{2h}}\left|\frac{\widehat F_h(\xi)+w'(\xi)}h -\xi w'(\xi)-\frac12\xi^2w''(\xi)\right|\,\dd \xi\to 0.$$

In order to tackle the integral for $|k|>\pi/2h$, note that for some $C>0$,
$$\left|\int_{1+i\frac{\pi}{2h}}^{1+i\infty}\left|\frac{w'(\xi)}h -\xi w'(\xi)-\frac12\xi^2w''(\xi)\right|\,\dd \xi\right|\le C \int_{\frac{\pi}{2h}}^\infty\bar w(k)\,\dd k\to 0.$$
So it remains to show that 
$$\left|\int_{1+i\frac{\pi}{2h}}^{1+i\infty}\left|\frac{\widehat F_h(\xi)}h \right|\,\dd\xi\right|=\left|\int_{1+i\frac{\pi}{2h}}^{1+i\infty}\left|\frac1{h\xi^2}\left(\frac{1-e^{-\xi h}}{ h}\right)^2w'\left(\frac{1-e^{-\xi h}}{ h}\right)\right|\,\dd \xi\right|\to 0.$$

Consider for $|k|>\frac{\pi}{2h}$
$$G(k,h):=\frac C{hk^2}\begin{cases} \sup_{\Re\zeta\ge 2^{-1}} |\zeta w(\zeta)| & 2n\pi-\sqrt h<h|k|<2n\pi+\sqrt h; n\in\mathbb N,\\ \\
g(h)& \mbox{else,}\\
\end{cases}
$$
for some constant $C>0$, where 
\[
g(h):=\sup \left\{|\tilde\xi w(\tilde\xi)|:
\,1/\sqrt{h}<k<2\pi/h-1/\sqrt{h}\right\}\quad  \text{for} \quad\tilde \xi :=\frac{1-e^{-(1+ik) h}}{ h}.
\]

Then $|\widehat F(\xi)/h|\le G(k,h)$ for any $C$ large enough, using 1-regularity of $f$. Furthermore,  if $1/\sqrt{h}<k<\pi/2h$ or $3\pi/2h<k<2\pi/h-1/\sqrt{h}$, $\sin(hk)/h\to\infty$, and for $\pi/2h<k<3\pi/2h$, $(1-\cos(hk)e^{-h})/h\to\infty$. Hence as $h\to 0$, 
\[
\inf_{1/\sqrt{h}<k<2\pi/h-1/\sqrt{h}}|\tilde \xi|\to\infty\quad\text{and}\quad \inf_{1/\sqrt{h}<k<2\pi/h-1/\sqrt{h}}\Re \tilde \xi\ge1/2,
\]
  and therefore $g(h)\to 0$ as $h\to 0$ by \eqref{eq:convlambdaw}. 

 Note that \eqref{eq:convlambdaw} also implies $\sup_{\Re\zeta\ge 2^{-1}} |\zeta w(\zeta)|<\infty$ and observe that 
\[
\sum_{n=1}^\infty\int_{2n\pi/h-1/\sqrt{h}}^{2n\pi/h+1/\sqrt{h}}\frac1{hk^2}\,\dd k=\sum_{n=1}^\infty\frac1{2n\pi-\sqrt{h}}-\frac1{2n\pi+\sqrt{h}}=\sum_{n=1}^\infty\frac{2\sqrt{h}}{4n^2\pi^2-h}\to0.
\]
 Then we can conclude by  $$\int_{\frac{\pi}{2h}}^{\infty}\left|\frac{\widehat F_h(1+ik)}h \right|\,\dd k\le\int_{\frac{\pi}{2h}}^\infty G(k,h)\,\dd k\to 0.$$
\end{proof}

\section{Proof of Lemma \ref{lem:firtorderapprox}}\label{app:firtorderapprox}

\begin{proof}[of Lemma \ref{lem:firtorderapprox}] We sometimes  simplify notation by writing $[\sym/\zeta ](\xi)=\sym(\xi)/\xi$ when    $\xi\in\mathbb 
C$ has a lengthy expression.
 Denote again by $g$ the canonical extension of $g$ to $\mathbb R$ (Remark \ref{rmk:canon}), so that
$$
 \Conelhzero \Ir  g( x)=\sum_{j=0}^{\infty}  \Gruone_j \Il g(x-jh)= \sum_{j=0}^{\infty}  \Gruone_j f(\tilde x-jh)= \Conelhzero f(\tilde x),
 $$ 
 with  $\tilde x=x+1$, where $f(\cdot)=\Il g(\cdot-1)\in C[0,\infty)$ and $f=0$ on $(-\infty,0]$,  so that $ \Conelhzero f\in C[0,\infty)$ with $ \Conelhzero f=0$ on $(-\infty,0]$.  Moreover,  $f\in C^6[0,\infty)$, with $f^{(j)}(0)=0$ for $0\le j\le 6$, implying  
 \[
 \left| \widehat f(\xi)\right|\le \frac C{|\xi|^6},\quad\text{for some $C>0$.}
 \] 
  Then we fix a $c>0$ and we can apply the Laplace inversion theorem \cite[Theorem II.7.3]{MR0005923} because of the integrability on the complex vertical line over $c$ due to
\begin{align*}
 \Conelhzero f(\tilde x)&=\frac1{i2\pi}\int_{c-i\infty}^{c+i\infty}e^{ \xi\tilde x} \widehat{ \Conelhzero f}(\xi)\,\dd\xi =\frac1{i2\pi}\int_{c-i\infty}^{c+i\infty}e^{ \xi \tilde x} [\sym/\zeta]\left(\frac{1-e^{-\xi h}}{h}\right)\widehat{f}(\xi)\,\dd\xi,
\end{align*}
using $| [\sym/\zeta](\frac{1-e^{-\xi h}}{h})|=O(1)$ as  $|k|\to\infty$ with $ \xi=c+ik$, which follows easily from $\sym(\zeta)/\zeta=\zeta\widehat\Phi(\zeta)$ and \eqref{eq:convhomega}. \\
We now obtain a second order Taylor expansion of $\sym(\zeta)/\zeta$ around $\xi$ with the exact first order Taylor expansion term. We use the Taylor's formula 
 \[
w(b)-w(a)=(b-a)w'(a)+(b-a)^2\int_0^1 (1-z)w''(l(z))\,\dd z, 
\] for an analytic $w$ and the path $ l(t)=a+(b-a)t$, $a,b\in\mathbb{C}$ inside the domain of $w$.  As $\sym(\zeta)/\zeta$ is analytic on the right half plane, $[\sym/\zeta]''=-\int_0^\infty e^{-\zeta y}y^2\phi(y,\infty)\,\dd y$ is bounded on $\mathbb C\cap \{\Re\zeta>\epsilon\}$ for any $\epsilon>0$, $\Re \xi= c>0$ and 
\[
\Re \frac{1-e^{-\xi h}}{h} = \frac1h\left(1-e^{-ch}\cos(kh)\right)\ge \inf_{h\in(0,1]}\frac{1-e^{-ch}}{h}=:\epsilon>0, 
\] then
we obtain by Taylor's formula for all $k\in \mathbb{R},\,h\in(0,1]$ (recalling that $ \xi=c+ik$) 
\begin{align*}
[\sym/\zeta]\left(\frac{1-e^{-\xi h}}{h}\right)&=\sym(\xi)/\xi +\left(\frac{1-e^{-\xi h}}{h}-\xi\right) [\sym(\xi)/\xi]'  + \left|\frac{1-e^{-\xi h}}{h}-\xi\right|^2O (1)\\
&=\sym(\xi)/\xi +\left(-\frac{\xi^2h}2  +h^2O(|\xi|^3)\right) [\sym(\xi)/\xi]'  +  h^2O(|\xi|^4)\\
&=\sym(\xi)/\xi -\frac{h\xi^2}2 [\sym(\xi)/\xi]' + h^2O(|\xi|^4),
\end{align*} 
where $O(\cdot)$ is independent of $h\in(0,1]$ and understood as $|k|\to\infty$,   and we used $\sup_k |[\sym(\xi)/\xi]'| <\infty$,  which is a consequence of \eqref{eq:omegaoverxi'}.
   Then
\begin{align*}
 \Conelhzero f(\tilde x)&=\frac1{i2\pi}\int_{c-i\infty}^{c+i\infty}e^{ \xi \tilde x} \left(\left(\sym(\xi)/\xi -\frac{h\xi^2}2 [\sym(\xi)/\xi]' \right)\widehat{f}(\xi) +h^2 O(|\xi|^4)\widehat{f}(\xi)\right)\,\dd\xi.\\
&=\frac1{i2\pi}\int_{c-i\infty}^{c+i\infty}e^{ \xi \tilde x} \left(\sym(\xi)/\xi -\frac{h\xi^2}2 [\sym(\xi)/\xi]' \right)\widehat{f}(\xi)\,\dd\xi+O(h^2),
\end{align*}
and so, using $\widehat f(\xi)=\widehat{\tilde g}(\xi)/\sym (\xi)$ where  $\tilde g(\cdot)=g(\cdot-1)$,   
\begin{align*}
 \Conelhzero f(\tilde x)
&=\frac1{i2\pi}\int_{c-i\infty}^{c+i\infty}e^{ \xi \tilde x} \left[\frac{\widehat{\tilde g}(\xi)}{\xi} - \frac{h\xi }2 \frac{[\sym(\xi)/\xi]'}{\sym(\xi)/\xi}\widehat{\tilde g}(\xi)\right] \,\dd\xi+O(h^2)\\
&=I_+g(x) -\frac h2\frac1{i2\pi}\int_{c-i\infty}^{c+i\infty}e^{ \xi \tilde x}  [\sym(\xi)/\xi]'  \frac{\xi}{\sym(\xi) } \widehat{\tilde g'}(\xi) \,\dd\xi+O(h^2)\\
&=I_+g(x)-\frac h2[ y\phi(y,\infty)]^+\star k_{-1}^+(y-1) \star  g'( x)+O(h^2),
\end{align*}
where the second last identity follows from $\widehat{\tilde g}(\xi)/\xi=\widehat{I_+\tilde g}(\xi)$ and $I_+\tilde g(\tilde x)=I_+g(x)$, meanwhile the last identity is another application of \cite[Theorem 7.3]{MR0005923}. Namely, we use Lemma \ref{thm:ker}, the identity
\begin{align}\label{eq:omegaoverxi'}
[\sym(\xi)/\xi]'&= \left(\int_0^\infty (1-e^{-\xi y})\phi(y,\infty)\,\dd y\right)'= \int_0^\infty e^{-\xi y}y\phi(y,\infty)\,\dd y,
\end{align}
 and that
\begin{align*}
[ y\phi(y,\infty)]^+\star k_{-1}^+(\cdot-1) \star  \tilde g'(\tilde x)&= \int_0^{\tilde x}\tilde g'(\tilde x-z)[ y\phi(y,\infty)]^+\star k_{-1}^+(\cdot-1)(z)\,\dd z\\
&= \int_0^{\tilde x}  g'(\tilde x-z-1)[ y\phi(y,\infty)]^+\star k_{-1}^+(\cdot-1)(z)\,\dd z\\
&= \int_0^{ x+1}  g'(  x-z)[ y\phi(y,\infty)]^+\star k_{-1}^+(y-1)(z)\,\dd z\\
&=[ y\phi(y,\infty)]^+\star k_{-1}^+(\cdot-1) \star  g'( x),
\end{align*}
belongs to $C_0^4(0,T]$ (in $\tilde x$) for each $T$ and it is of sub-exponential growth because   $g$ grows polynomially and
\begin{align*}
\int_0^{1+x}[ y\phi(y,\infty)]^+\star k_{-1}^+(\cdot-1)(z)\,\dd z
&\le k_{0}^+(x) \int_0^{1+x} z\phi(z,\infty)\,\dd z\\ &= k_{0}^+(x) \int_0^{1+x} z\int_{(z,\infty)}\phi(\dd y)\,\dd z\\
&\le k_{0}^+(\tilde x-1)\left(\int_{(0,1]}\frac{y^2}2\phi(\dd y)+\tilde x^2 \phi(1,\infty) \right).
\end{align*} 

 The `$-$' case follows by \eqref{eq:convolutionsign} and $ \Conerhzero [g(-\cdot)](x)= \Conelhzero g(-x)$.
\end{proof}



\section{Details for the proof of Lemma \ref{lem:varthek0}}\label{sec:app_lem:varthek0}

\begin{proof}[of equations \eqref{eq:partialphithetah0}-\eqref{eq:diffvarth}]
We rewrite each   function to recover a convolution  term (to apply Lemma \ref{lem:convid}) plus a remainder.  We repeatedly use Lemma \ref{lem:convid} without mention and we recall Remark \ref{rmk:Dphi_h}.   To prove \eqref{eq:partialphithetah0} we compute   
\begin{align*}
\Crh \vartheta^{k_{0}}_{\mh}(x)&=\frac1h\sum_{j=0}^{\iota(-x)} \Gru_{j} \left( \lambda\calG^{k_{0}}_{\iota(-(x+(j-1)h))-2}+\barlambda \calG^{k_{0}}_{\iota(-(x+(j-1)h))-1}\right)\\
&=\frac1h\sum_{j=0}^{\iota(-x)} \Gru_{j} \left( \lambda\calG^{k_{0}}_{\iota(-x)-j-1}+\barlambda \calG^{k_{0}}_{\iota(-x)-j}\right)\\
&=\frac\lambda h\sum_{j=0}^{\iota(-x)-1} \Gru_{j}  \calG^{k_{0}}_{\iota(-x)-j-1}+\frac{\barlambda }h\sum_{j=0}^{\iota(-x)} \Gru_{j} \calG^{k_{0}}_{\iota(-x)-j}\\
&=0,
\end{align*}
 where we use $\iota(-x)\ge2$ and $\calG^{k_{0}}_{-1}=0$. Similarly, to prove \eqref{eq:partialphithetah1} we compute 
\begin{align*}
h\Crh \vartheta^{k_{1}}_{\mh}(x)&=h\sum_{j=0}^{\iota(-x)} \Gru_{j}\vartheta^{k_{1}}_{\mh}(x+(j-1)h)\\ 
&=\sum_{j=0}^{\iota(-x)-1} \Gru_{j} \left( \lambda\calG^{k_{1}}_{\iota(-(x+(j-1)h))-3}+\barlambda \calG^{k_{1}}_{\iota(-(x+(j-1)h))-2}\right) -\lambda \Gru_{\iota(-x)}\calG^{k_1}_0\\
&=\lambda\sum_{j=0}^{\iota(-x)-2} \Gru_{j}  \calG^{k_{1}}_{\iota(-x)-2-j} + \barlambda \sum_{j=0}^{\iota(-x)-1} \Gru_{j}\calG^{k_{1}}_{\iota(-x)-1-j} -\lambda \Gru_{\iota(-x)}\calG^{k_1}_0\\
&=h -\lambda \Gru_{\iota(-x)}\calG^{k_1}_0.
\end{align*} 
To prove \eqref{eq:partialphi-1thetah1} we   compute 
\begin{align*} 
 \Conerhzero \vartheta^{k_{1}}_{\mh}(x)&= \sum_{j=0}^{\iota(-x)-1}\Gruone_j \vartheta^{k_{1}}_{\mh}(x+jh) \\
&= \sum_{j=0}^{\iota(-x)-2}\Gruone_j \vartheta^{k_{1}}_{\mh}(x+jh) +\Gruone_{\iota(-x)-1} \vartheta^{k_{1}}_{\mh}(x+(\iota(-x)-1)h) \\
&= \frac1h\sum_{j=0}^{\iota(-x)-2}\Gruone_j \left(\lambda \calG^{k_1}_{ \iota(-x-jh)-3} +\barlambda  \calG^{k_1}_{\iota(-x-jh)-2} \right) -\frac{\lambda}h\calG^{k_1}_0 \Gruone_{\iota(-x)-1} \\
&= \frac1h\sum_{j=0}^{\iota(-x)-2}\Gruone_j \left(\lambda \calG^{k_1}_{ \iota(-x)-3-j} +\barlambda  \calG^{k_1}_{ \iota(-x)-2-j} \right) -\frac{\lambda}h\calG^{k_1}_0 \Gruone_{\iota(-x)-1} \\
&= \frac\lambda h\sum_{j=0}^{\iota(-x)-3}\Gruone_j \lambda \calG^{k_1}_{\iota(-x)-3-j} +\frac{\barlambda }h\sum_{j=0}^{\iota(-x)-2}\Gruone_j  \calG^{k_1}_{\iota(-x)-2-j} -\frac{\lambda}h\calG^{k_1}_0 \Gruone_{\iota(-x)-1} \\
&=\frac{\lambda}h(\iota(-x)-2)h^2 + \frac{\barlambda }h(\iota(-x)-1)h^2-\frac{\lambda}h\calG^{k_1}_0 \Gruone_{\iota(-x)-1}.
\end{align*}
 To prove \eqref{eq:partialphi-1thetah0} we   compute 
\begin{align*}
h \Conerhzero \vartheta^{k_0}_{\mh}(x)&=\sum_{j=0}^{\iota(-x)-1} \Gruone_j\left(\lambda \calG^{k_0}_{ \iota(-x-jh)-2} +\barlambda  \calG^{k_0}_{\iota(-x-jh)-1} \right)\\
&=\sum_{j=0}^{\iota(-x)-1} \Gruone_j\left(\lambda \calG^{k_0}_{\iota(-x)-2-j} +\barlambda  \calG^{k_0}_{\iota(-x)-1-j} \right)\\
&=\lambda\sum_{j=0}^{\iota(-x)-2} \Gruone_j \calG^{k_0}_{\iota(-x)-2-j}+\barlambda \sum_{j=0}^{\iota(-x)-1} \Gruone_j\calG^{k_0}_{\iota(-x)-1-j} \\
&=h.
\end{align*} 
To prove \eqref{eq:partialphithetah0+} we   compute
\begin{align*}
\Clh \vartheta^{k_{0}}_{\ph}(x)&=\sum_{j=0}^{\iota(x)} \Gru_{j} \vartheta^{k_{0}}_{\ph}(x-(j-1)h)\\
&=\sum_{j=0}^{\iota(x)-1} \Gru_{j} \vartheta^{k_{0}}_{\ph}(x-(j-1)h) +\Gru_{\iota(x)} \vartheta^{k_{0}}_{\ph}(x-(\iota(x)-1)h) \\
&=\frac1h\sum_{j=0}^{\iota(x)-1} \Gru_{j}\left(\barlambda  \calG^{k_0}_{\iota(x)-1-j}+ \lambda \calG^{k_0}_{\iota(x)-j}\right)  +\Gru_{\iota(x)} \vartheta^{k_{0}}_{\ph}(x-(\iota(x)-1)h) \\
&=\frac{\barlambda }h\sum_{j=0}^{\iota(x)-1} \Gru_{j} \calG^{k_0}_{\iota(x)-1-j}+\frac\lambda h\sum_{j=0}^{\iota(x)-1} \Gru_{j} \calG^{k_0}_{\iota(x)-j}  +\Gru_{\iota(x)} \vartheta^{k_{0}}_{\ph}(x-(\iota(x)-1)h) \\
&=-\frac\lambda h \Gru_{\iota(x)} \calG^{k_0}_{0} +\frac1h \Gru_{\iota(x)} \frac{-\Gru_0}{\Gru_1}\left(\barlambda  \calG^{k_0}_{0}+ \lambda \calG^{k_0}_{1}\right). 
\end{align*}
To prove \eqref{eq:partialphi-1thetah0+} we compute
\begin{align*}
 \Conelhzero \vartheta^{k_{0}}_{\ph}(x)&=\sum_{j=0}^{\iota(x)-2} \Gruone_j \vartheta^{k_{0}}_{\ph}(x-jh)+\Gruone_{\iota(x)-1} \vartheta^{k_{0}}_{\ph}(x-(\iota(x)-1)h)\\
&=\sum_{j=0}^{\iota(x)-2} \Gruone_j \frac1h \left(\barlambda  \calG^{k_0}_{\iota(x-jh)-2}+ \lambda \calG^{k_0}_{\iota(x-jh)-1}\right)+\Gruone_{\iota(x)-1} \vartheta^{k_{0}}_{\ph}(x-(\iota(x)-1)h)\\
&=\frac{\barlambda }h\sum_{j=0}^{\iota(x)-2} \Gruone_j   \calG^{k_0}_{\iota(x)-j-2}+\frac\lambda h\sum_{j=0}^{\iota(x)-2} \Gruone_j \calG^{k_0}_{\iota(x)-j-1}+\Gruone_{\iota(x)-1} \vartheta^{k_{0}}_{\ph}(x-(\iota(x)-1)h)\\
&=1-\frac\lambda h \Gruone_{\iota(x)-1} \calG^{k_0}_{0}+\Gruone_{\iota(x)-1} \vartheta^{k_{0}}_{\ph}(x-(\iota(x)-1)h).
\end{align*}
To prove \eqref{eq:partialphi-1thetah-1+} we compute 
\begin{align*}
 \Conelhzero \vartheta^{k_{-1}}_{\ph}(x)&=\frac{(1-\theta)}h\sum_{j=0}^{\iota(x)-1} \Gruone_j \calG^{k_{-1}}_{\iota(x-jh)-2}+\frac{\theta}h\sum_{j=0}^{\iota(x)-1} \Gruone_j \calG^{k_{-1}}_{\iota(x-jh)-1}\\
&=\frac{(1-\theta)}h\sum_{j=0}^{\iota(x)-2} \Gruone_j \calG^{k_{-1}}_{\iota(x)-2-j}+\frac{\theta}h\sum_{j=0}^{\iota(x)-1} \Gruone_j \calG^{k_{-1}}_{\iota(x)-1-j}\\
&=0.
\end{align*}
To prove \eqref{eq:partialphithetah1+} we compute 
\begin{align*}
\Clh \vartheta^{k_{1}}_{\ph}(x)&=\frac1h\sum_{j=0}^{\iota(x)-1} \Gru_j \calG^{k_1}_{\iota(x-(j-1)h)-2}-\frac{\barlambda }h\Gru_{\iota(x)} \calG^{k_1}_{0}\\
&=\frac1h\sum_{j=0}^{\iota(x)-1} \Gru_j \calG^{k_1}_{\iota(x)-1-j}-\frac{\barlambda }h\Gru_{\iota(x)} \calG^{k_1}_{0}\\
&=1-\frac{\barlambda }h\Gru_{\iota(x)} \calG^{k_1}_{0}.
\end{align*}
To prove \eqref{eq:partialphi-1thetah1+} we compute 
\begin{align*}
 \Conelhzero \vartheta^{k_{1}}_{\ph}(x)&=\frac1h\sum_{j=0}^{\iota(x)-2} \Gruone_j \calG^{k_1}_{\iota(x-jh)-2}-\frac{\barlambda }h\Gruone_{\iota(x)-1} \calG^{k_1}_{0}\\
&=\frac1h\sum_{j=0}^{\iota(x)-2} \Gruone_j \calG^{k_1}_{\iota(x)-2-j}-\frac{\barlambda }h\Gruone_{\iota(x)-1} \calG^{k_1}_{0}\\
&=(\iota(x)-1)h-\frac{\barlambda }h\Gruone_{\iota(x)-1} \calG^{k_1}_{0}.
\end{align*}
To prove \eqref{eq:partialphithetah-1+} we compute 
\begin{align*}
\Clh \vartheta^{k_{-1}}_{\ph}(x)&=\frac{(1-\theta)}h\sum_{j=0}^{\iota(x)} \Gru_j \calG^{k_{-1}}_{\iota(x-(j-1)h)-2}+\frac{\theta}h\sum_{j=0}^{\iota(x)} \Gru_j \calG^{k_{-1}}_{\iota(x-(j-1)h)-1}\\
&=\frac{(1-\theta)}h\sum_{j=0}^{\iota(x)-1} \Gru_j \calG^{k_{-1}}_{\iota(x)-1-j}+\frac{\theta}h\sum_{j=0}^{\iota(x)} \Gru_j \calG^{k_{-1}}_{\iota(x)-j}\\
&=0.
\end{align*}
To prove the first identity in \eqref{eq:diffvarth} we compute 
\begin{align*} 
\vartheta^{k_{0}}_{\mh}(x)-\vartheta^{k_{0}}_{\mh}(x-h)&=\frac1h\left(\lambda \calG^{k_0}_{\iota(-x)-2}+\barlambda  \calG^{k_0}_{\iota(-x)-1}-\lambda \calG^{k_0}_{\iota(-x+h)-2}-\barlambda  \calG^{k_0}_{\iota(-x+h)-1}\right)\\
 &=-\frac1h\left(\lambda( \calG^{k_0}_{\iota(-x)-1}-\calG^{k_0}_{\iota(-x)-2})+\barlambda ( \calG^{k_0}_{\iota(-x)}-\calG^{k_0}_{\iota(-x)-1})\right)\\ 
 &=-\left(\lambda  \calG^{k_{-1}}_{\iota(-x)-1}+\barlambda  \calG^{k_{-1}}_{\iota(-x)}\right)\\ 
 &=-\left(  \calG^{k_{-1}}_{\iota(-x)-1}+\barlambda  (\calG^{k_{-1}}_{\iota(-x)}- \calG^{k_{-1}}_{\iota(-x)-1})\right)\\ 
 &=- \calG^{k_{-1}}_{\iota(-x)-1} -\barlambda h \calG^{k_{-2}}_{\iota(-x)},
\end{align*}
where we used \eqref{eq:phi+1tophi}, and the proof of the remaining identity is similar and omitted.
\end{proof}

\section{Proofs of the theorems of Section \ref{sec:caseC}}\label{app:caseC}

\begin{table}
\centering
\vline
\begin{tabular}{l|l|l|}
  \hline
  \multicolumn{3}{c}{$X = C_0(\Omega)$ and $g \in C_c^\infty(\Omega), \; g-g(1) \in C_c[-1,1)$}  
	\vline \\
	\hline
	$(\Gen^-, \BC )$ & $f \in \mathcal{C}(\Gen^-, \BC) $ & $f_h \in C_0(\Omega)$  \\
	\hline
	$(\Cr , \mathrm{DD})$ & $\Ir g - \frac{\Ir g(-1)}{k^-_{0}(-1)} k^-_{0}$ & $ \Ir g - \frac{\Ir g(-1)}{ \vartheta^{k_{0}}_{\mh}(-1)} \vartheta^{k_{0}}_{\mh}$  \\
  \hline
 $(\Cr , \mathrm{DN})$ &  $\begin{matrix}\Ir g  -\Ir g(-1)\end{matrix}$ & $\begin{matrix}\Ir [g-g(1)]+ \frac{g(1)k_1^-(-1)}{ \vartheta^{k_1}_{\mh}(-1)} \vartheta^{k_1}_{\mh} -\Ir g(-1)\end{matrix}$\\
  \hline
  	$(\Cr , \mathrm{ND})$ &$ \Ir g  -I_-g(-1)  k^-_0$ & $ \begin{matrix} \Ir g  -b_h  \vartheta^{k_0}_{\mh}+e_h  \end{matrix} $\\
  \hline
	$(\Cr , \mathrm{NN})$ 
	&
	  $\Ir  g-\frac{I_-g(-1)}{2}k^-_1 +c$ 
	  & 
	  $ \Ir [g-g(1)] +b_h\vartheta^{k_{1}}_{\mh}+e_h +c  $ \\
  \hline
		$(\Cr , \mathrm{N^*D})$ & $\Ir  g + \frac{ [\Ir  g]'}{k_{-1}^-(-1)} k_0^-$ & $
	\Ir  g +\frac{[\Ir  g]'(-1)}{\calG^{k_{-1}}_{n}/h}\vartheta^{k_{0}}_{\mh}$\\
  \hline
		$(\Cr , \mathrm{N^*N})$ &$ \Ir  g + \frac{ [\Ir  g]'(-1)}{k_{0}^-(-1)} k_1^- + c$ & $	\Ir  [g-g(1)] +\left(g(1)+\frac{[\Ir  g]'(-1)}{\calG^{k_{0}}_{n-1}/h}\right)\vartheta_{\mh}^{k_1}+c $\\
  \hline
\end{tabular}
\caption{\label{tab:fhexplicitC}
Functions $f_h \in C_0(\Omega)$ for the approximation theorems of Section \ref{sec:caseC}, where we respectively refer to Theorems \ref{thm:conv_C_ND} and \ref{thm:conv_C_NN} for the details of $f_h$ in the cases $(\Cr , \mathrm{ND})$ and  $(\Cr , \mathrm{NN})$.}

\end{table} 
\begin{proof}[of Theorem \ref{thm:conv_C_DD}] Recall that $h(n+1)=2$.
By Lemma \ref{lem:varthek0} and the choice of $b_h$, $f_h\in C_0(-1,1) $ and so does $G^{\mathrm{DD}}_{\mh}f_h$ by Lemma  \ref{Ghdissipative}. 
Also, by Lemma \ref{lem:varthek0},   $f_h\to f$ in uniformly on $[-1,1]$. It remains to check the convergence of $G^{\mathrm{DD}}_{\mh}f_h$.
For the  convergence on $(-1+h,1-h]$, observe that by Lemma \ref{lem:varthetaDD}, $G^{\mathrm{DD}}_{\mh}f_h=G^{\mathrm{DD}}_{\mh}\Ir g=\Crh \Ir g$, so that using Corollary \ref{cor:L94}
\[
\supnorm{G^{\mathrm{DD}}_{\mh}f_h-g}{(-1+h,-h+1]}\le \|\Crh \Ir g-g\|_{C[-1,1]}\to 0, \quad \text{as }h\to 0.
\]
Meanwhile,  on $(1-h,1]$, $G^{\mathrm{DD}}_{\mh}f_h=G^{\mathrm{DD}}_{\mh}\Ir g=0=g$ for all small $h$, as a consequence of  Lemma \ref{lem:varthetaDD} and $g,\Ir g\in C_c[-1,1)$. It remains to confirm the uniform convergence on $[-1,-1+h]$ (i.e. $\iota(-x)=n+1$). Using the canonical extensions to $[-1-h,-1)$ (Remark \ref{rmk:canon}), for $\iota(-x)=n+1$,
\begin{align}\nonumber
G^{\mathrm{DD}}_{\mh}f_h(x)&=\Gru_{1} f_h(x)+D^l(\lambda)\sum_{j=2}^{n} \Gru_{j} f_h(x+(j-1)h) \\  \nonumber
&=D^l(\lambda)\sum_{j=0}^{n+1} \Gru_{j} f_h(x+(j-1)h)-D^l(\lambda)\Gru_{n+1} f_h(x+nh)\\ \nonumber
&\quad +(1-D^l(\lambda) )\Gru_{1} f_h(x)-D^l(\lambda) \Gru_{0} f_h(x- h)\\ \nonumber
&=D^l(\lambda)\Crh \Ir g(x)-D^l(\lambda)\Gru_{n+1} f_h(x+nh)\\
&\quad + (1-D^l(\lambda) )\Gru_{1} f_h(x)-D^l(\lambda) \Gru_{0} f_h(x-h), \label{eq:termsDD}
\end{align}
where we used in the last identity  equation \eqref{eq:partialphithetah0} with Remark \ref{rmk:canon_extra}.
The first term in \eqref{eq:termsDD}  converges uniformly to $g$ on $[-1,-1+h]$ by using Corollary \ref{cor:L94} and $g(-1)=0$ with
\begin{align*}
 |D^l(\lambda)\Crh \Ir g(x)-g(x)|&\le |D^l(\lambda)\Crh \Ir g(x)-D^l(\lambda)g(x)|+|D^l(\lambda)g(x)-g(x)|\\
&\le\|\Crh \Ir g-g\|_{C[-1,1]}+2\|g\|_{C[-1,-1+h]}.
\end{align*} 
 For the second term in \eqref{eq:termsDD} observe that $f_h(x+nh)=b_h\vartheta^{k_{0}}_{\mh}(x+nh)=b_hh^{-1}\lambda\calG^{k_0}_0$ as $\Ir g(x+nh)=0$ for all $h$ small, and so,  as $h\to 0$
 \[
 |D^l(\lambda)\Gru_{n+1} f_h(x+nh)|\le \sup\{ |b_h|:h\in(0,1]\}\,\Gru_{n+1}\frac{\calG^{k_0}_0}{h}\to 0,
 \]
because of \eqref{eq:GClimphi} and $\calG^{k_0}_0/h=(h\sym(1/h))^{-1}\to 0$ by \eqref{eq:convhomega}. To conclude, as $g(-1)=0$, it is enough to show that the last two terms in \eqref{eq:termsDD} vanish uniformly on $[-1,-1+h]$. We follow the strategy of \cite[Page 142]{H14}. It is immediate from the definitions in \eqref{eq:GCdef} that
\[
\Gru_{0}\calG^{k_0}_{0}=1,\quad \Gru_{1}\calG^{k_0}_{1}=-\left( \frac{\sym'(1/h)}{h\sym(1/h)}\right)^2,\quad \Gru_{0}\calG^{k_0}_{1}=  \frac{\sym'(1/h)}{h\sym(1/h)},\quad 
\Gru_{1}\calG^{k_0}_{0}=-\Gru_{0}\calG^{k_0}_{1},
\] 
 so that we can rewrite and  then    bound the (absolute value of) last two terms in \eqref{eq:termsDD} as
\begin{align}\nonumber
&\,\Big|\frac{(\lambda \Gru_{0}\calG^{k_0}_{1}+\barlambda -\lambda \Gru_{0}\calG^{k_0}_{1})\Gru_{1} f_h(x)   +\lambda \Gru_{1} f_h(x-h)}{\lambda \Gru_{0}\calG^{k_0}_{1}+\barlambda   }\Big|\\ \nonumber
=&\,\Big|\frac{ \Gru_{1}(\barlambda f_h(x)+\lambda f_h(x-h))}{\lambda \Gru_{0}\calG^{k_0}_{1}+\barlambda   }\Big|\\ \nonumber
=&\,\Big|(\barlambda f_h(x)+\lambda f_h(x-h))\frac{ -\sym'(1/h)\sym(1/h) }{\lambda \sym'(1/h)+\barlambda h\sym(1/h)  } \Big|\\ \label{eq:bound_tricky_term}
\le&\,|\barlambda f_h(x)+\lambda f_h(x-h)|\frac{ \sym (1/h) }{\lambda   } .
\end{align}
where we used $\sym(1/h),\sym'(1/h)>0$ in the   inequality. If $x=-1$, then $D^l(\lambda(x)) =f_h(x)=0$. If $x\in(-1,-1+h]$, using the canonical extension of $g$ to $[-1-h,-1)$,  $x=\lambda h-1$ and Taylor expansions around $-1$,
\begin{align*}
f_h(x)&=\Ir  g(x)+b_h\vartheta^{k_{0}}_{\mh}(x)\\
&=\Ir  g(-1)+ \lambda h[\Ir  g]'(-1)+ O((\lambda h)^2)+b_h\vartheta^{k_{0}}_{\mh}(x)\\
&=R_1(\lambda,h)+ \lambda h[\Ir  g]'(-1)+\lambda O(h^2),\\
f_h(x-h)&=\Ir  g(x-h)+b_h\vartheta^{k_{0}}_{\mh}(x-h)\\
&=\Ir  g(-1)+ (\lambda-1) h[\Ir  g]'(-1)+  O(h^2)+b_h\vartheta^{k_{0}}_{\mh}(x-h)\\
&=R_2(\lambda,h)- \barlambda  h[\Ir  g]'(-1)+ O(h^2).
\end{align*}
Recalling that $b_h=-\Ir  g(-1)/\vartheta^{k_{0}}_{\mh}(-1)$ and
\[
\vartheta^{k_{0}}_{\mh}(-1)=\frac{1}{h}\calG^{k_{0}}_{n}, \quad 
\vartheta^{k_{0}}_{\mh}(x)=\frac{1}{h}(\lambda\calG^{k_{0}}_{n-1}+\barlambda \calG^{k_{0}}_{n}), 
\]
observe that 
\begin{align*}
R_1(\lambda,h)&=\Ir  g(-1)\left(1-\frac{\vartheta^{k_{0}}_{\mh}(x)}{\vartheta^{k_{0}}_{\mh}(-1)}\right)\\
&=\Ir  g(-1)\left(1-\frac{\lambda\calG^{k_{0}}_{n-1}+\barlambda \calG^{k_{0}}_{n}}{\calG^{k_{0}}_{n}}\right)\\
&=\Ir  g(-1)\frac\lambda{\calG^{k_{0}}_{n}}\left(\calG^{k_{0}}_{n}-\calG^{k_{0}}_{n-1}\right),
\end{align*}
and similarly (recalling the canonical extension of $\vartheta^{k_{0}}_{\mh}$)
\begin{align*}
R_2(\lambda,h)&=\Ir  g(-1)\left(1-\frac{\vartheta^{k_{0}}_{\mh}(x-h)}{\vartheta^{k_{0}}_{\mh}(-1)}\right)\\
&=\Ir  g(-1)\left(1-\frac{\lambda\calG^{k_{0}}_{n}+\barlambda \calG^{k_{0}}_{n+1}}{\calG^{k_{0}}_{n}}\right)\\
&=\Ir  g(-1)\frac{\barlambda }{\calG^{k_{0}}_{n}}\left(\calG^{k_{0}}_{n}-\calG^{k_{0}}_{n+1}\right),
\end{align*}
so that 
\begin{align*}
\barlambda f_h(x)+\lambda f_h(x-h)&=\barlambda R_1(\lambda,h)+\lambda R_2(\lambda,h)+\barlambda  \lambda O(h^2)+\lambda O(h^2)
\\
&=\barlambda \lambda \Ir  g(-1)\left(\frac{2\calG^{k_{0}}_{n}-\calG^{k_{0}}_{n+1}-\calG^{k_{0}}_{n-1}}{\calG^{k_{0}}_{n}}\right)+\lambda O(h^2).
\end{align*}
If the term in brackets is $O(h^2)$, we are done as we bound the last two terms in \eqref{eq:termsDD}, using \eqref{eq:bound_tricky_term}, by $O(h^2\sym(1/h))$ as $h\to0$, which vanishes by \eqref{eq:convhomega}.
Indeed, by  \eqref{eq:phi+1tophi}, $\calG_m^{k_i}=\frac{1}{h}(\calG_m^{k_{i+1}}-\calG_{m-1}^{k_{i+1}})$ for $i\in\{-1,-2\}$ and $m\in\mathbb N$, and so by \eqref{eq:GClimk-2k0}
\begin{align*}
\frac{2\calG^{k_{0}}_{n}-\calG^{k_{0}}_{n+1}-\calG^{k_{0}}_{n-1}}{\calG^{k_{0}}_{n}}&= -h\frac{\calG^{k_{-1}}_{n+1}-\calG^{k_{-1}}_{n}}{\calG^{k_{0}}_{n}}= -h^2\frac{\calG^{k_{-2}}_{n+1}}{\calG_n^{k_{0}}}=O(h^2).
\end{align*}

\end{proof}

\begin{remark}\label{rmk:convspeedC_DD}
In the proof of Theorem \ref{thm:conv_C_DD} we proved that $\|G^{\mathrm{DD}}_{\mh}f_h-g\|_{C[-1,1]}=O(h^2\sym(1/h))$, as all   other terms are $O(h)$. To see this observe that in the proof we have  that  $\|g\|_{C[-1,-1+h]}=0$ for all small $h$, and that by \ref{H1} we can conclude that $\calG^{k_0}_0\Gru_{n+1}/h= O(1/\sym(1/h)) $, as a consequence of $\Gru_{n+1}/h\to (\phi(2-)+\phi(2+))/2$ for the non-increasing density of $\phi$  \cite[Corollary VII.3c.1]{MR0005923}.
\end{remark}

\begin{proof}[of Theorem \ref{thm:conv_C_DN}] Recall that $h(n+1)=2$. For each $h>0$, by Lemmata \ref{lem:varthek0} and \ref{Ghdissipative} $f_h,\,G^{\mathrm{DN}}_{\mh}f_h\in C_0(-1,1]$. Observe that $b_h\vartheta^{k_{1}}_{\mh}\to g(-1)k_1^-$ as $h\to0$ in $C_0(-1,1]$, by Lemma \ref{lem:varthek0}. It remains to show the convergence for $G^{\mathrm{DN}}_{\mh}f_h$. If $\iota(-x)=1$ (i.e., $x\in (1-h,1]$), then for all small $h$   
\begin{align*}
G^{\mathrm{DN}}_{\mh} f_h(x) &=\Gru_{0} ( \Ir [g-g(1)] (x-h) -\Ir [g-g(1)] (x)) + b_h\frac1h\Gru_{0}  ( \barlambda  \calG^{k_1}_0 +\lambda \calG^{k_1}_0)=b_h,
\end{align*}
where  we used $\Gru_{0}\calG^{k_1}_0=h$ and $\Ir [g-g(1)]=0$  close to  $1$,
 and so, as $h\to 0$,
 \begin{align*} \supnorm{G^{\mathrm{DN}}_{\mh} f_h(x) -g}{(1-h,1]} 
 &\le   |b_h-g(1)|+\|g-g(1)\|_{C[1-h,1]}\to 0.
\end{align*}
 If    $2\le \iota(-x)\le n$ (i.e.,    $ x\in  (-1+h,1-h]$), $G^{\mathrm{DN}}_{\mh}1(x)=0$ and for any $w\in C[-1,1]$
\begin{align*}
G^{\mathrm{DN}}_{\mh}w(x) &=\barlambda \left[\sum_{j=0}^{\iota(-x)-1} \Gru_{j} w (x+(j-1)h) +\left(-\sum_{j=0}^{\iota(-x)-1}\Gru_j\right)w (x+(\iota(-x)-1)h)\right]\\
&\quad+\lambda\left[\sum_{j=0}^{\iota(-x)-2} \Gru_{j} w (x+(j-1)h) +\left(-\sum_{j=0}^{\iota(-x)-2}\Gru_j\right)w (x+(\iota(-x)-2)h)\right]\\
&=\barlambda \left[\Crh w (x)-\Gru_{\iota(-x)} w (x+(\iota(-x)-1)h) +\left(-\sum_{j=0}^{\iota(-x)-1}\Gru_j\right)w (x+(\iota(-x)-1)h)\right]\\
&\quad+\lambda\left[\Crh w (x)- \Gru_{\iota(-x)} w (x+(\iota(-x)-1)h)  +\left(-\sum_{j=0}^{\iota(-x)-1}\Gru_j\right)w (x+(\iota(-x)-2)h)\right]\\
&=\Crh w (x)-\Gru_{\iota(-x)} w (x+(\iota(-x)-1)h)\\
&\quad-\left(\barlambda w (x+(\iota(-x)-1)h)+\lambda w (x+(\iota(-x)-2)h)\right)  \sum_{j=0}^{\iota(-x)-1}\Gru_j.
\end{align*}
Then
\begin{align*}
G^{\mathrm{DN}}_{\mh}\vartheta^{k_{1}}_{\mh}(x) =1,
\end{align*} 
by applying \eqref{eq:partialphithetah1}  and the definition of $\vartheta^{k_{1}}_{\mh}$ on the last two grid intervals, namely $\vartheta^{k_{1}}_{\mh}(x+(\iota(-x)-1)h)=-h^{-1}\lambda\calG^{k_1}_0$ and $\vartheta^{k_{1}}_{\mh}(x+(\iota(-x)-2)h)=h^{-1}\barlambda \calG^{k_1}_0$.
Looking   at
\begin{align*}
G^{\mathrm{DN}}_{\mh} \Ir [g-g(1)] (x)&=\Crh \Ir [g-g(1)] (x) \\
&\quad-\left(\lambda \Gru_{\iota(-x)}+\barlambda \sum_{j=0}^{\iota(-x)}\Gru_j  \right) \Ir [g-g(1)]  (x+(\iota(-x)-1)h)\\
&\quad-\lambda  \left(\sum_{j=0}^{\iota(-x)-1}\Gru_j\right)\Ir [g-g(1)] (x+(\iota(-x)-2)h), 
\end{align*}
we see that   by Corollary \ref{cor:L94} the first term converges to $g-g(1)$ uniformly for $x\in(-1+h,1-h]$, and the last to terms are 0 for all small $h$ due $\Ir [g-g(1)]=0$  close to  $1$. Hence we showed that $\supnorm{G^{\mathrm{DN}}_{\mh}f_h-g}{(-1+h,1-h]}\to 0$ as $h\to 0$.

It remains to check uniform convergence for $\iota(-x)=n+1$ (i.e., $x\in [-1,-1+h]$). Exploiting the canonical extensions (Remark \ref{rmk:canon}) and  $\Ir [g-g(1)]=0$  close to $1$, we see that
\begin{align}\nonumber
G^{\mathrm{DN}}_{\mh}f_h(x) &=\Gru_{1}f_h(x)+D^l(\lambda)\left[\sum_{j=2}^{n-1} \Gru_{j} f_h (x+(j-1)h)  - f_h (x+(n-1)h)\sum_{j=0}^{n-1} \Gru_{j} \right]\\ \nonumber
&=\left[(1-D^l(\lambda))\Gru_{1} f_h (x)-D^l(\lambda)\Gru_{0} f_h (x-h)\right] \\ \nonumber
&\quad -D^l(\lambda)\left[\Gru_{n+1} f_h (x+nh)+f_h (x+(n-1)h) \sum_{j=0}^{n} \Gru_{j}\right]\\ \nonumber
&\quad +D^l(\lambda)\sum_{j=0}^{n+1} \Gru_{j} f_h (x+(j-1)h)\\ \nonumber
&=\left[(1-D^l(\lambda))\Gru_{1} f_h (x)-D^l(\lambda)\Gru_{0} f_h (x-h)\right] \\ \nonumber
&\quad -D^l(\lambda)b_h\left[\Gru_{n+1} \vartheta^{k_1}_{\mh}(x+nh)+\vartheta^{k_1}_{\mh} (x+(n-1)h) \sum_{j=0}^{n} \Gru_{j}\right]\\
&\quad +D^l(\lambda)\left[\Crh  \Ir [g-g(1)](x)+b_h\sum_{j=0}^{n+1} \Gru_{j} \vartheta^{k_1}_{\mh}(x+(j-1)h)\right] .  \label{eq:pieces}
\end{align}
The last term in \eqref{eq:pieces} converges uniformly to $g$ because using \eqref{eq:partialphithetah1} with Remark \ref{rmk:canon_extra}\footnote{The current summation equals the right hand side of \eqref{eq:partialphithetah1} also at $x=-1$.},  for $x\in[-1,-1+h]$ 
\begin{align*}
&\,\left|D^l(\lambda)\Crh  \Ir [g-g(1)](x)+b_h D^l(\lambda)\left(1-\lambda \calG^{k_1}_{0}\frac{\Gru_{n+1}}h\right)  -g(x)\right|\\
\le&\, D^l(\lambda)\left|\Crh  \Ir [g-g(1)](x)-g(x)+ b_h   \right| +D^l(\lambda)\lambda |b_h|\calG^{k_1}_{0}\frac{\Gru_{n+1}}h+ \|g\|_{C[-1,-1+h]}, 
\end{align*}
and as $h\to0$: the first term converges to uniformly to 0 by applying Corollary \ref{cor:L94} and $b_h\to g(1)$;  by  \eqref{eq:GClimphi} we bound the second term up to a constant independent of $h$ times $\calG^{k_1}_0/h=1/\sym(1/h)$, which vanishes by  \eqref{eq:convhomega}; and the third term vanishes because $g(-1)=0$. The second term in \eqref{eq:pieces}, by \eqref{eq:GClimphi} and \eqref{eq:convhomega}, is bounded by
\begin{align*}
\frac{\calG^{k_1}_0}h\left|b_h\right|\Big| \sum_{j=0}^{n+1}\Gru_j\Big|\le \frac{1}{\sym(1/h)} \sup \Big\{\left|b_h\right|\Big| \sum_{j=0}^{n+1}\Gru_{j,h}\Big|:h\in(0,1]\Big\}  \to 0, 
\end{align*} 
as $h\to 0$. For the first term in \eqref{eq:pieces} we proceed  as in the proof of Theorem \ref{thm:conv_C_DD}. For $x=-1$ this term vanishes. For $x\in(-1,-1+h]$, the first term in \eqref{eq:pieces} is bounded by  \eqref{eq:bound_tricky_term}, for the current $f_h$. Again we use the Taylor expansions around $-1$,
\begin{align*}
f_h(x)&=\Ir  [g-g(1)](x)+b_h\vartheta^{k_{1}}_{\mh}(x)-\Ir g(-1)\\
&=\Ir  [g-g(1)](-1)+ \lambda h(\Ir  [g-g(1)])'(-1)+ O((\lambda h)^2)+b_h\vartheta^{k_{1}}_{\mh}(x)-\Ir g(-1)\\
&=R_1(\lambda,h)+ \lambda h(\Ir  [g-g(1)])'(-1)+\lambda O(h^2),\\
f_h(x-h)&=\Ir  [g-g(1)](x-h)+b_h\vartheta^{k_{1}}_{\mh}(x-h)-\Ir g(-1)\\
&=\Ir  [g-g(1)](-1)+ (\lambda-1) h(\Ir  [g-g(1)])'(-1)+  O(h^2)+b_h\vartheta^{k_{1}}_{\mh}(x-h)-\Ir g(-1)\\
&=R_2(\lambda,h)- \barlambda  h(\Ir  [g-g(1)])'(-1)+ O(h^2),
\end{align*}
and recalling that $b_h= g(1)k^-_1(-1)/\vartheta^{k_{1}}_{\mh}(-1)$ and
\[
\vartheta^{k_{1}}_{\mh}(-1)=\frac{1}{h}\calG^{k_{1}}_{n-1}, \quad 
\vartheta^{k_{1}}_{\mh}(x)=\frac{1}{h}(\lambda\calG^{k_{1}}_{n-2}+\barlambda \calG^{k_{1}}_{n-1}), 
\]
we obtain
\begin{align*}
R_1(\lambda,h)&=-g(1)k^-_1(-1)\left(1-\frac{\vartheta^{k_{1}}_{\mh}(x)}{\vartheta^{k_{1}}_{\mh}(-1)}\right)\\
&=-g(1)k^-_1(-1)\left(1-\frac{\lambda\calG^{k_{1}}_{n-2}+\barlambda \calG^{k_{0}}_{n-1}}{\calG^{k_{1}}_{n-1}}\right)\\
&=-g(1)k^-_1(-1)\frac\lambda{\calG^{k_{1}}_{n-1}}\left(\calG^{k_{1}}_{n-1}-\calG^{k_{1}}_{n-2}\right),
\end{align*}
and similarly
\begin{align*}
R_2(\lambda,h)&=-g(1)k^-_1(-1)\left(1-\frac{\vartheta^{k_{1}}_{\mh}(x-h)}{\vartheta^{k_{1}}_{\mh}(-1)}\right)\\
&=-g(1)k^-_1(-1)\left(1-\frac{\lambda\calG^{k_{1}}_{n-1}+\barlambda \calG^{k_{1}}_{n}}{\calG^{k_{1}}_{n-1}}\right)\\
&=-g(1)k^-_1(-1)\frac{\barlambda }{\calG^{k_{1}}_{n-1}}\left(\calG^{k_{1}}_{n-1}-\calG^{k_{1}}_{n}\right),
\end{align*}
so that 
\begin{align*}
\barlambda f_h(x)+\lambda f_h(x-h)&=\barlambda R_1(\lambda,h)+\lambda R_2(\lambda,h)+\barlambda  \lambda O(h^2)+\lambda O(h^2)
\\
&=-\barlambda \lambda g(1)k^-_1(-1)\left(\frac{2\calG^{k_{1}}_{n-1}-\calG^{k_{1}}_{n}-\calG^{k_{1}}_{n-2}}{\calG^{k_{1}}_{n-1}}\right)+\lambda O(h^2).
\end{align*}
And finally, by \eqref{eq:phi+1tophi} and \eqref{eq:GClimk-1k1} (combined with the first part of Lemma \ref{lem:PWthms}),
\begin{align*}
\frac{2\calG^{k_{1}}_{n-1}-\calG^{k_{1}}_{n}-\calG^{k_{1}}_{n-2}}{\calG^{k_{1}}_{n-1}}&=  -h^2\frac{\calG^{k_{-1}}_{n}}{\calG_{n-1}^{k_{1}}}= O(h^2).
\end{align*}

\end{proof}

\begin{remark}\label{rmk:convspeedC_DN}
In the proof of Theorem \ref{thm:conv_C_DN} we proved that $\|G^{\mathrm{DN}}_{\mh}f_h-g\|_{C[-1,1]}=O(h^2\sym(1/h))$, because one can prove that   $|b_h-g(1)|=O(h)$ by standard Post–Widder arguments using the fact that $k_1^+$ is Lipschitz.
\end{remark}

\begin{proof}[of Theorem \ref{thm:conv_C_ND}]   Recall that $h(n+1)=2$. By Lemmata \ref{lem:varthek0} and \ref{Ghdissipative}, $f_h,G^{\mathrm{ND}}_{\mh}f_h\in  C_0[-1,1)$. Observe that $e_h\to 0$ uniformly, so that $f_h\to f$ in $ C_0[-1,1)$ by Lemma \ref{lem:varthek0}. It remains to prove the convergence of $G^{\mathrm{ND}}_{\mh}f_h$.
For $\iota(-x)=1$, using \eqref{eq:partialthetah0} ($G^{\mathrm{ND}}_{\mh}=G^{\mathrm{DD}}_{\mh}$ for $\iota(-x)=1$)  and $\Ir  g=0$ on $(1-2h,1]$ for small $h>0$,
\begin{align*}
G^{\mathrm{ND}}_{\mh}f_h(x)&=-I_- g(-1)G^{\mathrm{ND}}_{\mh}\vartheta^{k_{0}}_{\mh}(x) =0.
\end{align*} 
Observe that for $ 2\le\iota(-x)\le n -1$
\begin{align*}
G^{\mathrm{ND}}_{\mh}f_h(x)&= \sum_{j=0}^{\iota(-x)} \Gru_j f_h(x-(j-1)h)= \Crh \Ir  g(x)\to g(x)
\end{align*}
uniformly for $x\in(-1+2h,1-h]$ as $h\to 0$  by Corollary \ref{cor:L94}, where we   used \eqref{eq:partialphithetah0}. 
For $\iota(-x)=n+1$
\begin{align*}
G^{\mathrm{ND}}_{\mh}f_h(x)&=  \sum_{j=0}^{n-1} \left(-\sum_{i=0}^{j} \Gru_i\right) f_h(x+jh) \\
&= -\frac1h \sum_{j=0}^{n-1} \Gruone_j f_h(x+jh) \\
&= -\frac1h \sum_{j=0}^{n}  \Gruone_j f_h(x+jh) +\frac{1}{h}  \Gruone_n f_h(x+nh) \\ 
&= -\frac1h \left[\sum_{j=0}^{n}  \Gruone_j \Ir g(x+jh)+b_h\right]+\barlambda g(-1)  +o(1),
\end{align*}
using for the last equality:  \eqref{eq:partialphi-1thetah0}, $G^{\mathrm{ND}}_{\mh}e_h(x)=\barlambda  g(-1)$ and
\[
o(1)=\barlambda \left(\sum_{j=0}^n\Gru_j  \right)\frac{\calG^{k_0}_0}{h}=\barlambda \frac{ \Gruone_n}{h}  \frac{\calG^{k_0}_0}{h}=\frac{\Gruone_n}{h}   \vartheta^{k_0}_{\mh}(x+nh),
\]
 by  \eqref{eq:phi+1tophi}, \eqref{eq:GClimphi}   and \eqref{eq:convhomega}. By Lemma \ref{lem:firtorderapprox}
\begin{align*}
 \Conerhzero \Ir g(x)= I_-g(x)+ hF (x)+O(h^2),
\end{align*}
where we   simplified notation by writing $F=F_-[g]$, 
 and using  $x=\lambda h-1$ with the Taylor expansion at $-1$
\begin{align}\label{eq:ibp}
 I_-g(x)-I_-g(-1)
&=O(h^2)+(x+1) [I_-g]'(-1)=O(h^2)-\lambda h g(-1),
 \end{align}
we obtain
\begin{align*}
 -\frac1h\left[\sum_{j=0}^{n}  \Gruone_j \Ir g(x+jh)+b_h\right]&=-\frac1h\left[I_-g(x)-I_-g(-1)+h(F (x)-F (-1))\right]+O(h)\\
 &=\lambda  g(-1)-(F (x)-F (-1))+O(h),
 \end{align*}
which then gives us that for any $x\in[-1,-1+h]$
 \begin{align*}
|G^{\mathrm{ND}}_{\mh}f_h(x)-g(x)|&=| G^{\mathrm{ND}}_{\mh}f_h(x)-g(-1)+O(h)|\\
&= | F (x)-F (-1)+o(1)|\\
&\le \|F -F (-1) \|_{C[-1,-1+h]} +o(1)\to 0,
\end{align*}
as $h\to 0$.
The remaining case $\iota(-x)=n$ (i.e., $x\in(-1+h,-1+2h]$) is treated similarly, as 
\begin{equation}\label{eq:zzz}
\begin{split}
G^{\mathrm{ND}}_{\mh}f_h(x)
&=\lambda \left[\sum_{j=0}^{n-1} \Gru_j f_h(x+(j-1)h) - f_h(x+(n-2)h) \sum_{j=0}^{n-1} \Gru_j\right]\\
&\quad+\barlambda \left[ \sum_{j=0}^{n-1} \left(-\sum_{i=0}^{j} \Gru_i\right) f_h(x+jh) \right],
\end{split}\end{equation}
and the second term in the first squared brackets goes to 0 because $\Ir  g=0$ on $[1-2h,1]$ for all $h$ small and
\[
\left|\vartheta^{k_0}_{\mh}(x+(n-2)h) \sum_{j=0}^{n-1} \Gru_j\right|=\left|\barlambda  \frac{\calG_0^{k_0}}{h} \sum_{j=0}^{n-1} \Gru_j\right|\to 0,
\]
 as $h\to 0$ by   \eqref{eq:GClimphi} and   \eqref{eq:convhomega}. The first term in the first squared brackets in \eqref{eq:zzz} rewrites as
\begin{align*}
\Crh  f_h(x) -\Gru_n f_h(x+(n-1)h)&= \Crh  [\Ir g+b_h\vartheta^{k_{0}}_{\mh}+e_h](x) +o(1)\\
&= \Crh  \Ir g(x)-\barlambda g(-1) +o(1),
\end{align*}
where we used:  $\Gru_n f_h(x+(n-1)h)=o(1)$ (by the same argument as above); identity \eqref{eq:partialphithetah0}; and $\Crh   e_h(x)=-\barlambda g(-1)$.  And so, for $\iota(-x)=n$, using Lemma \ref{lem:firtorderapprox}, $ \Conerhzero e_h=0$ and the identities \eqref{eq:ibp} and \eqref{eq:partialphi-1thetah0},
\begin{align*}
G^{\mathrm{ND}}_{\mh}f_h(x)&=\lambda \Crh  \Ir g (x)+o(1)-\frac{\barlambda }h\left[  \Conerhzero \Ir g(x)+b_h \Conerhzero \vartheta^{k_{0}}_{\mh}(x)+\lambda hg(-1) \right]\\
&=\lambda \Crh  \Ir g (x)+o(1)-\frac{\barlambda }h\left[  I _-g(x)+hF(x)+b_h +\lambda hg(-1) \right]\\
&=\lambda \Crh  \Ir g(x)+o(1)-\frac{\barlambda }h\left[ -(1+\lambda) hg(-1)+\lambda hg(-1) \right]+F(x)-F(-1)\\
&=\lambda \Crh  \Ir g(x)+o(1)+\barlambda g(-1)+F(x)-F(-1),
\end{align*}
thus, recalling that $g(x)=(\lambda+\barlambda )g(-1)+o(1)$, $\|F -F(-1)\|_{C[-1+h,-1+2h]}=o(1)$ and Corollary \ref{cor:L94}, as $h\to0$
\[
\supnorm{G^{\mathrm{ND}}_{\mh}f_h-g}{(-1+h,-1+2h]} \le  \|\Crh  \Ir g-g\|_{C[-1+h,-1+2h]}+o(1) \to 0.
\]
  
\end{proof}

\begin{remark}
In the proof of Theorem \ref{thm:conv_C_ND} we proved that $\|G^{\mathrm{ND}}_{\mh}f_h-g\|_{C[-1,1]}=O(1/[h\sym(1/h)])$, using the fact that $F_-[g]\in  C^1[-1,1]$.
\end{remark}

\begin{proof}[of Theorem \ref{thm:conv_C_NN}] Recall that $h(n+1)=2$. Also, we simplify notation by writing $F=F_-[g-g(1)]$. By Lemmata \ref{lem:varthek0} and \ref{Ghdissipative}, $f_h,G^{\mathrm{NN}}_{\mh}f_h\in  C[-1,1]$.  From $\calG^{k_0}_0=1/\sym(1/h)$ with \eqref{eq:convhomega} and the definition of $d_h$ we immediately deduce that $e_h\to 0$ in $C[-1,1]$, so that by Lemma \ref{lem:varthek0},  $f_h\to f$ as $h\to0$ in $C[-1,1]$. It remains to check the convergence of  $G^{\mathrm{NN}}_{\mh}f_h$. \\
For $\iota(-x)=n+1$, using \eqref{eq:phi+1tophi}, 
\begin{align}\nonumber
G^{\mathrm{NN}}_{\mh}f_h(x)&=- \frac1h \sum_{j=0}^{n-2} \Gruone_j f_h(x+jh)  +\frac1h\sum_{j=0}^{n-2}\Gruone_j f_h(x+(n-1)h)\\ \nonumber
&=  - \frac1h\sum_{j=0}^{n} \Gruone_j  f_h(x+jh)+\frac1h  \Gruone_{n} f_h(x+nh)  \\ \nonumber
&\quad +\frac1h\sum_{j=0}^{n-1}\Gruone_j f_h(x+(n-1)h)\\ \nonumber
&=  -\frac1h\left[ \Conerhzero \Ir  [g-g(1)] (x)+b_h \sum_{j=0}^{n} \Gruone_j  \vartheta^{k_{1}}_{\mh}(x+jh)\right]+\barlambda d_h\\ \nonumber
&\quad + \frac1h  \Gruone_{n} b_h\vartheta^{k_{1}}_{\mh}(x+nh)+ \frac1h\sum_{j=0}^{n-2}\Gruone_{j}  b_h\vartheta^{k_{1}}_{\mh}(x+(n-1)h)\\
\label{eq:firstlineNN}
&=  -\frac1h\left[ \Conerhzero \Ir  [g-g(1)] (x)+b_h \sum_{j=0}^{n} \Gruone_j  \vartheta^{k_{1}}_{\mh}(x+jh)\right]+\barlambda d_h+o(1),
\end{align}
using in the second last identity $\Ir  [g-g(1)]=0$ close to 1,  $G^{\mathrm{NN}}_{\mh}1=0,\, G^{\mathrm{NN}}_{\mh}e_h= \barlambda d_h$, and in the last identity we used the estimate
\begin{align*}
&\quad\left| \frac1h  \Gruone_{n} b_h\vartheta^{k_{1}}_{\mh}(x+nh)+ \frac1h\sum_{j=0}^{n-2}\Gruone_{j}  b_h\vartheta^{k_{1}}_{\mh}(x+(n-1)h)\right|\\ 
&\le \sup\{|b_h|:h\in(0,1]\} \left( \frac{\calG^{k_{1}}_0}h\left|\sum_{i=0}^{n} \Gru_i \right| +\frac{\calG^{k_{1}}_0}{h^2} \left|  \sum_{j=0}^{n-2}\Gruone_{j}\right|\right)=o(1), 
\end{align*} 
as $h\to0$ using $\calG^{k_{1}}_0=h/\sym(1/h)$, \eqref{eq:convhomega} and \eqref{eq:GClimphi}.
To show that  \eqref{eq:firstlineNN} converges uniformly to $g-I_-g(-1)/2$ on $[-1,-1+h]$, we first observe that
 by \eqref{eq:phi+1tophi}, \eqref{eq:convhomega} and \eqref{eq:GClimphi}
\begin{equation}
\frac{\Gruone_{n}}{\sym(1/h)}=\frac{ h}{\sym(1/h)} \sum_{j=0}^n\Gru_{j}=o(h^2),
\label{eq:similar}
\end{equation}
and so, using 
 Lemma \ref{lem:firtorderapprox}, identity \eqref{eq:partialphi-1thetah1}\footnote{Recalling Remark \ref{rmk:coeftheta}, the right hand side of \eqref{eq:partialphi-1thetah1} for $x=-1$ equals  $\sum_{j=0}^{n-1}\Gruone_j   \vartheta^{k_{1}}_{\mh} (-1)$.} and \eqref{eq:ibp} with   $x=\lambda h -1$, \eqref{eq:firstlineNN} rewrites as
\begin{align*}
&\quad\frac{-1}h\left[ \Conerhzero \Ir  [g-g(1)] (x)+b_h\left( n h -\lambda h -\frac{\lambda\Gruone_{n}}{\sym(1/h)}\right)\right] +\barlambda  d_h +o(1)\\
&=-\left[\frac1hI_-[g-g(1)](x) +F(x)+b_hn -b_h\lambda \right] +d_h\barlambda  +o(1) \\
&=-\left[\frac1h\left(I_-g(x)-I_-g(-1)+(x+1)g(1)  \right) +F(x)-F(-1) -b_h\lambda\right] +d_h\barlambda +o(1) \\
&=-\frac1h\left(O(h^2)-\lambda hg(-1)+\lambda hg(1) \right)   +b_h\lambda +d_h\barlambda +o(1) \\
&= \lambda g(-1)-\lambda g(1)   +b_h\lambda +d_h\barlambda  +o(1)\\
&=\lambda g(-1)+\lambda g(1)\left(\frac{n+1}{n}-1\right)   -\lambda \frac{I_-g(-1)}2\frac{n+1}{n}  +d_h\barlambda  +o(1)\\
&=  g(-1)    - \frac{I_-g(-1)}2 +o(1) ,
\end{align*}
which combined with   $g(x)=g(-1) +o(1)$ gives the desired convergence $\|G^{\mathrm{NN}}_{\mh}f_h-g + I_-g(-1)/2\|_{C[-1,-1+h]}\to 0$ as $h\to 0$.\\
 For $\iota(-x)=n$ we have
\begin{align}\nonumber
G^{\mathrm{NN}}_{\mh}f_h(x)&= \barlambda \left[- \sum_{j=0}^{n-2} \frac1h  \Gruone_{j} f_h(x+jh)  + f_h(x+(n-1)h )\sum_{j=0}^{n-2}\frac1h\Gruone_{j}\right]\\ \nonumber
&\quad +\lambda\left[ \sum_{j=0}^{n-2}  \Gru_j f_h(x+(j-1)h)  -  f_h(x+(n-2)h)\sum_{j=0}^{n-2}\Gru_j\right] \\ \nonumber
&=  -\frac{\barlambda }h\left[ \Conerhzero \Ir  [g-g(1)] (x)+b_h  \Conerhzero \vartheta^{k_{1}}_{\mh} (x)\right] \\ \nonumber
&\quad+ \lambda\left[\Crh \Ir  [g-g(1)] (x)+b_h \Crh \vartheta^{k_{1}}_{\mh} (x)+\Crh e_h (x)\right]+o(1)\\ 
 \label{eq:firstlineNN2}
&=  -\frac{\barlambda }h\left[ \Conerhzero \Ir  [g-g(1)] (x)+b_h  \Conerhzero \vartheta^{k_{1}}_{\mh} (x)+\lambda h d_h\right]\\ \nonumber
&\quad+ \lambda\left[\Crh \Ir  [g-g(1)] (x)+b_h\right]+o(1)
\end{align}
where we used $\Crh e_h =-\barlambda d_h$ and the $o(1)$ term is obtained by bounding terms close to $1$ in a similar fashion as for the  $o(1)$ term of \eqref{eq:firstlineNN}. The second line in \eqref{eq:firstlineNN2} converges uniformly to $\lambda (g-I_-g(-1)/2)$ by Corollary \ref{cor:L94}. It is then enough to show that the first line in  \eqref{eq:firstlineNN2}  rewrites as $\barlambda [g(-1)-I_-g(-1)/2]+o(1)$. Indeed, by the same strategy   for  \eqref{eq:firstlineNN},  using  Lemma \ref{lem:firtorderapprox} and identity \eqref{eq:partialphi-1thetah1} with  $ \calG^{k_1}_{0}\Gruone_{n-1}/h =o(h^2)$, and  \eqref{eq:ibp} with $x+1=(\lambda+1)h$,
\begin{align*}
&\quad -\frac{\barlambda }h\left[ \Conerhzero \Ir  [g-g(1)] (x)+b_h (n-2+\barlambda )h +\lambda hd_h\right]+o(1)\\ 
&= -\frac{\barlambda }h\left[ \Conerhzero \Ir  [g-g(1)] (x)+b_h hn -b_h(\lambda+1)h +\lambda h d_h \right]+o(1)\\ 
&= -\frac{\barlambda }h\left[(1+\lambda) h( g(1)- g(-1))-b_h(\lambda+1)h +\lambda h d_h \right]+o(1)\\ 
&= -\barlambda \left[(1+\lambda) ( g(1)- g(-1))-\left[g(1) -\frac{I_-g(-1)}2 \right]\frac{n+1}{n}(\lambda+1) +\lambda  d_h \right]+o(1)\\ 
&= -\barlambda \left[-(1+\lambda)\left( g(-1)-\frac{I_-g(-1)}2 \right) +\lambda  d_h \right]+o(1) \\ 
&= -\barlambda \left[-(1+\lambda)\left( g(-1)-\frac{I_-g(-1)}2 \right) +\lambda \left( g(-1) -\frac{I_-g(-1)}2 \right)\right]+o(1).
\end{align*}
 For $2\le\iota(-x)\le n-1$, using $\Ir [g-g(1)]=0$ close to 1 and combining \eqref{eq:partialphithetah1} with the definition of $\vartheta^{k_1}_{\mh}$ in the second identity, we obtain 
\begin{align*}
\Gen^{\mathrm{NN}}_{\mh}f_h(x)&=  \sum_{j=0}^{\iota(-x)-2} \Gru_{j} \Ir [g-g(1)](x+(j-1)h) +b_h \sum_{j=0}^{\iota(-x)-2} \Gru_{j} \vartheta^{k_1}_{\mh}(x+(j-1)h) \\
&\quad +b_h\vartheta^{k_1}_{\mh}(x+(\iota(-x)-2)h)\left(\barlambda  \Gru_{\iota(-x)-1}  - \lambda \sum_{j=0}^{\iota(-x)-2}\Gru_{j}\right)\\
& \quad-b_h\barlambda  \vartheta^{k_1}_{\mh}(x+(\iota(-x)-1)h) \sum_{j=0}^{\iota(-x)-1}\Gru_{j}\\
&=   \Crh \Ir [g-g(1)](x) +b_h-\frac{b_h}h  \barlambda \calG^{k_1}_0\Gru_{\iota(-x)-1} \\
&\quad +\frac{b_h}h \barlambda \calG^{k_1}_0\left(\Gru_{\iota(-x)-1}   - \lambda \sum_{j=0}^{\iota(-x)-1}\Gru_{j}\right)+\frac{b_h}h\barlambda  \lambda \calG^{k_1}_0 \sum_{j=0}^{\iota(-x)-1}\Gru_{j}\\
&=   \Crh \Ir [g-g(1)](x) +b_h,
\end{align*}
which converges  to $g-I_-g(-1)/2$ as $h\to 0$ uniformly on $(-1+2h,1-h]$, by Corollary \ref{cor:L94}. For $\iota(-x)=1$ it is easy to check that $G^{\mathrm{NN}}_{\mh}f_h(x)=b_h$, and we are done.

\end{proof}

\begin{remark}
In the proof of Theorem \ref{thm:conv_C_NN} we proved that $$\|G^{\mathrm{NN}}_{\mh}f_h-(g-I_-g(-1)/2)\|_{C[-1,1]}=O(1/[h\sym(1/h)]),$$ using $\vartheta^{k_1}_{\mh}(-1)-k^1_-(-1)=O(h)$ by the same observation in Remark \ref{rmk:convspeedC_DN}.
\end{remark}

\begin{proof}[of Theorem \ref{thm:conv_C_N*D}] Recall that $h(n+1)=2$. By Lemmata \ref{lem:varthek0} and \ref{Ghdissipative},  $f_h,G^{\mathrm{N^*D}}_{\mh}f_h\in  C_0[-1,1)$. Note that by Lemma \ref{lem:PWthms}, $b_h\to [\Ir  g]'(-1)/k^-_{-1}(-1)$ as $h\to0$. Then by Lemma \ref{lem:varthek0} we obtain $f_h\to f$ in $C_0[-1,1)$. It remains to prove the convergence of $G^{\mathrm{N^*D}}_{\mh}f_h$. For $\iota(-x)=n+1$, using the canonical extensions for $\vartheta^{k_{0}}_{\mh}$ and $g$, that $g=0$ close to 1 and  \eqref{eq:partialphithetah0} with Remark \ref{rmk:canon_extra}, 
\begin{align*} 
G^{\mathrm{N^*D}}_{\mh}f_h(x)&=\Gru_{0}f_h(x)+\sum_{j=1}^{n}\Gru_{j}f_h(x+(j-1)h)\\
&=\Gru_{0}(f_h(x)-f_h(x-h))+\sum_{j=0}^{n+1}\Gru_{j}f_h(x+(j-1)h)- \Gru_{n+1}f_h(x+nh)\\
&=\Gru_{0}(f_h(x)-f_h(x-h))+\Crh \Ir  g(x) +o(1),
\end{align*}
where we used $- \Gru_{n+1}b_h\barlambda \calG^{k_0}_0/h=o(1)$ by \eqref{eq:GClimphi} and \eqref{eq:convhomega}. 
By \eqref{eq:diffvarth}  and \eqref{eq:GClimk-2k0} 
\begin{align*}
\vartheta^{k_{0}}_{\mh}(x)-\vartheta^{k_{0}}_{\mh}(x-h)&=- \calG^{k_{-1}}_{n} -\barlambda h^2\frac{ \calG^{k_{-2}}_{n+1}}{h}=- \calG^{k_{-1}}_{n} +O(h^2),
\end{align*}
and  by Taylor's formula and $x=\lambda h -1$
\[
\Ir  g(x)-\Ir  g(x-h)
=h[\Ir  g]'(x)+O(h^2)\quad\text{and}\quad [\Ir  g]'(x)-[\Ir  g]'(-1)=O(h),
\]
we obtain
\begin{align*}
f_h(x)-f_h(x-h)&= h[\Ir  g]'(x)-b_h\calG^{k_{-1}}_{n}+O(h^2)=O(h^2),
\end{align*}
and thus $\Gru_{0}(f_h(x)-f_h(x-h))=o(1)$ by \eqref{eq:convhomega}. Therefore Corollary \ref{cor:L94} proves that as $h\to 0$
$$\|G^{\mathrm{N^*D}}_{\mh}f_h-g\|_{C[-1,-1+h]}\le \|\Crh \Ir g-g\|_{C[-1,-1+h]} +o(1)\to 0. $$
For $\iota(-x)=n$,
\begin{align*}
G^{\mathrm{N^*D}}_{\mh}f_h(x)&= \sum_{j=0}^{n}\Gru_{j}f_h(x+(j-1)h)+\barlambda \Gru_{0}(f_h(x)-f(x-h))\\
&=\Crh \Ir  g(x) +b_h\Crh \vartheta^{k_{0}}_{\mh}(x)+\barlambda \Gru_{0}(f_h(x)-f(x-h))\\
&= \Crh \Ir  g(x)  +o(1),
\end{align*}
where in the last identity we used  \eqref{eq:partialphithetah0}, and similarly as above we estimated 
\begin{align*}
f_h(x)-f_h(x-h)&= h[\Ir  g]'(x)-b_h\calG^{k_{-1}}_{n-1}+O(h^2)\\
&= h\left([\Ir  g]'(x)-[\Ir  g]'(-1)\frac{\calG^{k_{-1}}_{n-1}}{\calG^{k_{-1}}_{n}}\right)+O(h^2)\\
&= h\left([\Ir  g]'(x)-[\Ir  g]'(-1)\right)+h[\Ir  g]'(-1)\left(1- \frac{\calG^{k_{-1}}_{n-1}}{\calG^{k_{-1}}_{n}}\right)+O(h^2)\\
&= h^2[\Ir  g]'(-1)\left(\frac{\calG^{k_{-2}}_{n}}{\calG^{k_{-1}}_{n}}\right)+O(h^2)= O(h^2),
\end{align*}
using \eqref{eq:phi+1tophi} in the second last identity.
For $2\le\iota(-x)\le n-1$ we obtain the required convergence using \eqref{eq:partialphithetah0}, $g=0$ close to 1 and Corollary \ref{cor:L94}. For $\iota(-x)=1$ we can apply  \eqref{eq:partialthetah0}  observing that $G^{\mathrm{N^*D}}_{\mh}=G^{\mathrm{DD}}_{\mh}$ on $(1-h,1]$, and we are done.

\end{proof}

\begin{remark} 
In the proof of Theorem \ref{thm:conv_C_N*D} we proved that $\|G^{\mathrm{N^*D}}_{\mh}f_h-g\|_{C[-1,1]}=O(h^2\sym(1/h))$, applying the same observation as in Remark \ref{rmk:convspeedC_DD}. 
\end{remark}

\begin{proof}[of Theorem \ref{thm:conv_C_N*N}] Recall that $h(n+1)=2$. By Lemmata \ref{lem:varthek0} and \ref{Ghdissipative},  $f_h,G^{\mathrm{N^*N}}_{\mh}f_h\in  C[-1,1]$.
Note that by Lemma \ref{lem:PWthms}, $b_h\to g(1)+[\Ir  g]'(-1)/k^-_{0}(-1)$, so that with Lemma \ref{lem:varthek0} we obtain $f_h\to f$ in $C[-1,1]$. It remains to prove the convergence of $G^{\mathrm{N^*N}}_{\mh}f_h$. Observe that  $G^{\mathrm{N^*N}}_{\mh}c=0$ on $[-1,1]$.\\
 For $\iota(-x)=n+1$, using the canonical extensions for $\vartheta^{k_{1}}_{\mh}$, $g$ and $c$,
\begin{align*}
G^{\mathrm{N^*N}}_{\mh}f_h(x)&=\Gru_{0}f_h(x)+\sum_{j=1}^{n-1}\Gru_{j}f_h(x+(j-1)h)-f_h(x+(n-1)h)\sum_{j=0}^{n-1}\Gru_{j}\\
&=\Gru_{0}(f_h(x)-f_h(x-h))+\sum_{j=0}^{n+1}\Gru_{j}f_h(x+(j-1)h)\\
&\quad- \Gru_{n}f_h(x+(n-1)h)- \Gru_{n+1}f_h(x+nh)-f_h(x+(n-1)h)\sum_{j=0}^{n-1}\Gru_{j}\\
&=\Gru_{0}(f_h(x)-f_h(x-h))+\Crh \Ir  [g-g(1)](x) +b_h+o(1),
\end{align*}
where in the last identity we used $\Ir  [g-g(1)]=0$ close to 1,   \eqref{eq:GClimphi}, \eqref{eq:convhomega}  and
\[
\sum_{j=0}^{n+1}\Gru_{j}\vartheta^{k_{1}}_{\mh}(x+(j-1)h)=1 - \lambda \frac{\calG^{k_1}_0}h\Gru_{n+1}=1+o(1).
\]
If we show that $(f_h(x)-f_h(x-h))=O(h^2)$ 
the first term is $o(1)$ because $h^2\Gru_{0}= o(1)$ by \eqref{eq:convhomega}, and Corollary \ref{cor:L94} proves the convergence on the interval $[-1,-1+h]$.  So we estimate, using \eqref{eq:diffvarth}, \eqref{eq:convhomega}, Lemma \ref{lem:PWthms}  and Taylor's formula,
\begin{align*}
f_h(x)-f_h(x-h)&=h\left(\Ir  [g-g(1)]\right)'(x)+O(h^2)-b_h\left(\calG^{k_0}_{n-1}+O(h^{2})\right)\\
&=h\left(\left(\Ir  [g-g(1)]\right)'(x)-[\Ir  g]'(-1)	\right)- g(1)\calG^{k_0}_{n-1}+O(h^2)\\
&=h\left([\Ir  g]'(x)-[\Ir  g]'(-1)	\right) +g(1)h\left(k^-_{0}(x)-  \frac{\calG^{k_0}_{n-1}}h\right)+O(h^2)\\
&= g(1)h\left(k^-_{0}(x)-k^-_{0}(-1) +k^-_{0}(-1)- \frac{1}{h}\calG^{k_0}_{n-1}\right)+O(h^2),
\end{align*}
and we conclude with   $k^-_{0}\in C^1[-1,1)$  and  $k^-_{0}(-1)-  \calG^{k_0}_{n-1} /h=O(h)$ by  \eqref{eq:GClimk-2k0Tricky}. \\ 
For $\iota(-x)=n$ 
\begin{align*}
G^{\mathrm{N^*N}}_{\mh}f_h(x)&=\barlambda \left[\Gru_{0}f_h(x)+\sum_{j=1}^{n-1}\Gru_{j}f_h(x+(j-1)h)-f_h(x+(n-1)h)\sum_{j=0}^{n-1}\Gru_{j}\right]\\
&\quad+\lambda\left[\sum_{j=0}^{n-2}\Gru_{j}f_h(x+(j-1)h)-f_h(x+(n-2)h)\sum_{j=0}^{n-2}\Gru_{j}\right]\\ 
&=\Crh f_h(x)  + \barlambda \left[\Gru_{0}(f_h(x)-f_h(x-h)) -f_h(x+(n-1)h)\sum_{j=0}^{n}\Gru_{j}\right]\\
&\quad+\lambda\left[ -f_h(x+(n-1)h)\Gru_{n}-f_h(x+(n-2)h)\sum_{j=0}^{n-1}\Gru_{j}\right]\\
&=\Crh [f_h-c](x) + \barlambda \Gru_{0}(f_h(x)-f_h(x-h))+o(1)\\
&=\Crh [f_h-c](x) +o(1),
\end{align*}
which gives the desired convergence combining Corollary \ref{cor:L94} and \eqref{eq:partialphithetah1}, where: the $o(1)$ term in the third equality  comes from bounding the terms close to $1$ as usual using \eqref{eq:convhomega}, Lemma \ref{lem:PWthms} and $g-g(1)=0$ close to 1; and the $o
(1)$ term in the fourth equality,  similarly as above but with the additional help of \eqref{eq:phi+1tophi}, is due to
\begin{align*}
f_h(x)-f_h(x-h)&=h\left(\Ir  [g-g(1)]\right)'(x)+O(h^2)-b_h(\calG^{k_0}_{n-2}+O(h^{2}))\\
&=h\left(\left(\Ir  [g-g(1)]\right)'(x)-[\Ir  g]'(-1)\frac{\calG^{k_0}_{n-2}}{\calG^{k_0}_{n-1}}	\right)- g(1)\calG^{k_0}_{n-2}+O(h^2)\\
&=h\left([\Ir  g]'(x)-[\Ir  g]'(-1)+[\Ir  g]'(-1)h\frac{\calG^{k_{-1}}_{n-1}}{\calG^{k_0}_{n-1}}	\right)\\
&\quad +hg(1)\left(k^-_{0}(x)-\frac{\calG^{k_0}_{n-2}}h\right)+O(h^2)\\
&= g(1)h\left(k^-_{0}(x)-k^-_{0}(-1) +k^-_{0}(-1)- \frac{1}{h}\calG^{k_0}_{n-2}\right)+O(h^2)=O(h^2).
\end{align*}
For $2\le \iota(-x)\le n-1$, using \eqref{eq:partialphithetah1},
\begin{align*} 
G^{\mathrm{N^*N}}_{\mh}f_h(x)&=\lambda\left[\sum_{j=0}^{\iota(-x)-2}\Gru_{j}f_h(x+(j-1)h)-f_h(x+(\iota(-x)-2)h)\sum_{j=0}^{\iota(-x)-2}\Gru_{j}\right]\\
&\quad+\barlambda \left[\sum_{j=0}^{\iota(-x)-1}\Gru_{j}f_h(x+(j-1)h)-f_h(x+(\iota(-x)-1)h)\sum_{j=0}^{\iota(-x)-1}\Gru_{j}\right]\\ 
&=\Crh \Ir  [g-g(1)](x)+ b_h\left[\Crh  \vartheta^{k_{1}}_{\mh}(x)-\Gru_{\iota(-x)}\vartheta^{k_{1}}_{\mh}(x+(\iota(-x)-1)h)\right]\\
&\quad-b_h\left[\lambda \vartheta^{k_{1}}_{\mh}(x+(\iota(-x)-2)h)+\barlambda \vartheta^{k_{1}}_{\mh}(x+(\iota(-x)-2)h)\right]\sum_{j=0}^{\iota(-x)-1}\Gru_{j}\\
&=\Crh \Ir  [g-g(1)](x)+ b_h\left[1 -\lambda \frac{\calG^{k_1}_0}h\Gru_{\iota(-x)}+\Gru_{\iota(-x)} \lambda \frac{\calG^{k_1}_0}h\right]\\
&\quad-b_h\left[\lambda \barlambda \frac{\calG^{k_1}_0}h-\barlambda \lambda \frac{\calG^{k_1}_0}h\right]\sum_{j=0}^{\iota(-x)-1}\Gru_{j}\\
&=\Crh \Ir  [g-g(1)](x)+ b_h.
\end{align*}
Lastly, for $\iota(-x)=1$ it is easy to check  that
\begin{align*}
G^{\mathrm{N^*N}}_{\mh}f_h(x)&=b_h\Gru_0\left(\vartheta^{k_{1}}_{\mh}(x-h)-\vartheta^{k_{1}}_{\mh}(x)\right) =b_h\Gru_0\left(\barlambda  \frac{\calG^{k_1}_0}h+\lambda \frac{\calG^{k_1}_0}h\right)=b_h.
\end{align*}

\end{proof}


\section{Proofs of the theorems of Section \ref{sec:caseL}}\label{app:caseL}
\begin{table}
\centering
\vline
\begin{tabular}{l|l|l|}
  \hline
  \multicolumn{3}{c}{$X = L^1[-1,1]$ and $g\in C^\infty_c(-1,1)$}  
	\vline \\
	\hline
	 $(\Gen^+, \BC)$ & $f \in \mathcal{C}(\Gen^+, \BC)$ & $f_h \in L^1[-1,1] $\\
	\hline
	$(\Cl , \mathrm{DD})$ & $\Il  g-\frac{\Il  g(1)}{k_{0}^+(1)}k_{0}^+$  &$\Il  g-\frac{\Il  g(1)}{\vartheta^{k_{0}}_{\ph}(1)}\vartheta^{k_{0}}_{\ph}$ \\
  \hline
	$(\Cl , \mathrm{DN})$ &$ \Il  g-I_+g(1)k^+_{0}$  & $ \Il  g-I_+g(1)\vartheta^{k_{0}}_{\ph}+e_h$   \\
  \hline
	$(\Cl , \mathrm{ND})$  &  $\Il  g-I_+g(1)$ & $\Il  g-I_+g(1) \vartheta^0_{+h}$   \\
  \hline
	$(\Cl , \mathrm{NN})$  &  $\Il  g-\frac{I_+g(1)}{2} k^+_1+c$ & $\Il  g+b \vartheta^{k_1}_{\ph}+c\vartheta^{0}_{\ph}+e_h $  \\
  \hline
		$(\RLl , \mathrm{N^*D})$   & $ \Il  g-\frac{\Il g(1)}{k^+_{-1}(1)}k^+_{-1}$ & $\Il  g-\frac{\Il g(1)}{\vartheta^{k_{-1}}_{\ph}(1)} \vartheta^{k_{-1}}_{\ph}$    \\
  \hline
	$(\RLl , \mathrm{N^*N})$   &$ \Il  g-\frac{I_+ g(1)}2k^+_{1}+c_{-1}k^+_{-1}$ & $\Il  g+b \vartheta^{k_{1}}_{\ph}(x)+c_{-1}\vartheta^{k_{-1}}_{\ph}+e_h$     \\
  \hline
  \end{tabular}
\caption{\label{fhexplicit1}
Functions $f_h \in L^1[-1,1]$ for the approximation theorems of Section \ref{sec:caseL}, where we respectively refer to Theorems  \ref{thm:conv_L_DN}, \ref{thm:conv_L_NN} and \ref{thm:conv_L_N*N} for the details of $f_h$ in the cases $(\Cl , \mathrm{DN})$, $(\Cl , \mathrm{NN})$ and  $(\RLl , \mathrm{N^*N})$  .}
\end{table}
\begin{proof}[of Theorem \ref{thm:conv_L_DD}] Recall the matrix \eqref{interpolationmatrixGBC_DD}, definition \eqref{G+hBC} and that $h(n+1)=2$. By Lemmata \ref{lem:varthek0} and \ref{Ghdissipative},  $f_h,G^{\mathrm{DD}}_{\ph}f_h\in  L^1[-1,1]$. By Lemma \ref{lem:PWthms}, $b_h\to- \Il  g(1)/k^+_{0}(1)$ as $h\to0$, and with  Lemma \ref{lem:varthek0} we obtain  $f_h\to f$  as $h\to0$ in $L^1[-1,1]$. It remains to prove the convergence of $G^{\mathrm{DD}}_{\ph}f_h$. 
 
For $\iota(x)=n+1$, using \eqref{eq:DDk0+}, $\Il g=0$ close to $-1$  and the canonical extensions for $\vartheta^{k_{0}}_{\ph}$ and $g$ on $[1,1+h]$ (Remark \ref{rmk:canon}),
\begin{align}\nonumber
G^{\mathrm{DD}}_{\ph}f_h(x)&=\sum_{j=1}^{\iota(x)-1} \Gru_{j} \Il g(x-(j-1)h)+ b_hG^{\mathrm{DD}}_{\ph}\vartheta^{k_{0}}_{\ph}(x) \\ \nonumber
&=\Clh  \Il  g(x)-\Gru_{0}\left( \Il g(x+h)+b_h \frac1h \left(\barlambda  \calG^{k_0}_{n}+ \lambda \calG^{k_0}_{n+1}\right)\right)  +o(1) \\ \nonumber
&=\Clh  \Il  g(x)+\Gru_{0}O( h)  +o(1),
\end{align} 
because by \eqref{eq:phi+1tophi}, \eqref{eq:GClimk-1k1} and $h\vartheta^{k_0}_{\ph}(1)=\calG^{k_0}_{n}$  
\begin{align*}
&\quad  \Il g(x+h)+b_h \frac1h \left(\barlambda  \calG^{k_0}_{n}+ \lambda \calG^{k_0}_{n+1}\right) \\
&=\Il g(x+h)-\Il g(1)+\Il g(1)\lambda \left( 1- \frac{  \calG^{k_0}_{n+1}}{\calG^{k_0}_{n}}\right) \\
&=O(h)+\lambda h  \frac{\calG^{k_{-1}}_{n} }{\calG^{k_0}_{n}} =O(h).
\end{align*}
 For $\iota(x)= n$, using again \eqref{eq:DDk0+} and proceeding as above in the last   identity below,
\begin{align*}\nonumber
G^{\mathrm{DD}}_{\ph}f_h(x)&=D^r(\lambda)\Gru_{0} \Il g(x+h)+\sum_{j=1}^{\iota(x)-1} \Gru_{j} \Il g(x-(j-1)h) + b_hG^{\mathrm{DD}}_{\ph}\vartheta^{k_{0}}_{\ph}(x) 
\\ \nonumber
&=(D^r(\lambda)-1)\Gru_{0}\left( \Il g(x+h) +b_h \vartheta^{k_{0}}_{\ph}(x+h)\right)+\Clh  \Il  g(x)+o(1) \\ 
&=\Gru_{0} O(h)+\Clh  \Il  g(x)+o(1).
\end{align*} 

For $1\le \iota(x)\le n-1$, by \eqref{eq:DDk0+},
\begin{align*}\nonumber
G^{\mathrm{DD}}_{\ph}f_h(x)&=G^{\mathrm{DD}}_{\ph}\Il g(x)=\Clh  \Il  g(x).
\end{align*} 
Collecting the above identities we get
\begin{align*}
 \int_{-1}^1|G^{\mathrm{DD}}_{\ph}f_h(x)-g(x)|\,\dd x&\le \|\Clh  \Il  g-g\|_{L^1[-1,1]}+   \int_{1-2h}^{1} |\Gru_0 O(h)+o(1)|\,\dd x \\
 &\le \|\Clh  \Il  g-g\|_{L^1[-1,1]}+ O(h^2)\Gru_0 +ho(1),
 \end{align*}
 and the right hand side vanishes as $h\to0$ by   Corollary \ref{cor:L94} and \eqref{eq:convhomega}.
 
 \end{proof}

\begin{proof}[of Theorem \ref{thm:conv_L_DN}] Recall the matrix \eqref{interpolationmatrixGBC_DN}, definition \eqref{G+hBC} and that $h(n+1)=2$.  First note that
\begin{equation*}
G^{\mathrm{DN}}_{\ph}e_h(x)=\left\{\begin{split}& 0, & x\in[-1,1-2h),\\
&- \Gru_0 \lambda \left( \Il g(x+h)+\vartheta^{k_{0}}_{\ph}(x+h)\right), & x\in[1-2h,1-h),\\
&\lambda\Gru_0 \left( \Il g(x)+\vartheta^{k_{0}}_{\ph}(x)\right),& x\in[1-h,1],
\end{split}\right. 
\end{equation*} 
and $\|e_h\|_{L^1[-1,1]}\to 0$ as $\sup_{h\in(0,1]} \supnorm{\Il g+\vartheta^{k_{0}}_{\ph}}{[1-h,1]}<\infty$. Then, by Lemmata \ref{lem:varthek0} and \ref{Ghdissipative},  $f_h,G^{\mathrm{DN}}_{\ph}f_h\in  L^1[-1,1]$, and by Lemma \ref{lem:varthek0}, $f_h\to f$ in $L^1[-1,1]$. It remains to prove the convergence of $G^{\mathrm{DN}}_{\ph}f_h$.\\
 For $\iota(x)=1$, as $\Il g(x)=0$ for all small $h$ and by \eqref{eq:DDk0+},  
\begin{align*}
G^{\mathrm{DN}}_{\ph}f_h(x)&= G^{\mathrm{DD}}_{\ph}\Il g(x)+b G^{\mathrm{DD}}_{\ph}\vartheta^{k_{0}}_{\ph}(x) =0.
\end{align*}
and similarly, for $2\le\iota(x)\le n-1$,
\begin{align*}
G^{\mathrm{DN}}_{\ph}f_h(x)&=G^{\mathrm{DD}}_{\ph}f_h(x)=\Clh \Il g(x).
\end{align*}
For $\iota(x)=n$, using \eqref{eq:partialphithetah0+}, \eqref{eq:partialphi-1thetah0+} and \eqref{eq:phi+1tophi}, 
\begin{align*}
G^{\mathrm{DN}}_{\ph}\vartheta^{k_{0}}_{\ph}(x)&=\barlambda \left[\sum_{j=0}^{n-1} \Gru_j \vartheta^{k_{0}}_{\ph}(x-(j-1)h) \right]\\
&\quad -\lambda\left[-\Gru_0 \vartheta^{k_{0}}_{\ph}(x+h)   +\frac1h\sum_{j=0}^{n-2} \Gruone_j \vartheta^{k_{0}}_{\ph}(x-jh) \right] \\
&\quad-D^l(\lambda)  \frac1h \Gruone_{n-1} \vartheta^{k_{0}}_{\ph}(x-(n-1)h)  \\
&=\barlambda \left[-\frac\lambda h \Gru_{n} \calG^{k_0}_{0} \right] -\lambda\left[-\Gru_0 \vartheta^{k_{0}}_{\ph}(x+h)  +\frac1h\left(1-\frac\lambda {h} \Gruone_{n-1} \calG^{k_0}_{0}\right) \right] \\
&\quad-\lambda \frac{\Gruone_{n-1} }h\frac{ \calG^{k_0}_0}h    \\
&=\lambda\Gru_0 \vartheta^{k_{0}}_{\ph}(x+h)-\frac\lambda h +o(1),
\end{align*}
where the $o(1)$ term is due to \eqref{eq:GClimphi} and \eqref{eq:convhomega}. 
Meanwhile for all small $h$
\begin{align*}
G^{\mathrm{DN}}_{\ph}\Il g(x)
&=\barlambda \Clh  \Il g(x)-\lambda\left[-\Gru_0 \Il g(x+h)   +\frac1h \Conelhzero \Il g(x) \right] \\
&=\Clh  \Il g(x)-\lambda \Clh  \Il g(x)-\lambda\left[-\Gru_0 \Il g(x+h)   +\frac1h \Conelhzero \Il g(x) \right] \\
&=\Clh  \Il g(x)+O(1)-\lambda\left[-\Gru_0 \Il g(x+h)   +\frac1h \Conelhzero \Il g(x) \right] .
\end{align*}
By Lemma \ref{lem:firtorderapprox}
\begin{align*}
 -\frac\lambda h\left[ \Conelhzero \Il g(x)+b \right] 
&= -\lambda\frac1h\left[ h F_+[g](x)+O(h)\right] 
= O(1),
\end{align*}
and thus we proved that for $\iota(x)=n$
\begin{align*}
G^{\mathrm{DN}}_{\ph}[\Il g+b \vartheta^{k_{0}}_{\ph}](x) &=\Clh  \Il g(x)+O(1)+\lambda\Gru_0\left( \Il g(x+h)+\vartheta^{k_{0}}_{\ph}(x+h)\right).
\end{align*}

For $\iota(x)=n+1$, using \eqref{eq:partialphi-1thetah0+}\footnote{Even though this holds at 1 for our current summation (recalling Remark \ref{rmk:Dphi_h}), we do not need it as we are working on $L^1$.}  and the usual estimates, 
\begin{align*}
G^{\mathrm{DN}}_{\ph}\vartheta^{k_{0}}_{\ph}(x)&=-\barlambda \left[\frac1h\sum_{j=0}^{n-1} \Gruone_j \vartheta^{k_{0}}_{\ph}(x-jh) \right]-\lambda  \Gru_0 \vartheta^{k_{0}}_{\ph}(x) \\
&=-\frac{\barlambda }h\left[1-\frac\lambda h \Gruone_{n} \calG^{k_0}_{0} \right]-\lambda  \Gru_0 \vartheta^{k_{0}}_{\ph}(x)\\
&=-\frac{\barlambda }h+o(1)-\lambda  \Gru_0 \vartheta^{k_{0}}_{\ph}(x),
\end{align*}
meanwhile for all small $h$
\begin{align*}
G^{\mathrm{DN}}_{\ph}\Il g(x)&=-\barlambda \left[  \frac1h\sum_{j=0}^{n-1} \Gruone_j \Il g(x-jh) \right] -\lambda  \Gru_0 \Il g(x)\\
&=-\barlambda \frac1h  \Conelhzero \Il g(x) -\lambda  \Gru_0 \Il g(x).
\end{align*}
Then, again by Lemma \ref{lem:firtorderapprox},
\begin{align*}
 -\frac{\barlambda }h\left[ \Conelhzero \Il g(x)+b \right] 
&= -\frac{\barlambda }h\left[h F_+[g](x)+O(h)\right] 
= O(1),
\end{align*}
so that for $\iota(x)=n+1$
\begin{align*}
G^{\mathrm{DN}}_{\ph}[\Il g+b \vartheta^{k_{0}}_{\ph}](x) &=O(1)-\lambda\Gru_0\left( \Il g(x)+\vartheta^{k_{0}}_{\ph}(x)\right).
\end{align*}
Then, combining the above with Corollary \ref{cor:L94}, as $h\to 0$
\begin{align*}
 \int_{-1}^1|G^{\mathrm{DN}}_{\ph}f_h(x)-g(x)|\,\dd x&\le \|\Clh  \Il  g-g\|_{L^1[-1,1]}+ \int_{1-2h}^{1} |O(1)|\,\dd x\to 0.
 \end{align*}
 
\end{proof}


\begin{proof}[of Theorem \ref{thm:conv_L_ND}] Recall the matrix  \eqref{interpolationmatrixGBC_ND}, definition \eqref{G+hBC} and that $h(n+1)=2$. Clearly $\vartheta_{\ph}^0\to 1$ in $L^1[-1,1]$ so that
$f_h,G^{\mathrm{ND}}_{\ph}f_h\in  L^1[-1,1]$, by Lemma  \ref{Ghdissipative}, and
 $f_h\to f$ in $L^1[-1,1]$.  It remains to prove the convergence of $G^{\mathrm{ND}}_{\ph}f_h$.\\ 
 For $1\le\iota(x)\le n-1$, as $\Il g=0$ close to $-1$, for all small $h>0$,  
\begin{align*}
G^{\mathrm{ND}}_{\ph}f_h(x)
&=\Clh  \Il  g(x).
\end{align*}
For  $\iota(x)=n$,   by \eqref{eq:phi+1tophi} and $\Il  g(x+h) -\Il  g(1)=O(h)$,
\begin{align*}
G^{\mathrm{ND}}_{\ph}f_h(x)&=\Clh  \Il  g(x)+(D^r(\lambda)-1) \Gru_{0} \left(\Il  g(x+h) +\Il  g(1)\right) \\
&=\Clh  \Il  g(x)+\Gru_{0}O(h).
\end{align*} For the remaining interval $\iota(x)=n+1$, using the canonical extension for $g$ (Remark \ref{rmk:canon}),
\begin{align*}
G^{\mathrm{ND}}_{\ph}f_h(x)&=\Clh  \Il  g(x)-  \Gru_{0} \left(\Il  g(x+h) -    \Il  g(1)\right)=\Clh  \Il  g(x)+  \Gru_{0}O(h).
\end{align*}
 Combining the above with Corollary \ref{cor:L94} and \eqref{eq:convhomega}, as $h\to 0$
\begin{align*}
 \int_{-1}^1|G^{\mathrm{ND}}_{\ph}f_h(x)-g(x)|\,\dd x&\le \|\Clh  \Il  g-g\|_{L^1[-1,1]}+  \Gru_{0} \int_{1-2h}^{1} |O(h)|\,\dd x \to 0.
 \end{align*}
\end{proof}


\begin{proof}[of Theorem \ref{thm:conv_L_NN}] Recall the matrix  \eqref{interpolationmatrixGBC_NN}, definition \eqref{G+hBC} and that $h(n+1)=2$. By Lemmata \ref{lem:varthek0} and \ref{Ghdissipative},  $f_h,G^{\mathrm{NN}}_{\ph}f_h\in L^1[-1,1]$, and  clearly $e_h\to 0$ and  $ \vartheta_{\ph}^0\to 1$ as $h\to0$ in $L^1[-1,1]$, which yields $f_h\to f$  as $h\to0$ in $L^1[-1,1]$ when combined with Lemma \ref{lem:varthek0}.  It remains to prove the convergence of $G^{\mathrm{NN}}_{\ph}f_h$.   First observe   that 
\begin{equation*} G^{\mathrm{NN}}_{\ph} e_h(x)=\left\{\begin{split}& 0, & x\in[-1,1-2h),\\ 
&-\Gru_0\lambda \left(\Il g(x+h)+b\vartheta_{\ph}^{k_1}(x+h)+c\right),& x\in[1-2h,1-h),\\
&\Gru_0\lambda \left(\Il g(x)+b\vartheta_{\ph}^{k_1}(x)+c\right),& x\in[1-h,1],
\end{split}\right.  
\end{equation*}
and 
\begin{equation*} G^{\mathrm{NN}}_{\ph} \vartheta_{\ph}^0(x)=\left\{\begin{split}& 0, & x\in[-1,1-2h),\\ 
&\lambda \Gru_0,& x\in[1-2h,1-h),\\
&-\lambda \Gru_0,& x\in[1-h,1].
\end{split}\right.  
\end{equation*}
 For $1\le \iota(x)\le n-1$, as $ \Il  g=0$ close to $-1$ we have  for all small $h$
\[
G^{\mathrm{NN}}_{\ph}\Il  g(x)=\Clh \Il  g(x),
\]
and 
\[
G^{\mathrm{NN}}_{\ph} b\vartheta^{k_1}_{\ph}(x)=-\frac{I_+g(1)}{2},
\]
because for $\iota(x)=1$ 
\begin{align*}
G^{\mathrm{NN}}_{\ph} \vartheta^{k_1}_{\ph}(x)&=\lambda\Gru_{0} \vartheta^{k_1}_{\ph}(x+h)- \Gru_{0} \vartheta^{k_1}_{\ph}(x)
=\Gru_{0}\left(\frac1h\lambda \calG^{k_1}_0+\frac1h\barlambda  \calG^{k_1}_0\right)
=1,
\end{align*}
and for $2\le \iota(x)\le n-1$, using   \eqref{eq:phi+1tophi} and \eqref{eq:dconvk1},
\begin{align*}
G^{\mathrm{NN}}_{\ph} \vartheta^{k_1}_{\ph}(x)&=\sum_{j=0}^{\iota(x)-2}\Gru_{j} \vartheta^{k_1}_{\ph}(x-(j-1)h)- \vartheta^{k_1}_{\ph}(x-(\iota(x)-1)h)\frac1h \Gruone_{\iota(x)-1}\\
&\quad+ \vartheta^{k_1}_{\ph}(x-(\iota(x)-2)h)\left[\lambda \Gru_{\iota(x)-1}-\barlambda \frac1h \Gruone_{\iota(x)-2} \right] \\
&=\frac1h\sum_{j=0}^{\iota(x)-2}\Gru_{j} \calG^{k_1}_{\iota(x)-1-j}+\frac{\barlambda }h\frac{\calG^{k_1}_0}h \Gruone_{\iota(x)-1}+ \frac1h\calG^{k_1}_{0}\left[\lambda \Gru_{\iota(x)-1}-\barlambda \frac1h \Gruone_{\iota(x)-2} \right] \\
&=1-\frac1h\Gru_{\iota(x)-1} \calG^{k_1}_{0} + \frac1h\calG^{k_1}_{0}  \Gru_{\iota(x)-1}  =1.
\end{align*}
For $\iota(x)=n$
\[
G^{\mathrm{NN}}_{\ph} b\vartheta^{k_1}_{\ph}(x)= I_+ g(1)\frac{\lambda}h+\lambda b \Gru_{0} \vartheta^{k_1}_{\ph}(x+h)+O(1),
\]
and 
\begin{equation}\label{eq:casen}
G^{\mathrm{NN}}_{\ph} \Il g(x)=-I_+g(x)\frac{\lambda}h+\lambda \Gru_{0}\Il g(x+h)+O(1),
\end{equation}
because,  using below Lemma \ref{lem:convid} in the fourth identity and \eqref{eq:convhomega}, \eqref{eq:phi+1tophi} and \eqref{eq:GClimphi} in the second last identity,
\begin{align*}
G^{\mathrm{NN}}_{\ph} \vartheta^{k_1}_{\ph}(x)&= -\lambda\frac1h\sum_{j=0}^{n-3}\Gruone_{j}  \vartheta^{k_1}_{\ph}(x-jh)+ \barlambda \sum_{j=0}^{n-2}\Gru_{j} \vartheta^{k_1}_{\ph}(x-(j-1)h)+\lambda \Gru_{0} \vartheta^{k_1}_{\ph}(x+h)\\
&\quad -\frac1h\Gruone_{n-2}\vartheta^{k_1}_{\ph}(x-(n-2)h)  +\vartheta^{k_1}_{\ph}(x-(n-1)h) \frac1h \sum_{j=0}^{n-2}\Gruone_{j}\\
&= -\lambda\frac1{h^2}\sum_{j=0}^{n-3}\Gruone_{j} \calG^{k_1}_{n-2-j}  + \frac{\barlambda }h\sum_{j=0}^{n-2}\Gru_{j} \calG^{k_1}_{n-1-j} +\lambda \Gru_{0} \vartheta^{k_1}_{\ph}(x+h) \\
&\quad -\frac1{h^2}\Gruone_{n-2}\calG^{k_1}_{0}  -\frac{\barlambda }h\calG^{k_1}_0 \frac1h \sum_{j=0}^{n-2}\Gruone_{j}\\
&= -\lambda\frac1{h^2}\sum_{j=0}^{n-2}\Gruone_{j} \calG^{k_1}_{n-2-j}  + \frac{\barlambda }h\sum_{j=0}^{n-2}\Gru_{j} \calG^{k_1}_{n-1-j} +\lambda \Gru_{0} \vartheta^{k_1}_{\ph}(x+h) \\
&\quad -\frac{\barlambda }{h^2}\Gruone_{n-2}\calG^{k_1}_{0}  -\frac{\barlambda }h\calG^{k_1}_0 \frac1h \sum_{j=0}^{n-2}\Gruone_{j}\\
&= -\lambda(n-1)  + \barlambda  - \frac{\barlambda }h\Gru_{n-1} \calG^{k_1}_{0} +\lambda \Gru_{0} \vartheta^{k_1}_{\ph}(x+h) \\
&\quad -\frac{\barlambda }{h^2}\Gruone_{n-2}\calG^{k_1}_{0}  -\frac{\barlambda }h\calG^{k_1}_0 \frac1h \sum_{j=0}^{n-2}\Gruone_{j}\\
&= -\lambda\left(\frac2h -2\right)  + \barlambda  +o(h) +\lambda \Gru_{0} \vartheta^{k_1}_{\ph}(x+h) +o(h)+ o(1) \\
&= -2\frac{\lambda}h +\lambda \Gru_{0} \vartheta^{k_1}_{\ph}(x+h)+O(1),
\end{align*}
and using Lemma \ref{lem:firtorderapprox} and Corollary \ref{cor:L94}
\begin{align*}
G^{\mathrm{NN}}_{\ph} \Il g(x)
&=-\lambda\frac1h \Conelhzero \Il g(x)+ \barlambda \Clh \Il g(x)+\lambda \Gru_{0}\Il g(x+h)\\
&=-\lambda\frac1h\left[I_+g(x)+hF_+[g](x)+O(h^2)\right]+ O(1)+\lambda \Gru_{0}\Il g(x+h)\\
&=-\lambda\frac{I_+g(x)}h+ O(1)+\lambda \Gru_{0}\Il g(x+h).
\end{align*} 
Lastly, for $\iota(x)=n+1$
\[
G^{\mathrm{NN}}_{\ph} b\vartheta^{k_1}_{\ph}(x)= I_+ g(1)\frac{\barlambda }h-\lambda b\Gru_{0} \vartheta^{k_1}_{\ph}(x)+O(1),
\]
and 
\[
G^{\mathrm{NN}}_{\ph} \Il g(x)=-I_+g(x)\frac{\barlambda }h-\lambda \Gru_{0}\Il g(x)+O(1),
\]
because
\begin{align*}
G^{\mathrm{NN}}_{\ph} \vartheta^{k_1}_{\ph}(x)&=-\frac{\barlambda }h\sum_{j=0}^{n-2}\Gruone_{j} \vartheta^{k_1}_{\ph}(x-jh)+ \vartheta^{k_1}_{\ph}(x-(n-1)h)\frac1h \sum_{j=0}^{n-2}\Gruone_{j} -\lambda \Gru_{0}\vartheta^{k_1}_{\ph}(x)\\
&=-\frac{\barlambda }{h^2}\sum_{j=0}^{n-2}\Gruone_{j}   \calG^{k_1}_{n-1-j}+\frac1{h^2}\calG^{k_1}_{0} \sum_{j=0}^{n-2}\Gruone_{j} -\lambda \Gru_{0}\vartheta^{k_1}_{\ph}(x)\\
&=-\frac{\barlambda }{h^2}\sum_{j=0}^{n-1}\Gruone_{j}  \calG^{k_1}_{n-1-j}+\frac{\barlambda }{h^2}\Gruone_{n-1}   \calG^{k_1}_{0}+\frac1{h^2}\calG^{k_1}_{0}   \sum_{j=0}^{n-2}\Gruone_{j} -\lambda \Gru_{0}\vartheta^{k_1}_{\ph}(x)\\
&=-\barlambda n+o(1) -\lambda \Gru_{0}\vartheta^{k_1}_{\ph}(x)\\
&=-2\frac{\barlambda }h +\barlambda+o(1) -\lambda \Gru_{0}\vartheta^{k_1}_{\ph}(x),
\end{align*}
and using Lemma \ref{lem:firtorderapprox}
\begin{align*}
G^{\mathrm{NN}}_{\ph} \Il g(x)
&=-\barlambda \frac1h \Conelhzero \Il g(x) -\lambda \Gru_{0}\Il g(x)\\
&=-\barlambda \frac1h\left[I_+g(x)+hF_+[g](x)+O(h^2)\right] -\lambda \Gru_{0}\Il g(x)\\
&=-\barlambda \frac{I_+g(x)}h +O(1) -\lambda \Gru_{0}\Il g(x).
\end{align*}
 Combining the above with Corollary \ref{cor:L94}, as $h\to 0$
\begin{align*}
 \int_{-1}^1|G^{\mathrm{NN}}_{\ph}f_h(x)-g(x)+I_+g(1)/2|\,\dd x&\le \|\Clh  \Il  g-g\|_{L^1[-1,1]}+   \int_{1-2h}^{1} |O(1)|\,\dd x \to 0.
 \end{align*}
 
\end{proof}


\begin{proof}[of Theorem \ref{thm:conv_L_N*D}] Recall the matrix \eqref{interpolationmatrixGBC_N*D}, definition \eqref{G+hBC} and that $h(n+1)=2$. By Lemmata \ref{lem:varthek0} and \ref{Ghdissipative},  $f_h,G^{\mathrm{N^*D}}_{\ph}f_h\in L^1[-1,1]$, and $f_h\to f$ as $h\to0$ in $L^1[-1,1]$. It remains to prove the convergence of $G^{\mathrm{N^*D}}_{\ph}f_h$. We first observe  that for all small $h$, using the canonical extension of $g$ and recalling that $\Il g=0$ close to $-1$,
\begin{equation}
G^{\mathrm{N^*D}}_{\ph} \Il g(x)=\left\{\begin{split}&\Clh \Il g(x), & 1\le\iota(x)\le n-1,\\
&\Clh \Il g(x)+(D^r(\lambda)-1)\Gru_0\Il g(x+h), & \iota(x)=n,\\
&\Clh \Il g(x)-\Gru_0\Il g(x+h), & \iota(x)=n+1.
\end{split}\right.
\end{equation}
By Lemma \ref{lem:NRk-1} for $1\le\iota(x)\le n-1$
\[
G^{\mathrm{N^*D}}_{\ph} \vartheta^{k_{-1}}_{\ph}(x)=0,
\]
 for $\iota(x)=n$ we directly compute, using \eqref{eq:partialphithetah-1+},
\begin{align*}
G^{\mathrm{N^*D}}_{\ph} \vartheta^{k_{-1}}_{\ph}(x)&= \Clh  \vartheta^{k_{-1}}_{\ph}(x) +(D^r(\lambda)-1)\Gru_0 \vartheta^{k_{-1}}_{\ph}(x+h)\\
&= (D^r(\lambda)-1)\Gru_0 \vartheta^{k_{-1}}_{\ph}(x+h),
\end{align*}
and for $\iota(x)=n+1$, using \eqref{eq:partialphithetah-1+}  for the canonical extension of $\vartheta^{k_{-1}}_{\ph}$ (see Remark \ref{rmk:canon_extra}),
\begin{align*}
G^{\mathrm{N^*D}}_{\ph} \vartheta^{k_{-1}}_{\ph}(x)&= \Clh  \vartheta^{k_{-1}}_{\ph}(x) -\Gru_0 \vartheta^{k_{-1}}_{\ph}(x+h)-\Gru_{n+1}\vartheta^{k_{-1}}_{\ph}(x-(\iota(x)-1)h)\\
&=-\Gru_0 \vartheta^{k_{-1}}_{\ph}(x+h)-\theta\frac{\Gru_{n+1}}{h}\calG^{k_{-1}}_{0}\\
&=-\Gru_0 \vartheta^{k_{-1}}_{\ph}(x+h)+o(1/h),
\end{align*}
because $\Gru_{n+1}\calG^{k_{-1}}_{0}/h=\Gru_{n+1}/(h^2\sym(1/h))=o(1/h)$ by   \eqref{eq:convhomega} and   \eqref{eq:GClimphi}.
Observe that for $\iota(x)=n$, using \eqref{eq:phi+1tophi} and \eqref{eq:GClimk-2k0},
\begin{align*}
 \Il g(x+h)-\frac{\Il g(1)}{\vartheta^{k_{-1}}_{\ph}(1)}\vartheta^{k_{-1}}_{\ph}(x+h) &=  O(h)+\Il g(1)\left(1-\frac{(1-\theta)\calG^{k_{-1}}_{n-1}+\theta \calG^{k_{-1}}_n}{\calG^{k_{-1}}_n}\right)\\
 &=  O(h)+\Il g(1)(1-\theta)\left(\frac{\calG^{k_{-1}}_n-\calG^{k_{-1}}_{n-1}}{\calG^{k_{-1}}_n}\right)\\
 &=  O(h)+\Il g(1)(1-\theta)h\left(\frac{\calG^{k_{-2}}_n}{\calG^{k_{-1}}_n}\right) =O(h),
\end{align*}
 and similarly, for $\iota(x)=n+1$,
\[
\Il g(x+h)-\frac{\Il g(1)}{\vartheta^{k_{-1}}_{\ph}(1)}\vartheta^{k_{-1}}_{\ph}(x+h)=O(h).
\]
Then by Corollary \ref{cor:L94} and \eqref{eq:convhomega}, as $h\to 0$,
\begin{align*} 
 \int_{-1}^1|G^{\mathrm{N^*D}}_{\ph}f_h(x)-g(x)|\,\dd x&\le \|\Clh  \Il  g-g\|_{L^1[-1,1]}+\Gru_0  O(h^2)+o(1/h)h\to 0.
 \end{align*}

\end{proof}

\begin{remark}\label{rem:thm12}
In the proof of Theorem \ref{thm:conv_L_N*D} we proved the convergence  $ \|G^{\mathrm{N^*D}}_{\ph}f_h-g\|_{L^1[-1,1]}=O(h^{2}\sym(1/h))$ if $\phi$ allows a density with right and left limits at 2. This is because we proved the   order of convergence $h^{2}\sym(1/h)+\Gru_n /(h \sym(1/h))$, but      $\Gru_n /h\to (\phi(2+)+\phi(2-))/2 $ \cite[Corollary  VII.3c.1]{MR0005923}, and so $\Gru_n /(h \sym(1/h))\sim  1/\sym(1/h)=o(h^{2}\sym(1/h))$.
\end{remark}

\begin{proof}[of Theorem \ref{thm:conv_L_N*N}] Recall the matrix \eqref{interpolationmatrixGBC_N*N}, definition \eqref{G+hBC} and that $h(n+1)=2$. By Lemma \ref{lem:varthek0} it follows that $\|e_h\|_{L^1[-1,1]}\to 0$ as $h\to 0$. Then, by Lemmata \ref{lem:varthek0} and \ref{Ghdissipative}, $f_h,G^{\mathrm{N^*N}}_{\ph}f_h\in L^1[-1,1]$ and $f_h\to f$ as $h\to0$ in $L^1[-1,1]$. It remains to prove the convergence of $G^{\mathrm{N^*N}}_{\ph}f_h$.   First observe that
\begin{equation} \label{eq:NN+1}
G^{\mathrm{N^*N}}_{\ph} e_h(x)=\left\{\begin{split}&0, & 1\le\iota(x)\le n-1,\\
&-\Gru_0\lambda \left(\Il g(x+h)+b\vartheta_{\ph}^{k_1}(x+h)+c_{-1}\vartheta_{\ph}^{k_{-1}}(x+h)\right), & \iota(x)=n,\\
&\Gru_0\lambda \left(\Il g(x)+b\vartheta_{\ph}^{k_1}(x)+c_{-1}\vartheta_{\ph}^{k_{-1}}(x)\right), & \iota(x)=n+1,
\end{split}\right.
\end{equation}
and for all small $h$
\begin{equation}\label{eq:NN+2}
G^{\mathrm{N^*N}}_{\ph} \Il g(x)=\left\{\begin{split}&\Clh \Il g(x), & 1\le\iota(x)\le n-1,\\
&\lambda\Gru_0\Il g(x+h)-\frac{\lambda}{h} \Conelhzero \Il g(x)+O(1), & \iota(x)=n,\\
&-\lambda\Gru_0\Il g(x)-\frac{\barlambda }{h} \Conelhzero \Il g(x) , & \iota(x)=n+1,
\end{split}\right.
\end{equation}
where we used $\Clh \Il g =O(1)$, by Corollary \ref{cor:L94}, and that $\Il g=0$ close to $-1$.
Also 
\begin{equation}\label{eq:NN+3}
G^{\mathrm{N^*N}}_{\ph} \vartheta^{k_{-1}}_{\ph}(x)=\left\{\begin{split}&0, & 1\le\iota(x)\le n-1,\\
&\lambda\Gru_0\vartheta^{k_{-1}}_{\ph}(x+h)+o(1/h), & \iota(x)=n,\\
&-\lambda\Gru_0\vartheta^{k_{-1}}_{\ph}(x)+o(1/h), & \iota(x)=n+1,
\end{split}\right.
\end{equation}
where we used, for $1\le\iota(x)\le n-1$, Lemma \ref{lem:NRk-1},   for $\iota(x)=n$, by \eqref{eq:partialphi-1thetah-1+},  \eqref{eq:partialphithetah-1+}, \eqref{eq:phi+1tophi}, \eqref{eq:GClimphi} and \eqref{eq:convhomega},
\begin{align*}
G^{\mathrm{N^*N}}_{\ph}  \vartheta^{k_{-1}}_{\ph}(x)&=\barlambda \Clh \vartheta^{k_{-1}}_{\ph}(x) -\barlambda \left(\Gru_n+\frac1h\Gruone_{n-1}\right)\vartheta^{k_{-1}}_{\ph}(x-(\iota(x)-1)h)\\
&\quad-\frac{\lambda}h \Conelhzero \vartheta^{k_{-1}}_{\ph}(x)  +\lambda \Gru_0\vartheta^{k_{-1}}_{\ph}(x+h) \\
&=-\barlambda \left(\Gru_n+\frac1h\Gruone_{n-1}\right)\frac{\theta\calG^{k_{-1}}_0}h+\lambda\Gru_0\vartheta^{k_{-1}}_{\ph}(x+h)\\
&=-\barlambda \theta\frac{\calG^{k_{-1}}_0}h\sum_{j=0}^n\Gru_{j}+\lambda\Gru_0\vartheta^{k_{-1}}_{\ph}(x+h)\\
&=o(1/h)+\lambda\Gru_0\vartheta^{k_{-1}}_{\ph}(x+h),
\end{align*}
and similarly for $\iota(x)=n+1$ 
\begin{align*}
G^{\mathrm{N^*N}}_{\ph}  \vartheta^{k_{-1}}_{\ph}(x)&=-\frac{\barlambda }h \Conelhzero \vartheta^{k_{-1}}_{\ph}(x) +\frac{\barlambda }h\Gruone_{n}\frac{\theta\calG^{k_{-1}}_0}h-\lambda\Gru_0\vartheta^{k_{-1}}_{\ph}(x)\\
&=o(1/h)-\lambda\Gru_0\vartheta^{k_{-1}}_{\ph}(x).
\end{align*}
We now show that
\begin{equation}\label{eq:NN+4}
G^{\mathrm{N^*N}}_{\ph}b \vartheta^{k_{1}}_{\ph}(x)=\left\{\begin{split}&o(1/h), & 1\le\iota(x)\le 2,\\
&b- b\barlambda \Gru_{\iota(x)} \frac{\calG^{k_1}_0}{h}, & 3\le\iota(x)\le n-1,\\
&I_+g(1)\frac{\lambda}h+b\lambda \Gru_0\vartheta^{k_{1}}_{\ph}(x+h)+O(1), & \iota(x)=n,\\
&I_+g(1)\frac{\barlambda }h-b\lambda\Gru_0\vartheta^{k_{1}}_{\ph}(x)+O(1), & \iota(x)=n+1.
\end{split}\right.
\end{equation}
Indeed for $\iota(x)= 1$
\begin{align*}
G^{\mathrm{N^*N}}_{\ph}   \vartheta^{k_1}_{\ph}(x)&=-\barlambda (\Gru_1+\Gru_0)\frac{\calG^{k_1}_0}{h}+\lambda \Gru_0  \frac{\calG^{k_1}_0}{h} =\barlambda \left(\frac{\sym'(1/h)}{h\sym(1/h)}-1\right) +\lambda =O(1),
\end{align*}
using \eqref{eq:2mon} (which holds under \ref{H0}), and similarly for $\iota(x)= 2$ 
\begin{align*}
G^{\mathrm{N^*N}}_{\ph}  \vartheta^{k_1}_{\ph}(x)&=-\barlambda \Gru_2\frac{\calG^{k_1}_0}{h}+\left( \Gru_1+\barlambda \Gru_0 \right)\frac{\calG^{k_1}_0}{h}+\Gru_0\frac{\calG^{k_1}_1}{h} \\
&=-\frac{\barlambda }{2} \frac{\sym''(1/h)}{h^2\sym(1/h)}+O(1)+1+\frac{\sym'(1/h)}{h\sym(1/h)}\\
&=o(1/h)+O(1),
\end{align*}
where we used in the third identity \eqref{eq:convhomega}  in combination with 
\begin{equation}
\label{eq:omega''}
\sym''(1/h)=\int_0^\infty e^{-\frac1h y}y^2 \phi(\dd y)\to 0,\quad \text{as } h\to 0.
\end{equation}
For $3\le\iota(x)\le n-1$, using \eqref{eq:partialphithetah1+} and \eqref{eq:convhomega},
\begin{align*}
G^{\mathrm{N^*N}}_{\ph}  \vartheta^{k_1}_{\ph}(x)&=\Clh \vartheta^{k_{1}}_{\ph}(x)=1- \barlambda \Gru_{\iota(x)} \frac{\calG^{k_1}_0}{h}=1- \barlambda \Gru_{\iota(x)} o(h).
\end{align*}
For $\iota(x)=n$, using \eqref{eq:partialphithetah1+}, \eqref{eq:partialphi-1thetah1+}, \eqref{eq:phi+1tophi}, \eqref{eq:GClimphi} and \eqref{eq:convhomega},
\begin{align*}
G^{\mathrm{N^*N}}_{\ph}  \vartheta^{k_1}_{\ph}(x)&=\barlambda \Clh \vartheta^{k_{1}}_{\ph}(x) -\barlambda \left( \Gru_n +\frac1h\Gruone_{n-1}\right)    \vartheta^{k_{1}}_{\ph}(x-(\iota(x)-1)h)\\
&\quad-\frac{\lambda}h \Conelhzero \vartheta^{k_{1}}_{\ph}(x) +\lambda \Gru_0\vartheta^{k_{1}}_{\ph}(x+h) \\
&=\barlambda \left(1-\frac{\barlambda }h\Gru_{n} \calG^{k_1}_{0}\right) +\barlambda \left(\Gru_n+\frac1h\Gruone_{n-1}\right)\frac{\barlambda }h \calG^{k_1}_{0}\\
&\quad-\frac{\lambda}h\left[(n-1)h-\frac{\barlambda }h\Gruone_{n-1} \calG^{k_1}_{0}\right]  +\lambda \Gru_0\vartheta^{k_{1}}_{\ph}(x+h) \\
&=-2\frac{\lambda}h+O(1)+\lambda \Gru_0\vartheta^{k_{1}}_{\ph}(x+h).
\end{align*}
For $\iota(x)=n+1$, similarly as above\footnote{Note that as we are working on $L^1[-1,1]$ the identity at $x=1$  in \eqref{eq:partialphi-1thetah1+} is not needed.} 
\begin{align*}
G^{\mathrm{N^*N}}_{\ph}  \vartheta^{k_1}_{\ph}(x)
&=-\frac{\barlambda }h(nh)-\lambda\Gru_0\vartheta^{k_{1}}_{\ph}(x)=-2\frac{\barlambda }h+O(1)-\lambda\Gru_0\vartheta^{k_{1}}_{\ph}(x).
\end{align*}
 Now, observing that  Lemma \ref{lem:firtorderapprox}   yields  $\Conelhzero \Il g(x)-I_+g(1)=O(h)$ on $[1-2h,1]$,  combining \eqref{eq:NN+1}, \eqref{eq:NN+2}, \eqref{eq:NN+3} and \eqref{eq:NN+4} we obtain
\[
G^{\mathrm{N^*N}}_{\ph}f_h(x)= o(1/h),  \quad n\le\iota(x)\le n+1.
\]
So we conclude by applying Corollary \ref{cor:L94} with \eqref{eq:NN+1}, \eqref{eq:NN+2}, \eqref{eq:NN+3} and \eqref{eq:NN+4} to obtain
\begin{align}\nonumber 
 \left\|G^{\mathrm{N^*N}}_{\ph}f_h-g-b\right\|_{L^1[-1,1]} &\le \|\Clh  \Il  g-g\|_{L^1[-1,1]}+\left|b\right|\left(\int_{-1}^{2h-1}+ \int^{1}_{1-2h}\dd x\right) \\ \nonumber
 &\quad +o(h) \sum_{j=3}^{n-1} \Gru_j\int_{(j-1)h-1}^{jh-1}\,\dd x
 \\ \nonumber
 &\quad + o(1/h)\left( \int_{1-2h}^{1}+\int_{-1}^{-1+2h}\dd x\right)
 \\ \nonumber
 &=   o(h^2) \left(\sum_{j=0}^{n-1} \Gru_j-\sum_{j=0}^{2} \Gru_j\right)
 + o(1) \\ \label{eq:rmkonrateNN}
  &= o(h^2) \left(\Gru_0 -\Gru_1+\Gru_2\right)+o(1)=o(1)
,
 \end{align}
 using \eqref{eq:GClimphi}, \eqref{eq:convhomega} and \eqref{eq:omega''} in the last identity.
 
\end{proof}

\begin{remark}
In the proof of Theorem \ref{thm:conv_L_N*N} we proved that $ \|G^{\mathrm{N^*N}}_{\ph}f_h-g-b\|_{L^1[-1,1]}= O(1/[h\sym(1/h)])$. This is because the   $o(1)$ term coming from \eqref{eq:NN+3} is $O(1/[h\sym(1/h)])$, which converges at a slower rate than the $o(1)$ term due to  \eqref{eq:NN+4} (the rest of the terms are $O(h)$). Indeed, a careful examination of the proof reveals that for $\iota\le 2$ (in   \eqref{eq:NN+4}) the $o(1)$ term has been computed to be  $O(\sym''(1/h)/[h\sym(1/h)])$, and for   $3\le \iota\le n-1$ (in   \eqref{eq:NN+4}), corresponding to the first term in \eqref{eq:rmkonrateNN}, we have \[
\frac{h}{\sym(1/h)} \left(\sym(1/h)+\frac{\sym'(1/h)}{h}+\frac{\sym''(1/h)}{2h^2}\right)
=O(h)+O\left(\frac{\sym''(1/h)}{h\sym(1/h)}\right).
\]
\end{remark}


\end{document}